\documentclass[fontsize=11pt,toc=bib]{scrartcl}

\setkomafont{title}{\sffamily}
\setkomafont{pagehead}{\sffamily}
\setkomafont{pagenumber}{\sffamily}
\setkomafont{author}{\sffamily}
\addtokomafont{author}{\vspace*{2em}\large}

\usepackage[automark]{scrlayer-scrpage}
\ihead{\headmark}
\ohead{\pagemark}
\usepackage[
  paper=a4paper,
  left=80pt,
  right=80pt,
  top=86pt,
  bottom=56pt,
  bindingoffset=0mm,
  includefoot,
  includehead,
  foot=16.5pt,
  head=16.5pt
]{geometry}

\DeclareMathAlphabet{\pazocal}{OMS}{zplm}{m}{n}
\DeclareMathAlphabet{\mymathbb}{U}{bbold}{m}{n}
\newcommand{\One}{\mymathbb{1}}
\newcommand{\Zero}{\mymathbb{0}}

\usepackage{indentfirst}
\usepackage{lipsum}
\usepackage[english]{babel}
\usepackage{amsmath} 
\usepackage{amsthm}
\usepackage[sfdefault]{FiraSans}
\usepackage[charter,expert]{mathdesign}
\usepackage[scaled=.96]{XCharter}
\usepackage{setspace}
\usepackage{nicefrac}
\usepackage[dvipsnames]{xcolor}
    \definecolor{linkcolor}{HTML}{0D6A9E}
    \definecolor{3blue}{HTML}{0072B2}
    \definecolor{3green}{HTML}{009E73}
    \definecolor{3ochre}{HTML}{E69F00}
    \definecolor{3yellow}{HTML}{F0E442}
    \definecolor{3cyan}{HTML}{56B4E9}
    \definecolor{3red}{HTML}{D55E00}
    \definecolor{3pink}{HTML}{CC79A7}
    \definecolor{2blue}{HTML}{1A85FF}
    \definecolor{2red}{HTML}{D41159}
\usepackage{graphicx}
    \graphicspath{ {./gfx/} }
\usepackage{caption}
\usepackage{tensor}
\usepackage{tikz-cd}
\tikzcdset{arrow style=tikz}
\usepackage[hang,flushmargin]{footmisc} 
\usepackage{hyperref}
    \hypersetup{
    colorlinks=true,
    linktoc=all,
    linkcolor=linkcolor,
    citecolor=linkcolor,
    filecolor=magenta,      
    urlcolor=linkcolor,
    pdftitle={Generalized double affine Hecke algebras, their representations and higher Teichm\"uller theory},
    pdfpagemode=FullScreen,
    }
\usepackage{cleveref}
    \newtheoremstyle{komait}
      {\topsep}   
      {\topsep}   
      {\itshape}  
      {0pt}       
      {\bfseries\sffamily} 
      {.}         
      {5pt plus 1pt minus 1pt} 
      {}          
    \newtheoremstyle{komanormal}
      {\topsep}   
      {\topsep}   
      {\rmfamily}  
      {0pt}       
      {\bfseries\sffamily} 
      {.}         
      {5pt plus 1pt minus 1pt} 
      {}          
    \theoremstyle{komait}
    \newtheorem{theorem}{Theorem}[section]
    
    \newtheorem{lemma}[theorem]{Lemma}
    \newtheorem{definition}[theorem]{Definition}
    \newtheorem{proposition}[theorem]{Proposition}
    \theoremstyle{komanormal}
    \newtheorem{example}[theorem]{Example}
    
    \newtheorem{remark}[theorem]{Remark}
    
\newcommand{\diag}{\mathrm{diag}}
\newcommand{\MC}{\mathcal{M}_q}

\newcommand{\C}{\mathcal{C}}
\newcommand{\Rep}{\mathsf{Rep}}
\newcommand{\Repfour}{\mathsf{Rep}\big(H_{D_4}(t,q)\big)}

\newcommand{\Mat}{\mathsf{Mat}}
\newcommand{\GL}{\mathsf{GL}}
\newcommand{\SL}{\mathsf{SL}}
\newcommand{\PGL}{\mathsf{PGL}}
\newcommand{\PSL}{\mathsf{PSL}}
\newcommand{\integer}{\mathbb{Z}}

\newcommand{\real}{\mathbb{R}}
\newcommand{\complex}{\mathbb{C}}
\newcommand{\ord}[1]{\tensor*[^\bullet_\bullet]{{#1}}{^\bullet_\bullet}} 

\begin{document}

\title{\usekomafont{subtitle}\LARGE\vspace{-3.5em}Generalized double affine Hecke algebras,\\ their representations,\\and higher Teichm\"uller theory\vspace{-.5em}}
\author{\raisebox{-.5ex}{\href{https://orcid.org/0000-0002-4975-8774}{\includegraphics[height=15pt]{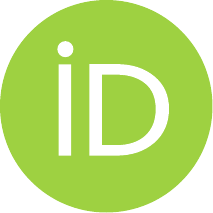}}}\hspace{.5em}Davide Dal Martello{\rmfamily\footnote{\hspace{.4em}University of Birmingham, contact@davidedalmartello.com}} \and \raisebox{-.5ex}{\href{https://orcid.org/0000-0001-9917-2547}{\includegraphics[height=15pt]{gfx/orcid.pdf}}}\hspace{.5em}Marta Mazzocco{\rmfamily\footnote{\hspace{.4em}University of Birmingham, m.mazzocco@bham.ac.uk}}}
\date{\vspace{-2.5em}}

\maketitle

\begin{abstract}\noindent
Generalized double affine Hecke algebras (GDAHA) are flat deformations of the group algebras of $2$-dimensional crystallographic groups associated to star-shaped simply laced affine Dynkin diagrams. In this paper, we first construct a functor that sends representations of the $\tilde D_4$-type GDAHA to representations of the $\tilde E_6$-type one for specialised parameters. Then, under no restrictions on the parameters, we construct embeddings of both GDAHAs of type $\tilde D_4$ and $\tilde E_6$ into matrix algebras over quantum cluster $\pazocal{X}$-varieties, thus linking to the theory of higher Teichm\"uller spaces. For $\tilde E_6$, the two explicit representations we provide over distinct quantum tori are shown to be related by quiver reductions and mutations.
\end{abstract}

\noindent
\begin{center}{\vspace{-.5em}\small{\textbf{\textsf{Keywords}}\hspace{.5em}Generalized double affine Hecke algebras, higher Teichm\"uller theory, category theory.}}
\end{center}

\renewcommand{\contentsname}{\textcolor{linkcolor}{Contents}}

\tableofcontents

\section{Introduction}

Double affine Hecke algebras are a class of associative unital algebras linked to root systems, depending on several parameters. They play a fundamental role in the proof of Macdonald conjectures \cite{Cherednik2005,Noumi2004, Sahi1999,Stokman2000} and are deeply related to integrable systems of Calogero-Moser type \cite{Etingof2002,Oblomkov2004} as well as to the Painlev\'e differential equations \cite{Oblomkov2004,Mazzocco2016}.

For any simply laced Dynkin diagram  $\mathcal D$, with star-shaped affinization $\tilde{\mathcal D}$, the generalized double affine Hecke 
algebra (GDAHA)  associated to $\mathcal{D}$
 was introduced by Etingof, Oblomkov and Rains in \cite{Etingof2007} as a flat deformation of $\complex[G_l]$, the group algebra of the $2$-dimensional crystallographic group 
 $G_l:=\mathbb Z_l \ltimes\mathbb Z^2$, $l=2,3,4,6$.
 
 In this paper, we
 focus on $l=2$ and $l=3$, which correspond respectively to $\Tilde{D}_4$ and $\Tilde{E}_6$.
\begin{figure}[!htb]
    \centering
    \includegraphics[width=130mm]{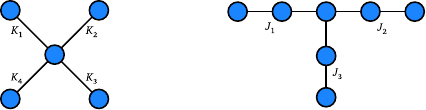}
    \caption{Affine Dynkin diagrams: $\Tilde{D}_4$ on the left and $\Tilde{E}_6$ on the right. Each leg contributes a generator and its length, defined by node-counting, determines the order of the corresponding Hecke relation. These two cases are special in that all legs have equal length.}\label{fig:affine}
\end{figure} 

For $\Tilde{D}_4$, the associated GDAHA we denote $H_{D_4}(t,q)$ recovers the $CC_1^{\vee}$-type Cherednik algebra.
Namely, $H_{D_4}(t,q)$  is the family of algebras over $\complex$ with parameters $t_1,t_2,t_3,t_4,q\in\mathbb C^*$, $q^m\neq 1$ for any $m\in\mathbb Z_{>0}$, generated by $K_1,K_2,K_3,K_4$ subject to the relations
\begin{equation*}
    \left(K_i-t_i\right)\left(K_i-\frac{1}{t_i}\right)=0,\quad i=1,2,3,4;\qquad K_1 K_2 K_3 K_4 = q^{\nicefrac{-1}{2}}.
\end{equation*}

For $\Tilde{E}_6$, the GDAHA $H_{E_6}(t,q)$ is a true generalisation of Cherednik's DAHA.
It is defined as the family of algebras over $\complex$ 
depending on $q,t_i^{(j)}\in\mathbb C^*$, $i=1,2,3$ and $j=1,2$, $q^m\neq 1$ for any $m\in\mathbb Z_{>0}$, by generators $J_1,J_2,J_3$ and relations
\begin{equation*}
    \left(J_i-t_i^{(1)}\right)\left(J_i-t_i^{(2)}\right)\left(J_i-\frac{1}{t_i^{(1)}t_i^{(2)}}\right)=0,\quad i=1,2,3;\qquad J_1 J_2 J_3 = q^{\nicefrac{-1}{3}}.
\end{equation*}

Denote by $\mathbb{T}$ the algebraic torus formed by the tuple $t$,  namely $\mathbb T= (\complex^*)^4$ in the case of $H_{D_4}(t,q)$ and $\mathbb T= (\complex^*)^6$ in the case of $H_{E_6}(t,q)$. Then, 
each family of GDAHAs can be obtained by specializations
of a universal algebra $\mathsf{H}_\mathcal{D}$ over $\complex[\mathbb{T}]\otimes\complex[q^{\nicefrac{\pm1}{l}}]$, in which $q$ and all elements of $t$ are
central.

Regarding the GDAHA representation theory,  the case of $H_{D_4}$ is well-understood \cite{Koornwinder2011,Macdonald2003,Oblomkov2009}. For the other three cases $\Tilde{E}_{6,7,8}$, only existence results or representations that restrict to special parameters are available.
In particular, \cite{Jordan14} gives a categorical description of spherical subalgebras in terms of a quotient on the category of equivariant $\pazocal{D}_q$-modules while \cite{Etingof2007} suggests that methods in multiplicative preprojective algebras could be adapted to classify the finite dimensional representation theory.
Etingof, Gan and Oblomkov also proved that the monodromy of the Knizhnik-Zamolodchikov connection defines a functor
between the categories of finite dimensional representations of the rational GDAHA and its corresponding non-degenerate one \cite{Etingof2006}.
In principle, this functor allows to construct a large supply of finite dimensional representations of a GDAHA starting from the finite dimensional representations of its rational degeneration.
Using this very technique, some finite dimensional representations of  $H_{E_6}(t,q)$ for special values of the parameters were produced in terms of the $R$-matrices of $U_q(\mathfrak{sl}_N)$ \cite{Fu2016}, but fail to address the universal algebra.

Our first result constructs
explicit representations of $H_{E_6}(\tilde t,q)$, a specialization of the GDAHA given by selected parameters $\tilde t$, from the
 representation theory of $H_{D_4}(t,q)$: 
\begin{theorem}\label{thm:functor}
Let $\mathsf{Rep}\big(H_{\mathcal{D}}(t,q)\big)$  be the category of representations of the algebra $H_{\mathcal{D}}(t,q)$.
There exists a functor 
$$
{\mathcal F}_q\ : \ \Repfour\ {\rightarrow}\ \Rep\big(H_{E_6}(\tilde{t},q)\big),
$$
where the specialized parameters $\tilde{t}$ are given by
\begin{equation}\label{eq:paramE6}
\begin{aligned}
\tilde{t}^{(1)}_1=\tilde{t}^{(2)}_1=1,\qquad
&\tilde{t}^{(1)}_2=t_1^{\nicefrac{-4}{3}}t_2^{\nicefrac{2}{3}}t_3^{\nicefrac{2}{3}},\quad
\tilde{t}^{(2)}_2=t_1^{\nicefrac{2}{3}}t_2^{\nicefrac{-4}{3}}t_3^{\nicefrac{2}{3}},\\
&\tilde{t}^{(1)}_3=\frac{q^{-\nicefrac{1}{3}}}{t_1^{\nicefrac{2}{3}}t_2^{\nicefrac{2}{3}}t_3^{\nicefrac{2}{3}}},\quad\,\,
\tilde{t}^{(2)}_3=q^{\nicefrac{1}{6}}t_1^{\nicefrac{1}{3}}t_2^{\nicefrac{1}{3}}t_3^{\nicefrac{1}{3}}t_4.
\end{aligned}
\end{equation}
\end{theorem}
In proving this theorem, we construct the analogue of the basic representation for this specialized GDAHA $H_{E_6}(\tilde{t},q)$, see formulae (\ref{eq:Lb-rep}-\ref{eq:Ub-rep}-\ref{eq:Pib-rep}).

The centers of $H_{D_4}(t,1)$ and $H_{E_6}(t,1)$ are both affine del Pezzo surfaces, obtained by removing a triangle and a nodal $\mathbb P^1$ respectively from a degree $3$ projective del Pezzo:
\begin{equation*}
    \begin{aligned}
    \tilde D_4:\, & \hspace{.5em} x_1 x_2 x_3 + x_1^2+x_2^2 + x_3^2+ a_1 x_1 + a_2 x_2 + a_3 x_1 + a_4=0,\\
       \tilde E_6:\, & \hspace{.5em} x_1 x_2 x_3 + x_1^3+x_2^3 + x_3^2+ a_1 x_1^2 + a_2 x_2^2 + a_3 x_1 + a_4 x_2 + a_5 x_3 + a_6=0,\\
    \end{aligned}
\end{equation*}
for the parameters $a_i\in\mathbb C$  \cite{Etingof2007}.
The quantizations of these surfaces are special cases of the generalized Sklyanin--Painlev\'e algebra introduced in \cite{Chekhov2021}. Based on the fact that the $ \tilde D_4$
del Pezzo can be obtained as a limit of the $\tilde E_6$ one, 
Chekhov, Rubtsov and the second author of the present paper conjectured that 
$H_{D_4}(t,q)$ could be obtained as a limit of $H_{E_6}(t,q)$. \Cref{thm:functor} proves this conjecture.
Note that in \Cref{thm:functor}, two of the parameters   $\Tilde{t}$ are fixed, leaving the remaining ones free. 
To tackle the representation theory of the universal $\mathsf{H}_{E_6}$, we bring the machinery of higher Teichm\"uller spaces and quantum cluster varieties into the GDAHA theory.

In \cite{Mazzocco2018}, the second author constructed an embedding 
of $H_{D_4}(t,q)$ into $\Mat_2(\mathbb{T}^2_q)$, i.e., $2\times2$ matrices with entries in the quantum 2-torus.
Our second result is 
an explicit embedding of $\mathsf{H}_{E_6}$ into $3\times 3$ matrices over a quantum cluster $\pazocal{X}$-variety.
Before enunciating this theorem, let us briefly introduce some notation. 
For any quiver $\pazocal{Q}$ having no loops nor $2$-cycles, we denote the corresponding quantum $\pazocal{X}$-torus by $\pazocal{X}_\pazocal{Q}$ (\Cref{sec:quantization}). In this paper, we mainly deal with the quivers  $\pazocal{Q}_1,\pazocal{Q}_2,\pazocal{Q}_3$ depicted in \Cref{fig:quivers}.
\begin{figure}[!htb]
    \centering
    \includegraphics[width=\textwidth]{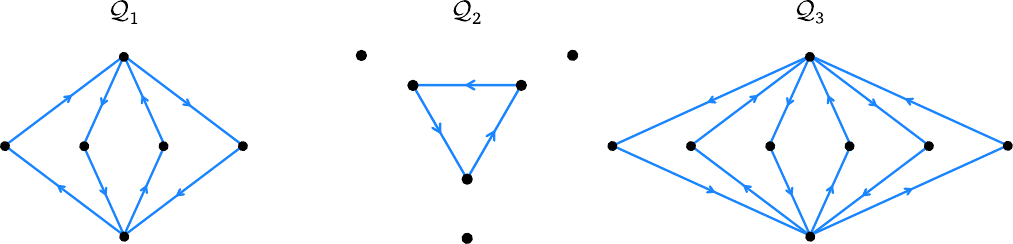}
    \caption{The main quivers starring in this paper. $\pazocal{Q}_{2,3}$ emerge from the rank $n=2,3$ higher Teichm\"uller theory, respectively. $\pazocal{Q}_1$ is a full subquiver of $\pazocal{Q}_3$, obtained by evaluating at $1$ specific central elements in $\pazocal{X}_{\pazocal{Q}_3}$.}\label{fig:quivers}
\end{figure} 
\begin{theorem}\label{thm:E6-emb}
Let $\pazocal{X}_{\pazocal{Q}_3}$ be the  quantum $\pazocal{X}$-torus generated by  $Z_{C1}$, $Z_{C2}$, $Z_{C3}$, $Z_{Y1}$, $Z_{Y2}$, $Z_{Y3}$, $Z^{(t)}_{111}$, $Z^{(b)}_{111}$ with relations encoded by the diamond-shaped quiver $\pazocal{Q}_3$ of \Cref{fig:3amalquiver}.
The matrices $\overline C, \overline Y, \overline R\in\SL_3(\pazocal{X}_{\pazocal{Q}_3}^{\nicefrac{1}{3}})$ given by equations \eqref{RCY}, where $\pazocal{X}_{\pazocal{Q}_3}^{\nicefrac{1}{3}}$ is the extension of $\pazocal{X}_{\pazocal{Q}_3}$ containing cubic roots of the variables $\{Z_{C1},\dots,Z^{(b)}_{111}\}$,
satisfy the Hecke relations
\begin{align*}
    \left(\overline{C}-Z_{C1}^{\nicefrac{1}{3}}Z_{C2}^{\nicefrac{2}{3}}Z_{Y3}^{\nicefrac{2}{3}}Z_{C3}^{\nicefrac{1}{3}}Z_{111}^{(t)\nicefrac{2}{3}}Z_{111}^{(b)\nicefrac{2}{3}}\One\right)&\left(\overline{C}-Z_{C1}^{\nicefrac{1}{3}}Z_{C2}^{\nicefrac{-1}{3}}Z_{Y3}^{\nicefrac{-1}{3}}Z_{C3}^{\nicefrac{1}{3}}Z_{111}^{(t)\nicefrac{-1}{3}}Z_{111}^{(b)\nicefrac{-1}{3}}\One\right)\\
    &\qquad \left(\overline{C}-Z_{C1}^{\nicefrac{-2}{3}}Z_{C2}^{\nicefrac{-1}{3}}Z_{Y3}^{\nicefrac{-1}{3}}Z_{C3}^{\nicefrac{-2}{3}}Z_{111}^{(t)\nicefrac{-1}{3}}Z_{111}^{(b)\nicefrac{-1}{3}}\One\right)=\Zero,\allowdisplaybreaks\\
    \left(\overline{Y}-Z_{Y1}^{\nicefrac{1}{3}}Z_{Y2}^{\nicefrac{2}{3}}Z_{Y3}^{\nicefrac{1}{3}}Z_{C3}^{\nicefrac{2}{3}}Z_{111}^{(t)\nicefrac{2}{3}}Z_{111}^{(b)\nicefrac{2}{3}}\One\right)&\left(\overline{Y}-Z_{Y1}^{\nicefrac{1}{3}}Z_{Y2}^{\nicefrac{-1}{3}}Z_{Y3}^{\nicefrac{1}{3}}Z_{C3}^{\nicefrac{-1}{3}}Z_{111}^{(t)\nicefrac{-1}{3}}Z_{111}^{(b)\nicefrac{-1}{3}}\One\right) \\
    &\qquad \left(\overline{Y}-Z_{Y1}^{\nicefrac{-2}{3}}Z_{Y2}^{\nicefrac{-1}{3}}Z_{Y3}^{\nicefrac{-2}{3}}Z_{C3}^{\nicefrac{-1}{3}}Z_{111}^{(t)\nicefrac{-1}{3}}Z_{111}^{(b)\nicefrac{-1}{3}}\One\right)=\Zero,\allowdisplaybreaks\\
    \left(\overline{R}-Z_{Y1}^{\nicefrac{2}{3}}Z_{Y2}^{\nicefrac{1}{3}}Z_{C1}^{\nicefrac{2}{3}}Z_{C2}^{\nicefrac{1}{3}}Z_{111}^{(t)\nicefrac{2}{3}}Z_{111}^{(b)\nicefrac{2}{3}}\One\right)&\left(\overline{R}-Z_{Y1}^{\nicefrac{-1}{3}}Z_{Y2}^{\nicefrac{1}{3}}Z_{C1}^{\nicefrac{-1}{3}}Z_{C2}^{\nicefrac{1}{3}}Z_{111}^{(t)\nicefrac{-1}{3}}Z_{111}^{(b)\nicefrac{-1}{3}}\One\right)\\
     &\qquad \left(\overline{R}-Z_{Y1}^{\nicefrac{-1}{3}}Z_{Y2}^{\nicefrac{-2}{3}}Z_{C1}^{\nicefrac{-1}{3}}Z_{C2}^{\nicefrac{-2}{3}}Z_{111}^{(t)\nicefrac{-1}{3}}Z_{111}^{(b)\nicefrac{-1}{3}}\One\right)=\Zero,
\end{align*}
and the cyclic one
\begin{equation*}    \overline{C}\,\overline{Y}\,\overline{R}=q^{\nicefrac{2}{3}}\One,
\end{equation*}
where the monomials $Z_{C1}^{\nicefrac{1}{3}}Z_{C2}^{\nicefrac{2}{3}}Z_{Y3}^{\nicefrac{2}{3}}Z_{C3}^{\nicefrac{1}{3}}Z_{111}^{(t)\nicefrac{2}{3}}Z_{111}^{(b)\nicefrac{2}{3}},\dots,Z_{Y1}^{\nicefrac{-1}{3}}Z_{Y2}^{\nicefrac{-2}{3}}Z_{C1}^{\nicefrac{-1}{3}}Z_{C2}^{\nicefrac{-2}{3}}Z_{111}^{(t)\nicefrac{-1}{3}}Z_{111}^{(b)\nicefrac{-1}{3}}$ 
appearing in the Hecke relations are central elements in $\pazocal{X}_{\pazocal{Q}_3}^{\nicefrac{1}{3}}$.\\
The map $
J_1\to \overline{C},\, J_2\to  \overline{Y},\, J_3\to \overline{R},\, q\to q^{-2}$ embeds the universal GDAHA $\mathsf{H}_{E_6}$ into $\Mat_3(\pazocal{X}_{\pazocal{Q}_3}^{\nicefrac{1}{3}})$.
\end{theorem}
In the proof of this theorem, we explicitly check that \emph{all} parameters in the Hecke relations, expressed over $\pazocal{X}_{\pazocal{Q}_3}$, are free. Moreover, by picking any faithful representation of $\pazocal{X}_{\pazocal{Q}_3}$, the theorem gives a faithful representation of $\mathsf{H}_{E_6}$.
As a primer to this result, we recast in cluster terms the aforementioned $\Mat_2(\mathbb{T}^2_q)$-embedding, promoting it to the universal $\mathsf{H}_{D_4}$ by replacing $\mathbb{T}^2_q$ with a more general quantum torus $\pazocal{X}_{\pazocal{Q}_2}$ whose relations are encoded by $\pazocal{Q}_2$.    
Both $\pazocal{X}_{\pazocal{Q}_2}$ and  $\pazocal{X}_{\pazocal{Q}_3}$  appear naturally by quantizing the \textit{moduli space of pinnings}  $\mathcal{P}_{\PGL_n(\real)}(\Sigma_{g,s,m})$, introduced by Goncharov and Shen \cite{Goncharov2022}
as an extension of the moduli space $\mathcal{X}_{\PGL_n(\real)}(\Sigma_{g,s,m})$  of \emph{framed} $\PGL_n(\real)$-local systems on a genus $g$ Riemann surface $\Sigma_{g,s,m}$ with $m$ marked points on the $s$ boundaries.
Specifically, $\pazocal{X}_{\pazocal{Q}_2}$  corresponds to $\mathcal{P}_{\PGL_2}\!(\Sigma_{0,4,0})$ while $\pazocal{X}_{\pazocal{Q}_3}$ to  $\mathcal{P}_{\PGL_3}\!(\Sigma_{0,3,0})$.

Our final result pertains the choice of specific central elements $c_i$ in $\pazocal{X}_{\pazocal{Q}}$ such that the corresponding quotient $\pazocal{X}_{\pazocal{Q}}\slash(c_1-c_1^{(0)},\dots,c_l-c_l^{(0)})$  is given by $\pazocal{X}_{\pazocal{\tilde Q}}$, where the subquiver $\pazocal{\tilde Q}$ is obtained from $\pazocal{Q}$ by a new operation we name \emph{quiver seizure}. 

Let us explain this operation. We call \emph{rhombus} in $\pazocal{Q}$ a $4$-cycle with vertices labelled cyclicly by variables $Z_1,Z_2,Z_3,Z_4\in\pazocal{X}_{\pazocal{Q}}$ such that the indegree and outdegree of both $Z_2$ and $Z_4$ equal one, namely $\mathrm{deg}^+(Z_i)=\mathrm{deg}^-(Z_i)=1$ for $i=2,4$.
\begin{definition}
    The quiver seizure at vertex $Z_i$ is the map
\begin{equation*}
    \pazocal{Q} \mapsto \pazocal{Q}\backslash Z_i,
\end{equation*}
where $\pazocal{Q}\backslash Z_i\subset\pazocal{Q}$ is the full subquiver obtained by removing $Z_i$ together with its two arrows. 
\end{definition}

This operation is illustrated in \Cref{fig:seizure}.
\begin{figure}[!htb]
    \centering
    \includegraphics[width=120mm]{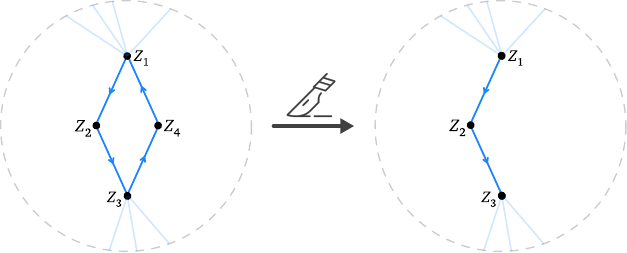}
    \caption{A quiver seizure removing $Z_4$.}\label{fig:seizure}
\end{figure} 

After showing that this operation corresponds to one such quotient in the algebra $\pazocal{X}_{\pazocal{Q}}$, we perform two seizures on the quantum $\pazocal{X}$-torus $\pazocal{X}_{\pazocal{Q}_3}$,  encoded by the diamond-shaped quiver  ${\pazocal{Q}_3}$ in \Cref{fig:quivers}, to obtain the subquiver ${\pazocal{Q}_1}$. The quotient's ideal $I$ evaluates the central elements to the unit. Denoting by $\pazocal{X}_{\pazocal{Q}_1}$ the resulting quantum $\pazocal{X}$-torus, we give a constructive proof of the following (see \Cref{prop-match} for a more detailed statement)
\begin{theorem}\label{thm:last}
    The image via the functor $\mathcal F_q$ of the $\Mat_2(\pazocal{X}_{\pazocal{Q}_2}^{\nicefrac{1}{2}})$-embedding of $H_{D_4}(t,q)$ is isomorphic to     
    the $\Mat_3(\pazocal{X}_{\pazocal{Q}_1}^{\nicefrac{1}{3}})$-embedding  of $H_{E_6}(t,q)$.
\end{theorem}

The two quiver seizures allowing for this identification can be interpreted as the result of two colliding holes in the
sense of \cite{Chekhov2017}.

The present paper is organised as follows:

In \Cref{sec:functor}, we define the functor of \Cref{thm:functor} by introducing a quantum analogue of both Katz's middle convolution and  Killing's factorization.  In the process, we construct the analogue of the basic representation for $H_{E_6}(\tilde{t},q)$.

In \Cref{sec:HTT}, we recall the basics of higher Teichm\"uller theory and the moduli space of pinnings. We give a brief self-consistent summary of the Fock-Goncharov coordinates for the moduli space of $\PGL_n(\real)$-local systems and their extension to the moduli space of pinnings due to Goncharov and Shen.
We describe the so-called snake calculus, detailing how to compute transport matrices and glue triangles by amalgamations.
En passant, we prove that (positive) Fock-Goncharov coordinates equally define the coordinate ring of the higher bordered cusped Teichm\"uller space
$$
\hbox{Hom}'\left(\pi_{\mathfrak a}(\Sigma_{g,s,m}), \PSL_n(\real) \right){\big\slash\Pi_{i=1}^m B_i},
$$
introduced in \cite{Chekhov2017}.
After giving a recipe to represent fat graph loops by strings of transport matrices, we conclude explaining the quantization of the Fock-Goncharov coordinates and their fractional extensions.

In \Cref{sec:GDAHAreps}, after reproducing the $\Mat_2(\mathbb{T}^2_q)$-embedding of the $\tilde{D}_4$-type GDAHA in the language of \Cref{sec:HTT}, we prove \Cref{thm:E6-emb}. 

In \Cref{sec:cluster-seizure}, we introduce the quiver seizure operation and prove \Cref{thm:last}.

Finally, in \Cref{app:analysis} we summarise the  analytical background behind our constructions while 
\Cref{app:formulae} lightens the reading by collecting the most involved formulae.

\paragraph{Acknowledgements}
\!\!\!The authors are grateful to 
L. Chekhov, P. Etingof, D. Kaplan, A. Neitzke, E. Rains, V. Rubtsov, A. Shapiro and V. Toledano Laredo for many helpful discussions. The research of D. Dal Martello was funded by the EPSRC Studentship 2438494. The research of M. Mazzocco was supported by the Leverhulme Trust Research Project Grant RPG-2021-047.
Neither funding sources had any involvement in the research and/or preparation of this article.

\section{The GDAHA functor}\label{sec:functor}

Let $\Rep(\pazocal{A})$ denote the category of representations of the algebra $\pazocal{A}$: objects are pairs $(\rho,V)$, for $V$ a vector space and $\rho: \pazocal{A} \to \mathrm{End}(V)$ an algebra homomorphism, while arrows are homomorphisms of representations.

In this section, by introducing a quantum analogue of both Katz’s
middle convolution and the Killing factorization, we give a constructive proof of \Cref{thm:functor}.

\subsection{Quantum middle convolution}\label{sec:MC}
Katz \cite{Katz1996} originally defined the middle convolution for local systems as a transformation preserving
rigidity and irreducibility.
We here introduce a noncommutative version of its algebraic analogue as defined in \cite{Dettweiler2007}, tailored to our GDAHA setting.

For convenience, we rescale the generators of $H_{D_4}(t,q)$ as \begin{equation}
    \widehat K_i= \frac{1}{t_i} K_i,\quad i=1,2,3;\qquad \widehat K_4=t_1 t_2 t_3 K_4,
\end{equation}
so that the Hecke relations can be written as follows:
\begin{equation}\label{D4-gen}
\begin{aligned}
    (\widehat K_i-1)\left(\widehat K_i-\frac{1}{t_i^2}\right)&=0,\quad i=1,2,3;\\
    (\widehat K_4- t_1 t_2 t_3 t_4)\left(\widehat K_4-\frac{t_1 t_2 t_3}{t_4}\right)&=0.
\end{aligned}
\end{equation}
In doing so, the cyclic relation is preserved:
\begin{equation}\label{D4-gen1}
\begin{aligned}
   \widehat K_1 \widehat K_2 \widehat K_3 \widehat K_4=q^{\nicefrac{-1}{2}}.
\end{aligned}
\end{equation}
For an object  $(\rho,V)\in\Repfour$, we denote $\rho(\widehat K_i)$ by $\widehat K_i$, namely we 
use the same notation for the generator and its representation $\widehat K_i \in \mathrm{End}(V)$.

Introducing the triple $\widehat{\mathbf{K}}:=(\widehat K_1,\widehat K_2,\widehat K_3)$, the first map we define is 
\begin{equation*}
   \begin{matrix}
       \C \ : & \mathrm{End}(V)^3 & \to & \mathrm{End}(\bigoplus_{3} V)^3\\
       & (\widehat K_1,\widehat K_2,\widehat K_3) & \mapsto & (N_1,N_2,N_3)
   \end{matrix}
\end{equation*}
where
\begin{equation}
N_1=\begin{pmatrix}
    \widehat K_1 & \widehat K_2-1 & \widehat K_3-1\\
0&1&0\\
0&0&1
\end{pmatrix},\phantom{-}
N_2=\begin{pmatrix}
    1&0&0\\
\widehat K_1-1 & \widehat K_2 & \widehat K_3-1\\
0&0&1
\end{pmatrix},\phantom{-}
N_3=\begin{pmatrix}
    1&0&0\\0&1&0\\
\widehat K_1-1 & \widehat K_2-1 & \widehat K_3
\end{pmatrix}.
\end{equation}
Notice that $(\C(\widehat{\mathbf{K}}),\bigoplus_{3} V )$ no longer defines a representation of $H_{D_4}(t,q)$.
The algebra structure in $\mathrm{End}(\bigoplus_{3}V)$ is given by the usual matrix multiplication, combined with the algebra operations in $H_{D_4}(t,q)$.
In particular, the ordering is dictated by that of matrix multiplication.

The next operation is a quantum quotient to a subspace encoding the Hecke properties of $\widehat{\mathbf{K}}$.
\begin{lemma}
The subspace $W\subset \bigoplus_{3} V$, defined as
\begin{equation}        W:=\bigoplus_{i=1}^3
          \ker(\widehat K_i-1),
\end{equation}
is invariant under the action of $N_1,N_2$ and $N_3$.
\end{lemma} 
\begin{proof}
    For any $\mathbf{v}=(v_1,v_2,v_3)\in W$, $N_1(\mathbf{v})=(\widehat{K}_1(v_1)+(\widehat{K}_2-1)(v_2)+(\widehat{K}_3-1)(v_3),v_2,v_3)$. Since $v_i$ is in the kernel of $(\widehat{K}_i-1)$, $(\widehat{K}_2-1)(v_2)+(\widehat{K}_3-1)(v_3)=0$ while $\widehat{K}_1(v_1)=v_1$. Analogous computations can be repeated for $N_2$ and $N_3$. 
\end{proof}
The quantum middle convolution is the restriction of $\C(\widehat{\mathbf{K}})$ to the quotient $(\bigoplus_{3} V) /W$.
To construct this quotient, we take advantage of the properties entailed by the Hecke relations. In particular, each operator $\widehat{K}_i:V \rightarrow V$ carries a natural direct sum decomposition of $V$ into eigenspaces:
\begin{equation}\label{dir-sum-dec}
V=V_i^{(1)}\oplus V_i^{(2)},
\end{equation}
where $V_i^{(1)}$ corresponds to eigenvalue $1$ and $V_i^{(2)}$ to the other eigenvalue $t_i^{-2}$.
\begin{lemma}\label{lm:idempotents}
    The operators $e_i$,  defined as
\begin{equation}
    e_i:=\frac{t_i^2}{1-t_i^2}(\widehat{K}_i-1),
\end{equation}
are idempotent and project onto the eigenspace $V_i^{(2)}$:
\begin{equation*}
     e_i^2=e_i, \quad \widehat K_i e_i =\frac{1}{t_i^2} e_i.
\end{equation*}
Moreover, denoting
\begin{equation}
    \Bar{e}_i:=\frac{t_i^2}{t_i^2-1}(\widehat{K}_i-t_i^{-2})
\end{equation}
the complement idempotent element projecting onto $V_i^{(1)}$, the following relations hold for $i=1,2,3$:
\begin{equation*}
    e_i\Bar{e}_i=\Bar{e}_ie_i=0, \quad e_i+\Bar{e}_i=1.    
\end{equation*}
\end{lemma}
\begin{proof}
    It all stems from the Hecke relation $(\widehat K_i-1)(\widehat K_i-t_i^{-2})=(\widehat K_i-t_i^{-2})(\widehat K_i-1)=0$:
    \begin{gather*}
        e_i^2=\frac{t_i^4}{(1-t_i^2)^2}(\widehat{K}_i-1)^2=\frac{t_i^4}{(1-t_i^2)^2}(t_i^{-2}-1)(\widehat{K}_i-1)=e_i,\allowdisplaybreaks\\
        \widehat{K}_i e_i=\frac{t_i^2}{1-t_i^2}\widehat{K}_i(\widehat{K}_i-1)=\frac{t_i^2}{1-t_i^2}\frac{1}{t_i^2}(\widehat{K}_i-1)=\frac{1}{t_i^2}e_i,\allowdisplaybreaks\\
        e_i\Bar{e}_i=-\frac{t_i^4}{(1-t_i^2)^2}(\widehat{K}_i-1)(\widehat K_i-t_i^{-2})=0,\allowdisplaybreaks\\
        e_i+\Bar{e}_i=\frac{t_i^2}{1-t_i^2}(\widehat{K}_i-1)+\frac{t_i^2}{t_i^2-1}(\widehat{K}_i-t_i^{-2})=\frac{t_i^2}{t_i^2-1}-\frac{1}{t_i^2-1}=1.\qedhere
    \end{gather*}
\end{proof}
Introducing the operator 
\begin{equation}
    E:=e_1\oplus e_2\oplus e_3\in\mathrm{End}(\bigoplus_{3} V),
\end{equation} 
we can finally give the following
\begin{definition}\label{def-MC}
    Let $\mathsf{VECT}^2$ be the arrow category of vector spaces and $E(V):=e_1(V)\oplus e_2(V) \oplus e_3(V)$. The \emph{quantum middle convolution} is the map
    \begin{equation}
       \MC \ : \Repfour  \to  \mathsf{VECT}^2
   \end{equation}
    sending an object $(\widehat{\mathbf{K}},V)\in\Repfour$ to a triple  $(EN_1,EN_2,EN_3)\in \mathrm{End}(E(V))^3$.
\end{definition}
\begin{proposition}\label{prop:Misfunctor}
    $\MC$ is a functor whose image consists of pseudo-reflections.
\end{proposition}
\begin{proof}
    Let us first characterize $\C(\widehat{\mathbf{K}})$ explicitly in $\mathrm{End}(E(V))$: since $\widehat K_i$ acts as the multiplication by $t_i^{-2}$ on $e_i(V)$, it is immediate to obtain the pseudo-reflection formulae
    \begin{equation}\label{EN}
        \begin{aligned}
            {EN_1}_{\big|_{E(V)}}&=\begin{pmatrix}
        t_1^{-2} & (t_2^{-2}-1)e_1 & (t_3^{-2}-1)e_1 \\
        0 & 1 & 0\\
        0 & 0 & 1
    \end{pmatrix},\\
    {EN_2}_{\big|_{E(V)}}&=\begin{pmatrix}
        1 & 0 & 0\\
        (t_1^{-2}-1)e_2 & t_2^{-2} & (t_3^{-2}-1)e_2\\
        0 & 0 & 1    \end{pmatrix},\\
        {EN_3}_{\big|_{E(V)}}&=\begin{pmatrix}
        1 & 0 & 0\\
        0 & 1 & 0\\
        (t_1^{-2}-1)e_3 & (t_2^{-2}-1)e_3 & t_3^{-2}
    \end{pmatrix}.
        \end{aligned}
    \end{equation} 
    
    Now let $(\mathbf{K},V),\,(\mathbf{K}',V')$ be objects in $\Repfour$ and $\phi:V\rightarrow V'$ be a homomorphism of representations, i.e., for $i=1,2,3$ the following diagram commutes:
    \begin{equation}\label{phidiagram}
    \begin{tikzcd}
V \arrow{r}{\phi} \arrow[swap]{d}{K_i} & V' \arrow{d}{K'_i} \\%
V \arrow{r}{\phi}& V'
\end{tikzcd}
    \end{equation} 
Since representations in the same category have the same parameters $(q,t)$, ${\phi}$ also commutes with the rescaled representations.
In order to define the functor on arrows, we first introduce the map 
\begin{equation}\label{Cphi}
\C(\phi):=(\bigoplus_3\phi)^3:(\bigoplus_{3} V)^3\to(\bigoplus_{3} V')^3
\end{equation}
as the arrow in $\mathsf{VECT}^2$ making the diagram
\begin{equation}
    \begin{tikzcd}
    (\bigoplus_{3} V)^3 \arrow{r}{\C(\phi)} \arrow[swap]{d}{(N_1,N_2,N_3)} & (\bigoplus_{3} V')^3 \arrow{d}{(N'_1,N'_2,N'_3)} \\%
    (\bigoplus_{3} V)^3 \arrow{r}{\C(\phi)}& (\bigoplus_{3} V')^3
    \end{tikzcd}
\end{equation}
commute. E.g.,
\begin{equation}\label{phiexample}
    (\bigoplus_3\phi)N_1=\begin{pmatrix}
    \phi\widehat K_1 & \phi\widehat K_2-\phi & \phi\widehat K_3-\phi\\
0&\phi&0\\
0&0&\phi
\end{pmatrix}\overset{\eqref{phidiagram}}{=\joinrel=}\begin{pmatrix}
    \widehat K'_1\phi & \widehat K'_2\phi-\phi & \widehat K'_3\phi-\phi\\
0&\phi&0\\
0&0&\phi
\end{pmatrix}=N'_1(\bigoplus_3\phi).
\end{equation}
Analogously, $(\bigoplus_3\phi)N_i=N'_i(\bigoplus_3\phi)$ holds for $i=2,3$.
Since $(\bigoplus_3\phi)E=E'(\bigoplus_3\phi)$ given that $\phi e_i=e'_i\phi$, the map \eqref{Cphi} restricts to $E(V)^3$ as
\begin{equation}
    \MC(\phi):=(\bigoplus_3\phi)^3\,:\,E(V)^3 \to E'(V')^3,
\end{equation}
defining the functor $\MC$ on arrows.
Indeed, the diagram
\begin{equation}
    \begin{tikzcd}
    E(V)^3 \arrow{r}{\MC(\phi)} \arrow[swap]{d}{(EN_1,EN_2,EN_3)} & E'(V')^3 \arrow{d}{(E'N'_1,E'N'_2,E'N'_3)} \\%
    E(V)^3 \arrow{r}{\MC(\phi)}& E'(V')^3
    \end{tikzcd}
\end{equation}
commutes given that \eqref{phiexample} restricts as
\[
(\bigoplus_3\phi)EN_1=E'(\bigoplus_3\phi) N_1=E'N'_1(\bigoplus_3\phi).
\]
Functoriality is a straightforward consequence of the definitions: for the identity $\mathrm{id}:V\to V$, the equality $\MC(\mathrm{id})=\mathrm{id}$ manifestly holds while given two arrows $\phi:V\to V'$ and  $\psi:V'\rightarrow V''$, $ \MC(\psi\phi)=\MC(\psi)\MC(\phi)$ 
follows from 
\begin{equation*}
    \bigoplus_3(\psi\phi)=(\bigoplus_3\psi)(\bigoplus_3\phi)\in\mathrm{Hom}(E(V),E''(V'')).\qedhere
\end{equation*}
\end{proof}
\begin{remark}
    In the language of Katz \cite{Katz1996}, \Cref{def-MC} is the quantum algebraic analogue of $M(\infty,\pazocal{F})$.
    It corresponds to quotient by only the $\pazocal{K}$ subspace in \cite{Dettweiler2007} or, equivalently, to assume $\lambda$ generic and set it to $1$ after the restriction is performed.    
    The noncommutative construction taking full account of the other subspace is to appear in a different paper: whenever $\pazocal{L}$ is nontrivial, the image of the quantum middle convolution must remain an object in $\Repfour$, with no hope of being mapped to a different GDAHA. 
\end{remark}

\subsection{Quantum Killing factorization}\label{se:q-killing}
This section extends to the noncommutative realm a classical result tracing its origin back to Killing \cite{Coleman1989}:
\begin{lemma}\label{th-q-killing}
For a noncommutative ring $\pazocal{R}$ with unit group $\pazocal{R}^*$, let $R_1,R_2,R_3 \in \Mat_3(\pazocal{R})$ be pseudo-reflections: for $a_{ij}\in\pazocal{R}$,
\begin{equation}
    R_1=\begin{pmatrix}
        a_{11} & a_{12} & a_{13}\\0 & 1 & 0\\0 & 0 & 1
    \end{pmatrix},\quad
    R_2=\begin{pmatrix}
        1 & 0 & 0\\a_{21} & a_{22} & a_{23}\\0 & 0 & 1
    \end{pmatrix},\quad
    R_3=\begin{pmatrix}
        1 & 0 & 0\\0 & 1 & 0 \\a_{31} & a_{32} & a_{33}
    \end{pmatrix}.
\end{equation}
Then, their product is uniquely factorized as
\begin{equation}\label{eq:fact-refl}
    R_1R_2R_3=UL,
\end{equation}
for $U$ upper unitriangular and $L$ lower triangular given by
\begin{equation}
    L-(U^{-1}-1)={A},\quad ({A})_{ij}=a_{ij}.
\end{equation}
Moreover, when $a_{ii}\in \pazocal{R}^*$, $R_i$ is invertible in $\Mat_3(\pazocal{R})$.
\end{lemma}
\begin{proof}
    With the ordering within the entries induced by that of the matrix multiplication, by direct computation we obtain
    \begin{multline}
        R_1R_2R_3=\\\begin{pmatrix}
        a_{11}+a_{12}a_{21}+a_{13}a_{31}+a_{12}a_{23}a_{31} & a_{12}a_{22}+a_{13}a_{32}+a_{12}a_{23}a_{32} & a_{13}a_{33}+a_{12}a_{23}a_{33}\\a_{21}+a_{23}a_{31} & a_{22}+a_{23}a_{32} & a_{23}a_{33}\\ a_{31} & a_{32} & a_{33}
    \end{pmatrix}.
    \end{multline}
    Multiplying this formula on the left by a suitable upper unitriangular matrix $U^{-1}$, we obtain a lower triangular result:
    \begin{equation}\label{invU}
        \underbrace{\begin{pmatrix}
        1 & -a_{12} & -a_{13}\\0 & 1 & -a_{23}\\ 0 & 0 & 1
    \end{pmatrix}}_{U^{-1}} R_1R_2R_3=\underbrace{\begin{pmatrix}
        a_{11} & 0 & 0\\a_{21} & a_{22} & 0\\ a_{31} & a_{32} & a_{33}
    \end{pmatrix}}_{L}.
    \end{equation}
To calculate $U$, we use the fact that $U^{-1}$ is unipotent: $(U^{-1}-\One)^3=\Zero$ implies
\begin{equation}
    U=U^{-2}+3U^{-1}-3\One=\begin{pmatrix}
        1 & a_{12} & a_{13}+a_{12}a_{23}\\0 & 1 & a_{23}\\ 0 & 0 & 1
    \end{pmatrix}.
\end{equation}

Assuming that $a_{11}\in \pazocal{R}^*$ and denoting by $a_{11}^{-1}$ its multiplicative inverse,
\begin{equation}
    R_1^{-1}=\begin{pmatrix}
        a_{11}^{-1} & -a_{11}^{-1}a_{12} & -a_{11}^{-1}a_{13}\\0 & 1 & 0\\0 & 0 & 1
    \end{pmatrix}
\end{equation}
and analogous formulae hold for $R_2$ and $R_3$.
When $a_{ii} \in \pazocal{R}^*$ for $i=1,2,3$, the triple product  $R_1R_2R_3$ can be inverted too.
Doing so via the factorization, since $U^{-1}$ is known it suffices to invert $L$:
\begin{equation}\label{invL}
    L^{-1}=\begin{pmatrix}
        a_{11}^{-1} & 0 & 0\\-a_{22}^{-1}a_{21}a_{11}^{-1} & a_{22}^{-1} & 0\\-a_{33}^{-1}a_{31}a_{11}^{-1}+a_{33}^{-1}a_{32}a_{22}^{-1}a_{21}a_{11}^{-1} & -a_{33}^{-1}a_{32}a_{22}^{-1} & a_{33}^{-1}
    \end{pmatrix}.\qedhere
\end{equation}
\end{proof}
Applying \Cref{th-q-killing} to $R_i=  {EN_i}_{\big|_{E(V)}}$ for $i=1,2,3$, the factorization \eqref{eq:fact-refl} takes the  form
\begin{equation}\label{UinMC}
    U=\begin{pmatrix}
        1 & (t_2^{-2}-1)e_1 & (t_3^{-2}-1)e_1+(t_2^{-2}-1)(t_3^{-2}-1)e_1e_2\\0 & 1 & (t_3^{-2}-1)e_2\\ 0 & 0 & 1
    \end{pmatrix},
    \end{equation}
    \begin{equation}\label{LinMC}
    L=\begin{pmatrix}
        t_1^{-2} & 0 & 0\\(t_1^{-2}-1)e_2 & t_2^{-2} & 0\\ (t_1^{-2}-1)e_3 & (t_2^{-2}-1)e_3 & t_3^{-2}
    \end{pmatrix}.
\end{equation}
Moreover, defining $h_i:=(t_i^{-2}-1)$,
\begin{equation}\resizebox{\textwidth}{!}{$
\begin{aligned}
    R_1&R_2R_3=\\
   &\begin{pmatrix}        t_1^{-2}+h_1h_2e_1e_2+h_1h_3e_1(1+h_2e_2)e_3 & t_2^{-2}h_2e_1+h_2h_3e_1(1+h_2e_2)e_3 & t_3^{-2}h_3e_1+t_3^{-2}h_2h_3e_1e_2\\h_1e_2+h_1h_3e_2e_3 & t_2^{-2}+h_2h_3e_2e_3 & t_3^{-2}h_3e_2\\ h_1e_3 & h_2e_3 & t_3^{-2}
    \end{pmatrix}.
    \end{aligned}$}
\end{equation}

\subsection{The functorial composition}

Composing the noncommutative analogues of the Killing factorization and the middle convolution provides a tool to construct representations of a specialized $\Tilde{E}_6$-type GDAHA: 
\begin{lemma}\label{lem:RepE_6}
    Given an object $(\rho,V)\in\Repfour$, let $U$ and $L$ be the quantum Killing factors of $EN_1EN_2EN_3$, where $(EN_1,EN_2,EN_3)$ is the triple of pseudo-reflections \eqref{EN}.
    Denoting the inverse triple product by $\Pi:=(EN_1EN_2EN_3)^{-1}$ , the following relations hold:
    \begin{equation}\label{eq-HULPi}
        \begin{aligned}
            (U-1)(U-1)(U-1)&=0,\allowdisplaybreaks\\
            \left(L-t_1^{-2}\right)\left(L-t_2^{-2}\right)\left(L-t_3^{-2}\right)&=0,\allowdisplaybreaks\\
            \left(\Pi-1\right)\left(\Pi-\sqrt{q}\,{t_1t_2t_3t_4}\right)\left(\Pi-\sqrt{q}\,\frac{t_1t_2t_3}{t_4}\right)&=0. 
        \end{aligned}
    \end{equation}
In particular, the rescaled operators
\begin{equation}\label{rescalings}
    \widehat{L}:=(t_1t_2t_3)^{\nicefrac{2}{3}}L, \quad \widehat{\Pi}:= \frac{1}{q^{\nicefrac{1}{3}}(t_1t_2 t_3)^{\nicefrac{2}{3}}}\Pi
\end{equation}
satisfy the Hecke relations
\begin{equation}
        \begin{aligned}
            \left(\widehat{L}-t_1^{\nicefrac{-4}{3}}t_2^{\nicefrac{2}{3}}t_3^{\nicefrac{2}{3}}\right)\left(\widehat{L}-t_1^{\nicefrac{2}{3}}t_2^{\nicefrac{-4}{3}}t_3^{\nicefrac{2}{3}}\right)\left(\widehat{L}-t_1^{\nicefrac{2}{3}}t_2^{\nicefrac{2}{3}}t_3^{\nicefrac{-4}{3}}\right)&=0,\\
             \left(\widehat{\Pi}-\frac{1}{q^{\nicefrac{1}{3}}t_1^{\nicefrac{2}{3}}t_2^{\nicefrac{2}{3}} t_3^{\nicefrac{2}{3}}}\right)\left(\widehat{\Pi}-
             q^{\nicefrac{1}{6}}t_1^{\nicefrac{1}{3}}t_2^{\nicefrac{1}{3}}t_3^{\nicefrac{1}{3}}t_4\right)\left(\widehat{\Pi}-q^{\nicefrac{1}{6}}\frac{t_1^{\nicefrac{1}{3}}t_2^{\nicefrac{1}{3}}t_3^{\nicefrac{1}{3}}}{t_4}\right)&=0,    
        \end{aligned}
    \end{equation}
together with the cyclic one
\begin{equation}\label{ULPhat}
    U \ \widehat{L} \ \widehat{\Pi}=q^{\nicefrac{-1}{3}}.
\end{equation}
\end{lemma}
\begin{proof}
    By construction, $U\,L\,\Pi=1$ and \eqref{ULPhat} follows immediately. As an \emph{upper} triangular matrix of operators, $U$ automatically satisfies a Hecke relation with its diagonal entries as parameters---which are forced to be unities by the quantum factorization. Being lower triangular, $L$ satisfies the analogous Hecke relation if and only if its diagonal is made of invertible elements---which is the case for the factorization of $(EN_1,EN_2,EN_3)$, see \eqref{LinMC}.
    
    To prove the remaining Hecke relation for $\Pi$, we use the so-called basic (faithful) representation of $H_{D_4}(t,q)$ \cite{Macdonald2003}.
    This is explicitly given by the three operators $T_0,T_1,Z$ acting on the space of Laurent polynomials $f[z]\in V:=\mathbb C[z^{\pm1}]$ as follows:
\begin{gather}
(Zf)[z]:=z\,f[z],
\label{14}\\
(T_1(a,b)f)[z]:=\frac{(a+b)z-(1+ab)}{1-z^2}\,f[z]+
\frac{(1-az)(1-bz)}{1-z^2}\,f[z^{-1}],
\label{15}\\
(T_0(a,b,c,d)f)[z]:=\frac{q^{-1}z((cd+q)z-(c+d)q)}{q-z^2}\,f[z]
-\frac{(c-z)(d-z)}{q-z^2}\,f[qz^{-1}].
\label{16}
\end{gather}
These operators satisfy the algebra relations
\begin{equation}\label{sahi2}
\begin{aligned}
(T_1+ a b)(T_1+1)&=0,\\
(T_0+q^{-1} c d)(T_0+1)&=0,\\
(T_1 Z  +a)(T_1 Z +b)&=0,\\
(qT_0 Z^{-1}+ c)(qT_0 Z^{-1}+d)&=0.
\end{aligned}
\end{equation}
To put these relations in form \eqref{D4-gen}, we set
\begin{equation}
\begin{gathered}\label{hKandt}
    \widehat{K}_1=- T_1, \quad \widehat{K}_2=-aT_1^{-1}Z^{-1},\quad \widehat{K}_3=-T_0,\quad \widehat{K}_4=-\frac{1}{a\sqrt{q}} T_0^{-1}Z,\\
    t_1^2=\frac{1}{ab},\qquad t_2^2=\frac{b}{a}, \qquad t_3^2= \frac{q}{cd}, \qquad
    t_4^2=\frac{c}{d}.
\end{gathered}
\end{equation}
Notice that with this choice, among the new relations we have the cyclic one as \eqref{D4-gen1}.

Despite the fact that the operators $\widehat K_i$ act on 
 the infinite dimensional $\complex$-vector space of Laurent polynomials $\complex[z^{\pm1}]$, we can give an explicit characterization to their eigenspaces:
\begin{lemma}[\cite{Koornwinder2007}]\label{lm:tk} Let ${\mathsf{Sym}}$ denote the space of symmetric Laurent polynomials, 
\begin{equation}
    {\mathsf{Sym}}=\left\{f\in \complex[z^{\pm1}]\,|\, f[z]=f[z^{-1}]\right\},
\end{equation}
and ${\mathsf{Sym}_q}$ denote the space of q-symmetric Laurent polynomials, 
\begin{equation}
    {\mathsf{Sym}_q}=\left\{f\in \complex[z^{\pm1}]\,|\, f[z]=f[qz^{-1}]\right\}.
\end{equation}
Then, 
    \begin{equation*}
    \begin{aligned}
         \widehat{K}_1f[z]&=abf[z] \,\iff \,f[z]\in {\mathsf{Sym}},\\
          \widehat{K}_2f[z]&=\frac{a}{b}f[z] \,\iff \,f[z]=(bz-1)p[z],\quad p[z]\in {\mathsf{Sym}},\\
           \widehat{K}_3f[z]&=\frac{cd}{q}f[z]\, \iff \,f[z]\in {\mathsf{Sym}_q}.
    \end{aligned}       
    \end{equation*}
\end{lemma}
Thanks to \Cref{lm:tk}, we have that 
$$
e_1(V)={\mathsf{Sym}},\quad 
e_2(V)=(bz-1)\,{\mathsf{Sym}},\quad 
e_3(V)={\mathsf{Sym}_q},
$$
which allows to give an explicit restriction for the triple of operators resulting from applying $\MC$ to the $(\widehat{K}_1,\widehat{K}_2,\widehat{K}_3)$ in \eqref{hKandt}. The restricted operators act on a generic element in the quotient $(v_1[z],v_2[z],v_3[z])\in E(V)=e_1(V)\oplus e_2(V)\oplus e_3(V)$ as follows:
    \begin{multline}    
    EN_1(v_1[z],v_2[z],v_3[z])=
    \left( {ab}v_1[z] + \frac{(a - b)(b-z)}{b(a b-1) z} v_2[z]\right.\\
  +\frac{(c d - q)}{(a b-1)q( z^2-1)}
  \left.
  \big((a z-1) (b z-1)v_3[z^{-1}]
  - (a -z) (b-z)v_3[z]\big),
 \vphantom{\frac{(c d - q)}{(a b-1)q( z^2-1)}}v_2[z],\, v_3[z]\right),
\end{multline}
\begin{multline}   
EN_2(v_1[z],v_2[z],v_3[z])=\left(v_1[z],\,-\frac{a(a b-1)(b z-1)}{(a - b)}v_1[z]+\frac{a}{b}v_2[z]\right.\\
\left.+\frac{(c d - q)(b z-1)}{(a -b)q( z^2-1)}\big((a-z)v_3[z]-z(az-1)v_3[z^{-1}]\big),\, v_3[z]\right),
\end{multline}
\begin{multline}    
EN_3(v_1[z],v_2[z],v_3[z])=\left(\vphantom{\frac{a}{b}}v_1[z],\, v_2[z],\right.\\
\frac{(a b-1)}{(c d - q) (q - z^2)}\big(q(c - z)(z-d)v_1[qz^{-1}]
-( c z-q)(d z-q)v_1[z]\big)\\\left.
+\frac{(a-b)}{b(c d - q) (q - z^2)}\big(q(c - z)(d-z)v_2[qz^{-1}]-(c z-q)(d z-q)v_2[z]\big)+ \frac{cd}{q}v_3[z]\right).
\end{multline}
It is immediate to put these operators in matrix form and read off their Killing factors as explained in \Cref{se:q-killing}.
We obtain the following operators:
\begin{multline}\label{eq:Lb-rep}
     L(v_1[z],v_2[z],v_3[z])= \left( {ab}v_1[z],\, \frac{a(a b-1)(b z-1)}{a-b }v_1[z]+ \frac{a}{b}v_2[z],\right.\\
     \left.\frac{(a b-1)}{(c d - q) (q - z^2)}
     \big(q(c-z)(d-z)v_1[qz^{-1}]-(c z-q)(d z-q)v_1[z]\big)\right.\\
   +\left.\frac{(a-b)}{b(c d - q) (q - z^2)}
  \big(q(c-z)(d-z)v_2[qz^{-1}]-(c z-q)(d z-q)v_2[z]\big)+ \frac{q}{cd}v_3[z]\right),
\end{multline}
\begin{multline}\label{eq:Ub-rep}
    U(v_1[z],v_2[z],v_3[z])=
    \left(v_1[z] + 
    \frac{(a - b)(b-z)}{b (a b-1)z}v_2[z]\right.\\+ \frac{(c d - q)}{q(a b-1)b (z^2-1)}
    \big( z(a z-1)(b z-1) v_3[z^{-1}]- \frac{(a-z)(b-z)}{z} v_3[z]\big),\\
  \left.v_2[z] + \frac{(c d - q) (b z-1)}{(a - b) q (z^2-1)}\big((a-z) v_3[z]-z(a z-1) v_3[z^{-1}] \big),\, v_3[z]\right).
\end{multline}
Moreover, we set $\Pi= L^{-1}U^{-1}$, where $L^{-1}$ and $U^{-1}$ are computed as prescribed in \Cref{se:q-killing}:
\begin{multline}\label{eq:Pib-rep}
   \Pi(v_1[z],v_2[z],v_3[z]) = \left(\frac{1}{ab} v_1[z]+\frac{(a-b)(z-b)( b z -1)}{a b^2(a b-1)z} v_2[z]\right.\\
   +\frac{(c d-q)}{a b (a b-1)q(z^2-1)} \big((z-a)(z-b) v_3[z]-(az-1)(bz-1) v_3[z^{-1}]\big),\\
   \frac{(a b-1)(bz-1)}{a(a-b)}v_1[z]+\frac{(b z-1)(b-z+b z^2)}{a b z}v_2[z]\\
   +\frac{(c d-q)(b z-1)}{a(a-b)q(z^2-1)}\big((a z-1) v_3[z^{-1}]-(a-z)v_3[z]\big),\\
     \frac{(ab-1)q}{a c d(c d-q) (z^2-q)}\Big(\frac{q^2}{z}(c-z)(d-z)v_1[qz^{-1}]-z(c z-q)(d z-q)v_1[z]\Big)\\
     +\frac{(a-b)q}{a b c d (cd-q)(z^2-q)}\Big(\frac{q^2}{z^2}(b q-z)(c-z)(d-z)v_2[qz^{-1}]-z(b z-1)(c z-q)(d z-q)v_2[z]\Big)\\
    -\frac{q(c-z)(d-z)(a z-q)}{a c d z (z^2-q)} v_3[qz^{-1}]-\left.\frac{z\left(a q(d-z)+(a c+ q-c z)(q-d z)\right)}{a c d(z^2-q)}\vphantom{\frac{1}{ab}}\right).
\end{multline}

The Hecke relations for $L$ and $U$  
can be easily checked directly using formulae (\ref{eq:Lb-rep}-\ref{eq:Ub-rep}), while $U\,L\,\Pi= 1$ holds by construction.
Verifying the Hecke relation for $\Pi$ is a heavy computation best performed with Mathematica \cite{DalMartello2023}.
This concludes the proof of formulae \eqref{eq-HULPi} with parameters \eqref{hKandt}. \end{proof}
We are now ready to conclude the proof of \Cref{thm:functor}, restated here in more detail:
\begin{theorem}\label{thm:truefunctor}
The quantum Killing factorization of the quantum middle convolution gives a functor of (faithful) representations
\begin{equation}
    \begin{array}{lccc}
\mathcal{F}_q\ : & \ \Repfour &\to & \Rep\big(H_{E_6}(\tilde{t},q)\big),\\
&(\rho,V)&\mapsto&\left(\eta,E(V)\right),
\end{array}
\end{equation}
where
$$
E(V):=e_1(V)\oplus e_2(V) \oplus e_3(V), 
$$
with $e_i$ defined in \Cref{lm:idempotents}, and $\eta:H_{E_6}(\tilde{t},q)\to \mathrm{End}(E(V))$ is the algebra homomorphism 
acting on the generators $J_1,J_2,J_3$ of $H_{E_6}(\tilde{t},q)$
as
\begin{equation}
    \eta(J_1)=U,\quad \eta(J_2)=\widehat{L},\quad \eta(J_3)=\widehat{\Pi},
\end{equation}
with $U,\ \widehat{L}$ and $\widehat{\Pi}$ defined in \Cref{lem:RepE_6} and the 
 parameters $\tilde{t}$ given by \eqref{eq:paramE6}.
\end{theorem}
\begin{proof}
In \Cref{lem:RepE_6}, we have already proven that $\mathcal F_q$ maps  objects $(\rho,V)\in\Repfour$ to objects $(\eta,E(V))\in \Rep\big(H_{E_6}(\tilde{t},q)\big)$. 
The faithfulness of $(\eta,E(V))$ stems from the same argument used in the proof of \Cref{thm:TrueE6-emb}.

Now, let $(\rho,V)$ and $(\rho',V')$ be two objects in $\Repfour$ and $\phi:V\rightarrow V'$ a homomorphism of representations.
The map of arrows defined in the proof of \Cref{prop:Misfunctor} carries through the factorization: for $i=1,2,3$, $\mathcal{F}_q(\phi):=\bigoplus_3\phi$  gives the commutative diagram
\begin{equation}
    \begin{tikzcd}
    E(V) \arrow{r}{\mathcal{F}_q(\phi)} \arrow[swap]{d}{\eta(J_i)} &  E'(V') \arrow{d}{\eta'(J_i)} \\%
    E(V) \arrow{r}{\mathcal{F}_q(\phi)}& E'(V')
    \end{tikzcd}
\end{equation}
Indeed, each Killing factor's entry is a (linear combination of) composition of entries from the triple $(EN_1,EN_2,EN_3)$ and these suitably commute with $\phi$: as previously observed, $\phi e_i=e'_i\phi$.

To conclude, functoriality holds unaffected: for the identity $\mathrm{id}:V \rightarrow V$, $\mathcal{F}_q(\mathrm{id})=\mathrm{id}$ is manifest while for $\psi:V'\rightarrow V''$ satisfying $\psi \rho'(K_i)=\rho''(K_i)\psi$, $\mathcal{F}_q(\psi\phi)=\bigoplus_3(\psi\phi)=(\bigoplus_3\psi)(\bigoplus_3\phi)$ as maps in $\mathrm{Hom}(E(V),E''(V''))$.
\end{proof}
\begin{remark}
    In principle, one can obtain a wealth of representations of $H_{E_6}(\tilde{t},q)$ by feeding the functor $\mathcal{F}_q$ with the representation theory of $H_{D_4}$. We provide two examples of distinct nature with the basic representation above and the quantum matricial one in \Cref{sec:cluster-seizure}.
    
    Notice also that $\Tilde{t}$ specializes only the two parameters $t_1^{(1)}$ and $t_1^{(2)}$: the remaining four are free, as confirmed by the following inversion formulae
    \begin{equation}
            t_1=q^{\nicefrac{-1}{6}}\tilde{t}^{(1)\nicefrac{-1}{2}}_2\tilde{t}^{(1)\nicefrac{-1}{2}}_3,\quad
            t_2=q^{\nicefrac{-1}{6}}\tilde{t}^{(2)\nicefrac{-1}{2}}_2\tilde{t}^{(1)\nicefrac{-1}{2}}_3,\quad
            t_3=q^{\nicefrac{-1}{6}}\tilde{t}^{(1)\nicefrac{1}{2}}_2\tilde{t}^{(2)\nicefrac{1}{2}}_2\tilde{t}^{(1)\nicefrac{-1}{2}}_3,\quad
            t_4=\tilde{t}^{(1)\nicefrac{1}{2}}_3\tilde{t}^{(2)}_3.        
    \end{equation}
    This specialization is deeply connected with monodromy theory, see \Cref{app:analysis}.
\end{remark}

\section{Higher Teichm\"uller theory}\label{sec:HTT}

In \cite{Chekhov2017}, the concept of bordered cusps arises naturally from the process where two boundary components (or the sides of one boundary component) collide. As a result,  an infinitely thin and infinitely long (in metric sense) domain, called chewing-gum, is obtained. Upon taking the limit of its length to infinity, the chewing-gum breaks into two bordered cusps in the resulting Riemann surface. Closed geodesics that were passing along the chewing gum become arcs, namely infinitely long geodesics that start and terminate at bordered cusps.

Let us now denote by 
$\Sigma_{g,s,m}$ a genus $g$ topological surface with $m$ bordered cusps on the $s$ boundary components having negative Euler characteristic.

In the absence of bordered cusps, the Teichm\"uller space $\mathcal{T}_{\PSL_2(\real)}(\Sigma_{g,s,0})$, i.e., the moduli space of
complex structures on $\Sigma_{g,s,0}$ modulo diffeomorphisms isotopic to the identity, is identified with the space of discrete faithful representations $\pi_1(\Sigma_{g,s,0})\rightarrow\PSL_2(\real)$ modulo 
conjugation.

For $m\geq 1$, the bordered cusped Teichm\"uller space $\widehat{\mathcal{T}}_{\PSL_2(\real)}(\Sigma_{g,s,m})$ was introduced in \cite{Chekhov2017} as 
\begin{equation}
    \widehat{\mathcal{T}}_{\PSL_2(\real)}(\Sigma_{g,s,m})=\hbox{Hom}'\left(\pi_{\mathfrak a}(\Sigma_{g,s,m}), \PSL_2(\real) \right){\Big\slash\Pi_{i=1}^m B_i},
\end{equation}
where  $\pi_{\mathfrak a}(\Sigma_{g,s,m})$ is the fundamental groupoid of arcs in $\Sigma_{g,s,m}$, $\hbox{Hom}'$ denotes the connected component of the representation space consisting entirely of discrete and faithful representations, and $B_1,\ldots,B_m$ is the choice of a Borel unipotent subgroup in $\PSL_2(\real)$ at each bordered cusp. Notice that an element in $\widehat{\mathcal{T}}_{\PSL_2(\real)}(\Sigma_{g,s,m})$ uniquely fixes a metric of constant negative curvature on $\widehat\Sigma_{g,s,m}$, the surface obtained by cutting tubular neighbourhoods of the boundary and adding ideal triangles at the bordered cusps.

Both the Teichm\"uller space and the bordered cusped Teichm\"uller space admit a \emph{higher} generalization by replacing $\PSL_2(\real)$ with any split semi-simple algebraic group G.
For $\mathrm{G}=\PSL_n(\real)$, we denote the \emph{higher bordered cusped Teichm\"uller space} by 
$\widehat{\mathcal{T}}_{\PSL_n(\real)}(\Sigma_{g,s,m})$.

For every bordered cusp, we can assign a marked point on the boundary to which the bordered cusp belongs. This marked point is ``dual'' to the cusp, namely the bordered cusp projects to a pinning over the edge attached to the marked point in the chosen triangulation \cite{Facciotti2024}. Therefore, a Riemann surface of genus $g$  with $m$ bordered cusps on the $s$ boundary components corresponds to a genus $g$ topological surface with $m$ marked points on the $s$ boundary components. For this reason, we denote them both by $\Sigma_{g,s,m}$.

In \Cref{th:HBCT}, we show the coordinate ring of the higher bordered cusped Teichm\"uller space $\widehat{\mathcal{T}}_{\PSL_n(\real)}(\Sigma_{g,s,m})$ can be identified with the positive component of the \emph{moduli space of pinnings} $\mathcal{P}_{\PGL_n(\real)}(\Sigma_{g,s,m})$ introduced by Goncharov and Shen \cite{Goncharov2022}. The latter is an extension by additional data of the moduli space $\mathcal{X}_{\PGL_n(\real)}(\Sigma_{g,s,m})$  of \emph{framed} $\PGL_n(\real)$-local systems \cite{Fock2009}, i.e., principal $\PGL_n(\real)$-bundles with framed flat connections defined by attaching an invariant flag to each marked point.

\subsection{Combinatorial description of the moduli space of pinnings}\label{suse;cmbT}

In this section, we recall the main ingredients of the combinatorial description of $\mathcal{P}_{\mathrm{G}}(\Sigma)$ and its quantization  
\cite{Goncharov2022}.
We restrict to $\mathrm{G}=\PGL_n(\real)$ and closely follow \cite{Chekhov2022} as well as \cite{Coman2015}: since notations have been tailored to our needs, for the sake of the reader the exposition is self-consistent.

In subsection \ref{suse:triangle}, we  describe the moduli space of framed $\PGL_n(\real)$-local systems
for the disk with three marked points $1,2,3$ on its boundary.
We picture such surface $\Sigma_{0,1,3}$ as the equilateral triangle $\triangle_{123}$ in \Cref{fig:transport} and assign a clockwise orientation.

In subsection \ref{suse:amlgam}, we introduce pinnings on $\triangle_{123}$ and explain how to glue triangles together to form the moduli space $\mathcal{P}$ for any Riemann surface $\Sigma_{g,s,m}$.

\subsubsection{The snake calculus on a triangle}\label{suse:triangle}

For a given $n\in\integer_{>0}$, we cover $\triangle_{123}$ by its unique tessellation of $n^2$ identical equilateral triangular tiles, arranged between upward and downward.
Each vertex of this tessellation is labelled by a triple of non-negative integers $(i,j,k)$ by the minimum number of tiles connecting it to the sides of $\triangle_{123}$: $i$ for side $23$, $j$ for side $31$ and $k$ for side $12$ (\Cref{fig:IBT}).
Since $i+j+k=n$, these triples are called \emph{barycentric coordinates}.
\begin{figure}[!htb]
    \centering
    \includegraphics[scale=1]{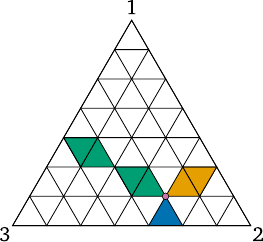}
    \caption{$n=7$ tessellation of $\triangle_{123}$ and barycentric coordinates $({\color{3blue}1},{\color{3green}4},{\color{3ochre}2})$ for the {\color{3pink}pink} vertex, with colors highlighting the tile-counting ruling the coordinatization.  The total $n^2$ tiles are all similar to $\triangle_{123}$ and arranged between $\binom{n+1}{2}$ upward and $\binom{n}{2}$ downward ones.}\label{fig:IBT}
\end{figure}
\\This coordinatization naturally extends to a tile by assigning a triple $(a,b,c)$ to its \emph{center}: 
\begin{itemize}
    \item $a+b+c=n-1$ in the upward case, where vertices appear in the form
    $$\big\{(a+1,b,c),\,(a,b+1,c),\,(a,b,c+1)\big\}$$
    \item $a+b+c=n-2$ in the downward case, where vertices appear in the form
    $$\big\{(a,b+1,c+1),\,(a+1,b,c+1),\,(a+1,b+1,c)\big\}$$
\end{itemize}
\begin{remark}
    Let us highlight the resulting combinatorics: barycentric coordinates are assigned to vertices of the tessellation and centers of the tiles so that the type of object they label can be detected by just inspecting the (integral) result of their sum.
\end{remark}  

Since any flat connection on the contractible $\triangle_{123}$ is trivial, $\mathcal{X}_{\PGL_n(\real)}(\triangle_{123})$ is identified with the space of triples of holonomy-invariant complete flags in $\real^n$.
Snake calculus is a way to construct elementary change-of-basis matrices between projective bases of $\real^n$ induced by a choice of flags in generic position.
Let us detail the combinatorial features of this construction.
\begin{definition}
    A \emph{complete flag} $F_\bullet$ in a vector space $V$ is a collection of consecutively embedded subspaces
    \begin{equation*}
       \{0=F_0 \subset F_1 \subset \ldots \subset F_{n-1} \subset F_n=V\},\quad \dim(F_k)=k.
    \end{equation*}
\end{definition}
Let $F^1_\bullet,F^2_\bullet,F^3_\bullet$ be the (generic) complete flags in $\real^n$ attached to the vertices of $\triangle_{123}$.
To any center $(a,b,c)$ of a tile in the tessellation of $\triangle_{123}$, we attach the subspace $F^1_{n-a}\cap F^2_{n-b}\cap F^3_{n-c}$ : a line $\lambda_{abc}$ for upward tiles and a plane $\pi_{abc}$ for downward ones.
By construction, a plane $\pi_{abc}$ contains the lines $\lambda_{(a+1)bc}, \lambda_{a(b+1)c}, \lambda_{ab(c+1)}$ attached to the three upward tiles adjacent to the downward one it is attached to.
Let us visually highlight this correspondence: after labelling each center with its subspace, we stick on each plane a grey upward triangle whose vertices match the three coplanar lines it contains.
\Cref{fig:config} gives a step-by-step display of the resulting configuration on $\triangle_{123}$.  
\begin{figure}[!htb]
    \centering
    \includegraphics[width=\linewidth]{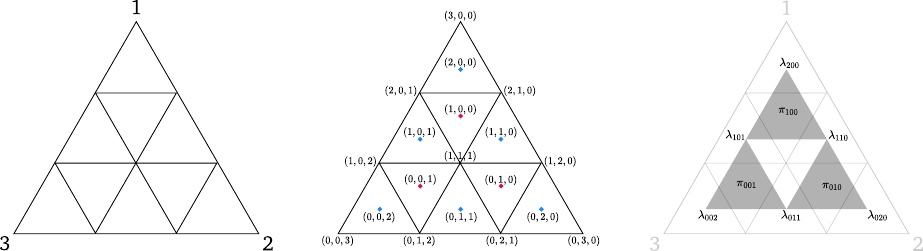}
    \caption{For $n=3$, from left to right: tessellation of $\triangle_{123}$, barycentric coordinates for vertices of the tessellation and centers of the tiles, configuration of subspaces with the grey triangles (and one white triangle surrounded by them).}\label{fig:config}
\end{figure}

For the rest of this section, we forget the tessellation focusing on these $\binom{n}{2}$ grey triangles---and the resulting $\binom{n-1}{2}$ white downward ones among them---looking at specific paths called snakes that run over their sides. Notice that the upward grey and downward white triangles gives precisely the $n-1$ tessellation of a triangle connecting $\{\lambda_{(n-1)00},\lambda_{0(n-1)0},\lambda_{00(n-1)}\}$.
\begin{definition}
    A \emph{snake} $\mathbf{p}$ is an oriented piece-wise path composed by exactly $n-1$ sides of grey triangles, which starts from a tile sharing a vertex with $\triangle_{123}$ and ends on a tile in contact with the opposite side.
\end{definition}
Notice that the length requirement implies no segment can be parallel to the snake's target side of $\triangle_{123}$. 
We call $\mathbf{p}_{IJ}$, the unique snake running parallel to side $IJ$ of $\triangle_{123}$, a $\partial$-snake.
\begin{figure}[!htb]
    \centering
    \includegraphics[width=65mm]{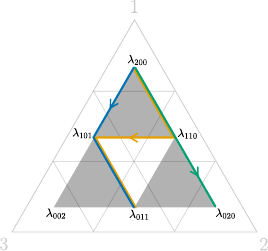}
    \caption{For $n=3$, two snakes and a forbidden {\color{3ochre}ochre} path. The {\color{3green}green} snake $\mathbf{p}_{12}$ has basis $\{\mathbf{v}_{200},\mathbf{v}_{110},\mathbf{v}_{020}\}$ in $\real^3$, whose vectors satisfy $\mathbf{v}_{101}=\mathbf{v}_{110}+\mathbf{v}_{200}, \mathbf{v}_{011}=\mathbf{v}_{020}+\mathbf{v}_{110}$. The {\color{3blue}blue} snake has basis $\{\mathbf{v}_{200},\mathbf{v}_{101},\mathbf{v}_{011}\}$, with $\mathbf{v}_{110}=\mathbf{v}_{101}-\mathbf{v}_{200}, \mathbf{v}_{002}=\mathbf{v}_{011}+\mathbf{v}_{101}$.}\label{fig:snakes}
\end{figure}

Let Greek letters denote a generic triple of barycentric coordinates: e.g., $\lambda_{ijk}$ is equally denoted by $\lambda_\alpha$.
As shown in \Cref{fig:2projbasis}, each segment of a snake connects two vertices $\alpha, \beta$ of a grey triangle.
The corresponding lines  $\lambda_\alpha,\lambda_\beta$ are coplanar to $\lambda_\gamma$, where $\gamma$ is the remaining vertex of the grey triangle.
By coplanarity, a choice of vector $\mathbf{v}_\alpha\in \lambda_\alpha$ uniquely determines $\mathbf{v}_\beta\in\lambda_\beta$ by the following orientation rule
\begin{equation}\label{rule}
    \lambda_\gamma\ni\mathbf{v}_\gamma=
    \begin{cases}
        \mathbf{v}_\beta+\mathbf{v}_\alpha,\quad\circlearrowright\\
        \mathbf{v}_\beta-\mathbf{v}_\alpha,\quad\circlearrowleft.
    \end{cases}
\end{equation}
\begin{figure}[!htb]
    \centering
    \includegraphics[width=90mm]{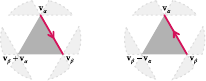}
    \caption{Segments of two oppositely oriented snakes. The
vertices of the grey triangle correspond to 3 coplanar lines $\lambda_\alpha,\lambda_\beta,\lambda_\gamma$ and $\mathbf{v}_\gamma=\mathbf{v}_\beta\pm\mathbf{v}_\alpha$ depending on whether the segment
is oriented clockwise or counterclockwise with respect to its grey triangle.}\label{fig:2projbasis}
\end{figure}

Therefore, a snake inductively determines a projective basis of $\real^n$: chosen the first vector and iteratively applying the rule, the resulting $n$ vectors are defined up to a global scaling factor. Their linear independence is a consequence of the flags being assumed generic.
Given any two snakes, a change-of-basis matrix maps between their corresponding projective bases.
The snake calculus gives a simple recipe to write down these matrices: since the elementary moves in \Cref{fig:snakemoves} suffice to decompose any snake transformation, they are constructed out of the elementary building blocks in the following
\begin{figure}[!htb]
    \centering
    {\includegraphics[width=\textwidth]{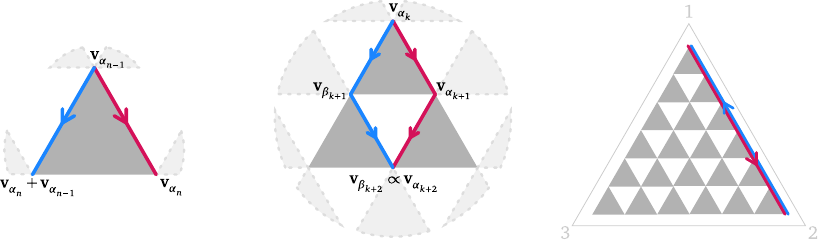}}
    \caption{From left to right, elementary snake moves \MakeUppercase{\romannumeral 1},\MakeUppercase{\romannumeral 2} and \MakeUppercase{\romannumeral 3} mapping {\color{2red}red} to {\color{2blue}blue} segments of a sample snake with $\mathbf{v}_1\in\lambda_{n00}$. Notice that move \MakeUppercase{\romannumeral 1} can only be performed on the last segment of a snake, i.e. when no subsequent segments can be affected. In this sense, move \MakeUppercase{\romannumeral 2} can be thought of as the extension of move \MakeUppercase{\romannumeral 1} to any other segment.}\label{fig:snakemoves}
\end{figure}
\begin{definition}\label{definition:blocks}
    Let $E_{rs}$ be the matrix unit, i.e., $(E_{rs})_{ij}=\delta_{ri}\delta_{sj}$.
    For $\One$ denoting the identity matrix, $k\in\{1,\ldots,n\}$ and a parameter $t\in\real_{>0}$\,, 
    define the $\SL_n(\real_{>0})$ matrices 
    \begin{align}
        L_k&=\One+E_{k+1,k},\\   
        H_k(t)&=t^{-\frac{n-k}{n}}\mathrm{diag}\big(\underbrace{1,\ldots,1}_\text{$k$ times},t,\ldots,t\big),
    \end{align}
    and the $\SL_n(\real)$ antidiagonal matrix
    \begin{equation}
        (S)_{ij}=(-1)^{n-i}\delta_{i,n+1-j}.
    \end{equation}
\end{definition}
Let us sketch the origin of this advantageous feature, adapting Appendix A in \cite{Coman2015}.
Move \MakeUppercase{\romannumeral 1} flips the last segment pivoting its source center across a grey triangle, by rule \eqref{rule} yielding:
\begin{equation}
    \begin{bmatrix}
        \mathbf{v}_{\alpha_1}\\\vdots\\\mathbf{v}_{\alpha_{n-1}}\\\mathbf{v}_{\alpha_n}
    \end{bmatrix}
    \mapsto
    \begin{bmatrix}
        \mathbf{v}_{\alpha_1}\\\vdots\\\mathbf{v}_{\alpha_{n-1}}\\\mathbf{v}_{\alpha_n}+\mathbf{v}_{\alpha_{n-1}}
    \end{bmatrix}
    =\underbrace{\left(
\begin{array}{cc}
\One_{n-2} &  \\
 & \begin{matrix}
        1 & 0 \\
        1 & 1
    \end{matrix}
\end{array}\right)}_{L_{n-1}}
    \begin{bmatrix}
        \mathbf{v}_{\alpha_1}\\\vdots\\\mathbf{v}_{\alpha_{n-1}}\\\mathbf{v}_{\alpha_n}
    \end{bmatrix}.
\end{equation}
Move \MakeUppercase{\romannumeral 2} flips any two non parallel consecutive segments.
Analogously to move \MakeUppercase{\romannumeral 1}, sweeping the grey triangle yields  $\mathbf{v}_{\alpha_{k+1}}\mapsto\mathbf{v}_{\beta_{k+1}}=\mathbf{v}_{\alpha_{k+1}}+\mathbf{v}_{\alpha_{k}}$. However, this drags the second segment in a flip that pivots its target center: we expect the transformed $2$-segment  portion of the snake to end on a different vector within the same line, i.e. $\mathbf{v}_{\beta_{k+2}}\propto\mathbf{v}_{\alpha_{k+2}}$. Denoting by $Z$ the proportionality constant, restricted to be strictly positive (see \Cref{rmk:positive}), the full move reads as
\begin{equation}\label{moveII}
    \begin{bmatrix}
    \mathbf{v}_{\alpha_1}\\\vdots\\
    \mathbf{v}_{\alpha_{k}} \\
\mathbf{v}_{\alpha_{k+1}} \\ \mathbf{v}_{\alpha_{k+2}}\\
    \vdots\\
    \mathbf{v}_{\alpha_n}
    \end{bmatrix}
    \mapsto
    \begin{bmatrix}\mathbf{v}_{\alpha_1}\\\vdots\\ \mathbf{v}_{\alpha_{k}} \\ \mathbf{v}_{\beta_{k+1}} \\ \mathbf{v}_{\beta_{k+2}}\\
    \vdots\\
    \mathbf{v}_{\alpha_n}
    \end{bmatrix}
    =\left(\begin{array}{ccc}
\One_{k-1} & & \\
 & \begin{matrix}
    1 & 0 & 0\\
    1 & 1 & 0\\
    0 & 0 & Z
  \end{matrix} &  \\ 
   &  & Z\ \One_{n-k-2}\end{array}\right)
    \begin{bmatrix}
    \mathbf{v}_{\alpha_1}\\\vdots\\
    \mathbf{v}_{\alpha_{k}} \\
\mathbf{v}_{\alpha_{k+1}} \\ \mathbf{v}_{\alpha_{k+2}}\\
    \vdots\\
    \mathbf{v}_{\alpha_n}
    \end{bmatrix}.
\end{equation}
Since the elementary blocks $L_i$ and $H_j(t)$ commute for $i\neq j$, this change-of-basis matrix can be factorized, inside $\PGL_n(\real)$, as $L_kH_{k+1}(Z)$ and the move as a whole is well-defined. 
Finally, move \MakeUppercase{\romannumeral 3}, inverting a 
 clockwise oriented $\partial$-snake, is unravelled tracking segment reversals:
\begin{equation}
    \begin{bmatrix}
        \mathbf{v}_{\alpha_1}\\\mathbf{v}_{\alpha_2}\\\vdots\\\mathbf{v}_{\alpha_{n-1}}\\\mathbf{v}_{\alpha_n}
    \end{bmatrix}
    \mapsto
    \begin{bmatrix}
        \mathbf{v}_{\alpha_n}\\-\mathbf{v}_{\alpha_{n-1}}\\\vdots\\(-1)^{n-2}\mathbf{v}_{\alpha_2}\\(-1)^{n-1}\mathbf{v}_{\alpha_1}
    \end{bmatrix}
    =S^{-1}
    \begin{bmatrix}
        \mathbf{v}_{\alpha_1}\\\vdots\\\mathbf{v}_{\alpha_{n-1}}\\\mathbf{v}_{\alpha_n}
    \end{bmatrix}.
\end{equation}
Notice that $S^{-1}=(-1)^{n-1}S=S^T$.

There are $\binom{n-1}{2}$ type \MakeUppercase{\romannumeral 2} moves, one for each downward white triangle, and the corresponding proportionality constants are the so-called (positive) Fock-Goncharov variables.
Topologically, notice that Fock-Goncharov variables are in bijection with inner vertices of the tessellation of $\triangle_{123}$: there is exactly one such vertex inside any white triangle. We thus denote them $Z_{ijk}$ by the barycentric coordinates of the unique corresponding vertex, $i,j,k\in\integer_{>0}$.

Taking advantage of this calculus, the general formula for the change-of-basis matrix corresponding to the $\partial$-snake map $\mathbf{p}_{12}\mapsto\mathbf{p}_{31}$ reads
\begin{equation}\label{coarseT1}
    \PSL_n(\real)\ni C_{12\rightarrow31}=S\ L_{n-1}\prod_{j=1}^{n-2}\bigg[\prod_{i=1}^j\big[L_{n-i-1}H_{n-i}(Z_{n-j-i,i,j})\big]L_{n-1}\bigg].
\end{equation}
\begin{example}
For $n=2$, there are no inner vertices and formula 
\eqref{coarseT1} simplifies to
\begin{equation}
    C_{12\rightarrow31}=S\ L_{1}=\begin{pmatrix}
        0 & -1\\1 & 0
    \end{pmatrix}\begin{pmatrix}
        1 & 0\\1 & 1
    \end{pmatrix}.
\end{equation}

For $n=3$, there is just a single Fock-Goncharov variable $Z_{111}$:
\begin{equation}
\begin{aligned}
    C_{12\rightarrow31}&=S\ L_{2} L_{1} H_2(Z_{111})L_{2}\\
    &=\begin{pmatrix}
        0 & 0 & 1\\0 & -1 & 0\\1 & 0 & 0
    \end{pmatrix}\begin{pmatrix}
        1 & 0 & 0\\0 & 1 & 0\\0 & 1 & 1
    \end{pmatrix}\begin{pmatrix}
        1 & 0 & 0\\1 & 1 & 0\\0 & 0 & 1
    \end{pmatrix}\begin{pmatrix}
        Z_{111}^{\nicefrac{-1}{3}} & 0 & 0\\0 & Z_{111}^{\nicefrac{-1}{3}} & 0\\0 & 0 & Z_{111}^{\nicefrac{2}{3}}
    \end{pmatrix}\begin{pmatrix}
        1 & 0 & 0\\0 & 1 & 0\\0 & 1 & 1
    \end{pmatrix}.
\end{aligned}
\end{equation}
\begin{figure}[!htb]
    \centering
    \includegraphics[width=\linewidth]{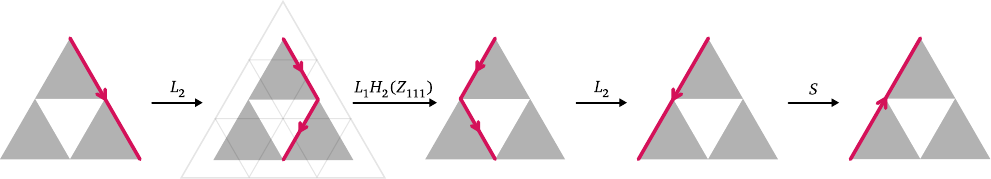}
    \caption{Sequence of snake moves factorizing $C_{12\rightarrow31}$ for $n=3$. At step 2, the tessellation's only inner vetrex of barycentric coordinates $(1,1,1)$ labels the Fock-Goncharov variable $Z_{111}$. At step 4, the $\partial$-snake runs counterclockwise and an $S$ matrix is needed.}\label{fig:T1n2}
\end{figure}
\end{example}

\subsubsection{Transport matrices and amalgamation}\label{suse:amlgam}
In order to glue triangles together, we need to attach additional variables to the sides of $\triangle_{123}$.
This is formally done by extending $\mathcal{X}_{\PGL_n(\real)}(\triangle_{123})$ to the \emph{moduli space of pinnings} $\mathcal{P}_{\PGL_n(\real)}(\triangle_{123})$, in which each oriented side of the triangle comes equipped with a $1$-dimensional subspace of $\real^n$ in generic position to the corresponding pair of flags.

A \emph{pinning} on side $IJ$ is given by the triple $\{F^I_\bullet,F^J_\bullet,\Lambda\}$ with $\Lambda\subset\real^n,\mathrm{dim}(\Lambda)=1$. 
A choice of $\Lambda$ equals a choice of projective basis $\{\mathbf{v}_{\alpha_1},\ldots\mathbf{v}_{\alpha_n}\}$ in $\real^n$, a vector for each vertex along $IJ$ from the corresponding line, via the condition $\sum_{i=1}^n\mathbf{v}_{\alpha_i}\in \Lambda$.
Therefore, each oriented side $IJ$ comes with two projective bases, one from the pinning and the other from the corresponding $\partial$-snake $\mathbf{p}_{IJ}$, and the unimodular change-of-basis matrix between them takes the form $\prod_{i=1}^{n-1}H_i(t_i)$.
These $n-1$ proportionality constants are thought of as additional Fock-Goncharov variables $Z_{ijk}$, labelled by the vertices on the interior of $IJ$.
Adding these extra variables from all three sides to the ones birthed by type \MakeUppercase{\romannumeral 2} moves, we get a total of
$3(n-1)+\binom{n-1}{2}=\frac{(n+4)(n-1)}{2}$ Fock-Goncharov variables (\Cref{fig:FG}).
As a whole, they are in bijection with the tessellation's vertices except $1,2,3$.
\begin{remark}\label{rmk:positive}
    We impose all Fock-Goncharov variables to be strictly positive: this restriction is known \cite{Chekhov2022} to provide a parametrization of the moduli space describing its positive connected component $\mathcal{P}^+_{\PGL_n(\real)}(\Sigma)$. In particular, this choice makes the transport matrices \eqref{T-formulae} genuine elements of $\PSL_n(\real)\subseteq\PGL_n(\real)$.
\end{remark}
\begin{figure}[!htb]
    \centering
    \includegraphics[width=.9\textwidth]{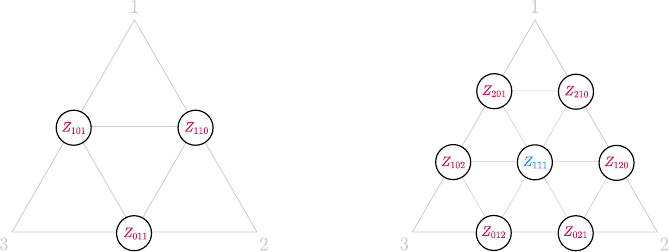}
    \caption{Fock-Goncharov variables $Z_\alpha$ for $\mathcal{P}_{\PGL_2(\real)}(\triangle_{123})$ on the left and $\mathcal{P}_{\PGL_3(\real)}(\triangle_{123})$ on the right. {\color{2blue}Blue} variables are associated with moves \MakeUppercase{\romannumeral 2} and {\color{2red}red} ones with side pinnings.}\label{fig:FG}
\end{figure}

We finally define the triple of {\em transport matrices}  $T_i$ in \Cref{fig:transport}.
They correspond to the special change-of-basis matrices between the pinning-induced projective bases associated to the oriented sides of $\triangle_{123}$. For example, $T_1$ maps the pinning of side $12$ first to the  snake $\mathbf{p}_{12}$, then maps the snake $\mathbf{p}_{12}$ to the snake $\mathbf{p}_{31}$, and finally maps the snake $\mathbf{p}_{31}$ to the pinning of side $31$.
\begin{figure}[!htb]
    \centering
    \includegraphics[scale=2]{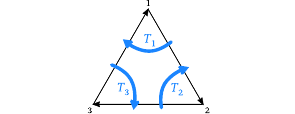}
    \caption{The triple of transport matrices on the oriented triangle $\triangle_{123}$. $T_1$ corresponds to the map of oriented sides $12\mapsto31$, $T_2$ to $23\mapsto12$ and $T_3$ to $31\mapsto23$.}\label{fig:transport}
\end{figure}
\begin{definition}
The transport matrices $T_1,T_2,T_3$ are the following $\PSL_n(\real)$ matrices
\begin{equation*}
    T_1
        =S\prod_{k=1}^{n-1}\Big[H_{n-k}(Z_{k,0,n-k})\Big]\ \ L_{n-1}\prod_{j=1}^{n-2}\Bigg[\prod_{i=1}^j\Big[L_{n-i-1}H_{n-i}(Z_{n-j-i,i,j})\Big]L_{n-1}\Bigg]\ \ \prod_{k=1}^{n-1}H_k(Z_{n-k,k,0}),
\end{equation*}
\begin{equation}\label{T-formulae}
    \begin{aligned}
    & T_2
        =S\prod_{k=1}^{n-1}\Big[H_{n-k}(Z_{n-k,k,0})\Big]\ \ L_{n-1}\prod_{j=1}^{n-2}\Bigg[\prod_{i=1}^j\Big[L_{n-i-1}H_{n-i}(Z_{j,n-j-i,i})\Big]L_{n-1}\Bigg]\ \ \prod_{k=1}^{n-1}H_k(Z_{0,n-k,k}),\\
    & T_3=S\prod_{k=1}^{n-1}\Big[H_{n-k}(Z_{0,n-k,k})\Big]\ \ L_{n-1}\prod_{j=1}^{n-2}\Bigg[\prod_{i=1}^j\Big[L_{n-i-1}H_{n-i}(Z_{i,j,n-j-i})\Big]L_{n-1}\Bigg]\ \ \prod_{k=1}^{n-1}H_k(Z_{k,0,n-k}).
    \end{aligned}
    \end{equation}
\end{definition}

As expected, a direct computation confirms that $T_1T_2T_3=\One$.
Together with their inverses, $T_1, T_2$ and $T_3$ suffice to map between any two sides.
Notice that, for the permutation map $\sigma$ acting on matrices $T(Z_{ijk})$ depending on Fock-Goncharov variables $Z_{ijk}$ as
\begin{equation}
    \sigma T(Z_{ijk}):=T(Z_{jki}),
\end{equation}
we have $T_2=\sigma T_1$ and $T_3=\sigma^2 T_1$.

We introduce the following shorthand notation:
$$
T_1=H^{31}_{out}\ C_{12\rightarrow31}\ H^{12}_{in}
$$
where
\begin{align}
        &H^{12}_{in}:=\prod_{k=1}^{n-1}H_k(Z_{n-k,k,0}),\\
        &H^{31}_{out}:=(H^{31}_{in})^{-1}=\prod_{k=1}^{n-1}\Big[H^{-1}_k(Z_{k,0,n-k})\Big]=S\prod_{k=1}^{n-1}\Big[H_{n-k}(Z_{k,0,n-k})\Big]S^{-1}.
    \end{align}
\begin{remark}
    These diagonal factors modifying the change-of-basis matrix can be visualized as passing from the side's pinning to the inner $\partial$-snake and vice versa: $H^{12}_{in}$ for the oriented side $12$ the path crosses to enter the triangle (pinning-to-snake), $H^{31}_{out}$ for the oriented side of exit (snake-to-pinning). 
\end{remark}
\begin{example}\label{ex-tr}
    Explicitly, for $n=2$ and $n=3$ we get
    \begin{equation}\label{2dimT1}
        T_1=S\ H_1(Z_{101})L_{1}H_1(Z_{110})=\setlength\arraycolsep{3pt}\begin{pmatrix}
            -Z_{101}^{\nicefrac{1}{2}}Z_{110}^{\nicefrac{-1}{2}} & -Z_{101}^{\nicefrac{1}{2}}Z_{110}^{\nicefrac{1}{2}}\\[.5em]
            Z_{101}^{\nicefrac{-1}{2}}Z_{110}^{\nicefrac{-1}{2}} & 0
        \end{pmatrix},   
    \end{equation}
    \begin{equation}
        \begin{aligned}\label{3dimT1}
            T_1&=S\ H_2(Z_{102})H_1(Z_{201})L_{2} L_{1} H_2(Z_{111})L_{2}H_1(Z_{210})H_2(Z_{120})\\
            &=\setlength\arraycolsep{3pt}\begin{pmatrix}
                Z_{102}^{\nicefrac{2}{3}}Z_{111}^{\nicefrac{-1}{3}}Z_{120}^{\nicefrac{1}{3}}Z_{201}^{\nicefrac{1}{3}}Z_{210}^{\nicefrac{-2}{3}} & Z_{102}^{\nicefrac{2}{3}}(Z_{111}^{\nicefrac{-1}{3}}+Z_{111}^{\nicefrac{2}{3}})Z_{120}^{\nicefrac{-1}{3}}Z_{201}^{\nicefrac{1}{3}}Z_{210}^{\nicefrac{1}{3}} & Z_{102}^{\nicefrac{2}{3}}Z_{111}^{\nicefrac{2}{3}}Z_{120}^{\nicefrac{2}{3}}Z_{201}^{\nicefrac{1}{3}}Z_{210}^{\nicefrac{1}{3}} \\[.5em]
                -Z_{102}^{\nicefrac{-1}{3}}Z_{111}^{\nicefrac{-1}{3}}Z_{120}^{\nicefrac{-1}{3}}Z_{201}^{\nicefrac{1}{3}}Z_{210}^{\nicefrac{-2}{3}} & -Z_{102}^{\nicefrac{-1}{3}}Z_{111}^{\nicefrac{-1}{3}}Z_{120}^{\nicefrac{-1}{3}}Z_{201}^{\nicefrac{1}{3}}Z_{210}^{\nicefrac{1}{3}} & 0 \\[.5em]
                Z_{102}^{\nicefrac{-1}{3}}Z_{111}^{\nicefrac{-1}{3}}Z_{120}^{\nicefrac{-1}{3}}Z_{201}^{\nicefrac{-2}{3}}Z_{210}^{\nicefrac{-2}{3}} & 0 & 0            
            \end{pmatrix}
        \end{aligned}       
    \end{equation}
    Notice that, in both cases, no Fock-Goncharov variables from side $23$ appear, in accordance with the crossing of $\triangle_{123}$ associated with $T_1$.
    Cyclicly permute the indices once and twice to get the expressions for $T_2$ and $T_3$.
\end{example}
As anticipated, pinnings allow to \emph{amalgamate} variables of two adjacent triangles, creating the set of parameters describing the moduli
space $\mathcal{P}_{\PGL_n(\real)}(\square)$ of the quadrangle obtained by gluing the pair along the
common side.
The amalgamation procedure orderly identifies the two $(n-1)$-tuples of vertices on the interior of the sides to be glued, assigning to each resulting vertex a Fock-Goncharov amalgamated variable via the product of the parent ones: if the identified vertices $\alpha_1,\alpha_2$ result in the single vertex $\alpha$, $Z_\alpha:=Z_{\alpha_1}Z_{\alpha_2}$ (\Cref{fig:amalg}).
\begin{figure}[!htb]
    \centering
    \includegraphics[scale=1.3]{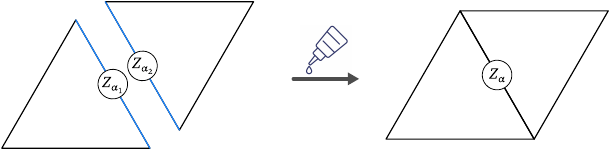}
    \caption{Amalgamation example $Z_\alpha=Z_{\alpha_1}Z_{\alpha_2}$, with the sides to be glued in {\color{2blue}blue}. }\label{fig:amalg}
\end{figure}
This operation allows to parametrize $\mathcal{P}_{\PGL_n(\real)}(\mathbb{S})$, for any (suitable) triangulated surface $\mathbb{S}$, by amalgamation of the moduli spaces of pinnings assigned to the individual triangles.

\subsection{Fat graph loops via transport matrices}\label{sec:looprep}

Using the machinery of \Cref{suse;cmbT},  we here explain how to assign sequences of transport matrices to loops running over a fat graph. 
In order to explain this transport matrix factorization of fat graph paths, we assume a clockwise labelling of vertices from the set $\{1,2,3\}$ is chosen for each triangle coming from the dual fat graph $\Gamma^{\vee}_{g,s,m}$ of the surface $\Sigma_{g,s,m}$, where $\Gamma_{g,s,m}$ is constructed following the recipe in \cite{Chekhov2017}.

One should picture a path as transporting an oriented side along the triangulation by the action of transport matrices.
In order to consistently compose two transport matrices---that we defined as maps between the \emph{clockwise} oriented sides of a 
triangle---a reversal of the transported side must be performed.
This is done by inserting an $S$ block between the matrices.
\begin{figure}[!htb]
    \centering
    \includegraphics[scale=1.35]{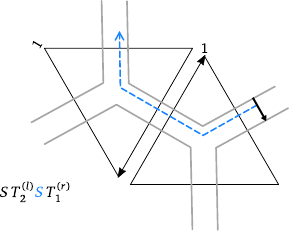}
    \caption{The {\color{2blue}blue} path transports the oriented side indicated by the thick black arrow. When constructing its representation, the transport composition rule prescribes the insertion of a {\color{2blue}S} block between the transport matrices. The orientation is explicitly indicated on those sides interacting via the composition: the $31$ oriented side of the right $(r)$ triangle is reversed by $S$ to match the $23$ one of the left $(l)$ triangle. The leftmost $S$ block performs a final reversal: without it, the path would flip the side it is just meant to transport.}\label{fig:gluing}
\end{figure}
In terms of Fock-Goncharov variables, this transport composition rule  enforces the amalgamation, performed in the language of transport matrices by letting the diagonal pinning factors multiply each other.
Notice that the side reversal is exactly the operation needed to absorb the leftmost $S$ factor of a transport matrix and let the $H$ blocks generate the amalgamated variable.
\begin{example}
    For $n=2$, the composition in \Cref{fig:gluing} is the amalgamation $Z':=Z^{(l)}_{011}Z^{(r)}_{101}$:
    \begin{equation*}
    \begin{aligned}
        T^{(l)}_2 \ S \ T^{(r)}_1 &= SH_1(Z^{(l)}_{110})L_{1}H_1(Z^{(l)}_{011}) \ S \ SH_1(Z^{(r)}_{101})L_{1}H_1(Z^{(r)}_{110})\\
        &=-S\ H_1(Z^{(l)}_{110})L_{1}\underbrace{H_1(Z^{(l)}_{011})H_1(Z^{(r)}_{101})}_{H_1(Z^{(l)}_{011}Z^{(r)}_{101})}L_{1}H_1(Z^{(r)}_{110})\\
        &=-S\ H_1(Z^{(l)}_{110})L_{1}H_1(Z')L_{1}H_1(Z^{(r)}_{110}),
    \end{aligned}
    \end{equation*}
    i.e. $H_1(Z_{\alpha_2})H_1(Z_{\alpha_1})=H_1(Z_{\alpha_1} Z_{\alpha_2})=H_1(Z_\alpha)$, for $Z_\alpha$ the amalgamated variable.   
    For $n=3$, the mechanism reads
    \begin{align*}    H_1(Z_{\beta_2})H_2(Z_{\alpha_2})H_2(Z_{\alpha_1})H_1(Z_{\beta_1})&=H_1(Z_{\beta_2}Z_{\beta_1})H_2(Z_{\alpha_2}Z_{\alpha_1})\\
&=H_2(Z_{\alpha_2}Z_{\alpha_1})H_1(Z_{\beta_2}Z_{\beta_1})=H_2(Z_\alpha)H_1(Z_\beta),
    \end{align*}
    with $Z_\alpha,Z_\beta$ as the two amalgamated variables.
\end{example}
\begin{remark}\label{rmk:conjugation}
    When dealing with loops, the amalgamation at the base-point cannot be captured by the mere factorization over transport matrices: the unavoidable choice of a starting point prevents the composition between the first and last transport matrices from happening.
    
    Nevertheless, this issue is easily fixed by a \emph{global} conjugation. E.g., the path in \Cref{fig:gluing} can be closed into a loop conjugating its factorization $ST_2^{(l)}ST_1^{(r)}$ by $C=H_1(Z^{(r)}_{110})$: denoting by $Z'':=Z^{(r)}_{110}Z^{(l)}_{110}$ the new amalgamated variable due to the closure,
    \[ CST^{(l)}_2ST^{(r)}_1C^{-1} =H_1(Z^{(r)}_{110})H_1(Z^{(l)}_{110})L_{1}H_1(Z')L_{1}H_1(Z^{(r)}_{110})H^{-1}_1(Z^{(r)}_{110})=H_1(Z'')L_{1}H_1(Z')L_{1}. \]
Notice that in this last factorization none of the variables forming $Z''$ remains.
\end{remark}
Summing up, once the sequence of transport matrices associated to the directed crossings of triangles is read off, each loop's matrix is assembled by transport compositions and finalized by a global conjugation.

\subsection{Higher bordered cusped Teichm\"uller space}\label{suse:bordered}

In this section, following the work by the second author and her collaborators \cite{Chekhov2017,Chekhov2018}, we recall the definition of the higher bordered cusped Teichm\"uller space and show that it can be identified with the positive connected component $\mathcal{P}^+_{\PGL_n(\real)}(\Sigma_{g,s,m})$ of the moduli space of pinnings.

Given a Riemann surface $\Sigma_{g,s,m}$ of genus $g$ with $s>0$ boundaries and $m\ge 0$ bordered cusps on the boundary, denote by $\widehat\Sigma_{g,s,m}$ the surface resulting by removing tubular domains around each boundary of $\Sigma_{g,s,m}$ and adding an ideal triangle at each bordered cusps. The Riemann surface $\Sigma_{g,s,m}$ is called  \emph{hyperbolic} if the hyperbolic area of $\widehat\Sigma_{g,s,m}$,
\begin{equation}
    \hbox{Area\ of\ }\widehat\Sigma_{g,s,m}=\bigl(4g-4+2s+m\bigr)\pi,
\end{equation}
is strictly positive.
For a hyperbolic Riemann surface $\Sigma_{g,s,m}$, denote by   $p_1,\dots,p_m$ the bordered cusps on the boundary.
The fundamental groupoid of arcs $\pi_{\mathfrak a}(\Sigma_{g,s,m})$ is the set of all directed paths $\gamma_{ij}:[0,1]\to  \Sigma_{g,s,m}$ such that $\gamma_{ij}(0)=p_i$ and $\gamma_{ij}(1)=p_j$ modulo homotopy. The groupoid structure is dictated by the usual path--composition rules.  
The higher bordered cusped Teichm\"uller space is defined as
\begin{equation}
    \widehat{\mathcal T}_{\PSL_n(\real)}(\Sigma_{g,s,m})=\hbox{Hom}'\left(\pi_{\mathfrak a}(\Sigma_{g,s,m}), \PSL_n(\real) \right){\Big\slash\Pi_{i=1}^m B_i},
\end{equation}
where $\hbox{Hom}'$ is a connected component of the representation space consisting entirely of discrete faithful representations and $B_1,\dots,B_m$ is the choice of an unipotent Borel element in $ \PSL_n(\real)$ for every bordered cusp.
The representation space $\hbox{Hom}'\big(\pi_{\mathfrak a}(\Sigma_{g,s,m}), \PSL_n(\real) \big)$ is endowed with the Fock-Rosly pre-Poisson bracket, which is compatible with the cluster variety Poisson structure \cite{Chekhov2022}.
\begin{theorem}\label{th:HBCT}
The higher bordered cusped Teichm\"uller space is isomorphic to the positive component of the
moduli space of pinnings:
$$
\widehat{\mathcal T}_{\PSL_n(\real)}(\Sigma_{g,s,m})\simeq \mathcal P^+_{\PGL_n(\real)}(\Sigma_{g,s,m}).
$$
The Fock-Goncharov coordinates define the coordinate ring on $\widehat{\mathcal T}_{\PSL_n(\real)}(\Sigma_{g,s,m})$.
\end{theorem}
 \begin{proof}
     For a hyperbolic Riemann surface $\Sigma_{g,s,m}$, $\widehat\Sigma_{g,s,m}$ can be triangulated by $4g-4+2s+n$ ideal triangles. Each  arc $\gamma_{ij}\in\pi_{\mathfrak a}$ uniquely intersects two sides of any triangle it crosses\footnote{In $\widehat{\mathcal T}$, the paths go between bordered cusps---each corresponding to the addition of an extra ideal triangle \cite{Chekhov2017}.
     These additional triangles are not part of or detected by the ideal triangulation: their information remains codified in the form of pinnings.
     Therefore, all arcs do start and finish \emph{outside} the triangles of the ideal triangulation.}.  
As explained in \Cref{suse;cmbT}, a transport matrix $T$ is associated to any such oriented crossing, with $T^{-1}$ corresponding to the opposite direction.
Transport matrices (\Cref{fig:transport}) fulfill the role of building blocks for associating a matrix  $M_{ij}\in \PSL_n(\mathbb R)$ to each arc $\gamma_{ij}\in\pi_{\mathfrak a}(\Sigma_{g,s,m})$: as in \Cref{sec:looprep}, they are suitably composed resulting in a given matrix $M_{ij}$ for any path $\gamma_{ij}$.
The entries of each $M_{ij}$ are expressed in Fock-Goncharov variables.
Therefore, the Fock-Goncharov variables provide the coordinate ring of $\widehat{\mathcal{T}}_{\PSL_n(\real)}(\Sigma_{g,s,m})$ if
 $$
 \dim\left(\mathcal{P}^+_{\PGL_n(\real)}(\Sigma_{g,s,m})\right)=\dim\left(\widehat{\mathcal{T}}_{\PSL_n(\real)}(\Sigma_{g,s,m})\right).
 $$
 
 Let us compute the left-hand side first. As explained above,   $\widehat\Sigma_{g,s,m}$ can be triangulated by $F=4g-4+2s+m$ ideal triangles. Each of these triangles in endowed with $(n-1)(n-2)/2$ inner variables and $n-1$ pinnings for each side, resulting in a total of  $(n+4)(n-1)/2$ Fock-Goncharov  variables. However, most of the sides of these triangles are amalgamated: indeed, the only non-amalgamated sides are the ones dual to an open edge in the cusped fat graph corresponding to a bordered cusp. Therefore, we have $3F-m=2(6g-6+3 s+m)$ amalgamated sides that absorb $(6g-6+3 s+m)(n-1)$ variables, leading to
\begin{equation}
    \dim\left(\mathcal{P}^+_{\PGL_n(\real)}(\Sigma_{g,s,m})\right)=(4g-4+2s+m)(n+4)(n-1)/2-(6g-6+3 s+m)(n-1).
\end{equation}

We now compute the dimension of $\widehat{\mathcal{T}}_{\PSL_n(\real)}(\Sigma_{g,s,m})$. 
The fundamental groupoid of arcs is generated by $2g+s-2+m$ arcs. For each of these, we associate a matrix in $\PSL_n(\real)$ and, taking into account the quotient by the unipotent Borel subgroups, we obtain 
\begin{equation}
    \dim\left(\widehat{\mathcal{T}}_{\PSL_n(\real)}(\Sigma_{g,s,m})\right)=(2g+s-1+m-1)(n^2-1)-m n(n-1)/2.
\end{equation}
The two numbers coincide.
\end{proof}

\subsection{Quantization}\label{sec:quantization}

The moduli space of pinnings $\mathcal{P}_{\PGL_n(\real)}(\triangle_{123})$ is quantized by promoting the Fock-Goncharov variables to invertible coordinates of a quantum torus, with relations encoded by a quiver constructed from the tessellation of $\triangle_{123}$.
\begin{remark}
When quantized, fractional powers of positive Fock–Goncharov variables $Z_\alpha\in\real_{>0}$ can be represented, on a dense Hilbert subspace of $L_2(\real^d)$, as unbounded self-adjoint operators with 
\emph{positive} spectrum \cite{Goncharov2022}.
\end{remark}
Provided the removal of $1,2,3$, the quiver's vertices coincide with the tessellation's ones and arrows are defined by consistently extending the clockwise orientation of $\triangle_{123}$ to the tiles: upward ones are clockwise and downward ones counterclockwise.
Arrows from the sides of $\triangle_{123}$ are dashed.
The resulting quiver is displayed in \Cref{fig:FGquiver}.
\begin{figure}[!htb]
    \centering
    \includegraphics[scale=1.2]{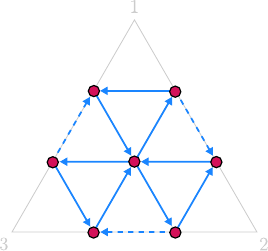}
    \caption{$n=3$ quiver for $\triangle_{123}$.}
    \label{fig:FGquiver}
\end{figure}
The set of vertices of the quiver is in bijection with quantum Fock-Goncharov variables $Z_\alpha$ and arrows rule their commutation relations: for a central invertible variable $q$,
\begin{equation}\label{qcommut}    \begin{matrix}\includegraphics[width=0.465\textwidth]{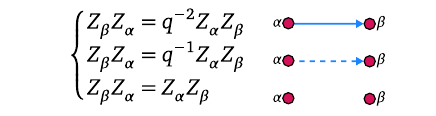}\end{matrix}
\end{equation}
Denoting by $\pazocal{Q}$ the encoding quiver, the resulting noncommutative algebra is a so-called quantum (cluster) $\pazocal{X}$-torus 
\begin{equation}\label{countarrow}
    \pazocal{X}_{\pazocal{Q}}:=\complex[q^{\nicefrac{\pm1}{2}}]\big\langle\{Z_\alpha^{\pm1}\}\big\rangle\big/(Z_\beta Z_\alpha-q^{-2\#}Z_\alpha Z_\beta),
\end{equation}
for $\#$ counting the number of arrows from $\alpha$ to $\beta$ ($\nicefrac{1}{2}$ for a dashed one), which defines via mutations a corresponding quantum cluster $\pazocal{X}$-variety \cite{Fock2009}.
Such an algebra naturally carries the \emph{Weyl quantum ordering}: for any monomial $Z_{\alpha_1}\cdots Z_{\alpha_n}$, we denote it by double bullets
\begin{equation}\label{Weyl}
    \ord{Z_{\alpha_1}\cdots Z_{\alpha_n}}:=q^W Z_{\alpha_1}\cdots Z_{\alpha_n},\quad W:=\sum^n_{\substack{k=2\\j<k}}w_{jk}\text{\ \ for\ \ }Z_{\alpha_j}Z_{\alpha_i}=q^{2w_{ij}}Z_{\alpha_i} Z_{\alpha_j}.
\end{equation}
A handy way to master rule \eqref{Weyl} is to imagine the weight $w_{ij}=-w_{ji}$ measuring a flow carried by the arrows: $-1$ for an outgoing arrow $\alpha_i\rightarrow\alpha_j$ (outflow) and $+1$ for an incoming arrow $\alpha_i\leftarrow\alpha_j$ (inflow), with dashed arrows corresponding to half flows.
Notice that \emph{inside} the double bullets the order does not matter, e.g., $\ord{Z_\alpha Z_\beta}=\ord{Z_\beta Z_\alpha}$.
This is the very reason this definition provides a well-defined ordering for quantum variables.

Due to the normalized $H_k$ block, we need an extension of the $\pazocal{X}$-torus containing $n$-th roots of Fock-Goncharov variables:
\begin{equation}
\pazocal{X}_{\pazocal{Q}}^{\nicefrac{1}{n}}:=\complex[q^{\pm\frac{1}{2n^2}}]\big\langle\{Z_\alpha^{\nicefrac{\pm1}{n}}\}\big\rangle\Big/\big( Z_\beta^{\nicefrac{1}{n}}Z_\alpha^{\nicefrac{1}{n}}-q^{-\frac{2\#}{n^2}}Z_\alpha^{\nicefrac{1}{n}}Z_\beta^{\nicefrac{1}{n}}\big).
\end{equation}
The $n^2$ denominator is found by factorizing each $Z_\alpha$ as $\prod_{i=1}^n Z_\alpha^{\nicefrac{1}{n}}$, while the quantum ordering formula remains valid.

Transport matrices are quantized following a straightforward recipe:
\begin{definition}
For the diagonal matrix $(Q)_{ii}:=q^{2-i-\frac{n+1}{n^2}}$, the triplet of \emph{quantum transport matrices} is given by
\begin{equation}
     T^q_i=Q\ord{T_i}, \quad i=1,2,3,
\end{equation}
where the Weyl quantum ordering acts linearly on each entry.
\end{definition}
The matrix $Q$ is uniquely defined by enforcing the quantum groupoid relation \cite{Chekhov2022} 
\begin{equation}
T^q_1T^q_2T^q_3=\One.
\end{equation}
Within the framework of \Cref{sec:looprep}, this translates to the topological consistency visualized by \Cref{fig:qgroupoid}. 
In the classical case, we have already observed that $T_1T_2T_3=\One$ follows directly.
\begin{figure}[!htb]
    \centering
    \includegraphics[scale=1.5]{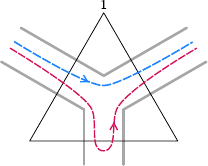}
    \caption{Interpretation of the quantum groupoid relation for paths: being topologically equivalent, {\color{2blue}blue} and {\color{2red}red} must be assigned the same matrix, i.e. ${\color{2blue}(T^{q}_1)^{-1}}={\color{2red}T^q_2T^q_3}$.}\label{fig:qgroupoid}
\end{figure}
Notice that the quantum correction introduced with the matrix $Q$ causes the entries of the quantum transport matrices \emph{not} to be Weyl-ordered monomials.
Moreover, interpreting this correction as the quantization
\begin{equation}
S \mapsto S^q:=QS,
\end{equation}
we quantize the transport composition rule (\Cref{fig:gluing}) by prescribing the insertion of a $S^q$ block instead.

Finally, the quantum extension of the amalgamation procedure is straightforward: under the simplest rule prescribing commutation for quantum Fock-Goncharov variables coming from different triangles, a quantum amalgamated variable reads $Z_\alpha:=\ord{Z_{\alpha_2}Z_{\alpha_1}}=Z_{\alpha_2}Z_{\alpha_1}=Z_{\alpha_1}Z_{\alpha_2}$.
For the special self-gluing case where two sides of the \emph{same} triangle are identified, only the first equality holds and the amalgamated variable must be taken as the quantum ordering.

\section{GDAHAs from higher Teichm\"uller theory}\label{sec:GDAHAreps}

In this section, we consider two specific fat graphs, and construct the matrix algebras resulting from the transport representation. We show that these algebras provide embeddings of universal GDAHAs: the $n=2$ case recovers the known representation of $\mathsf{H}_{D_4}$ \cite{Mazzocco2018}, serving as both a quantum showcase of the machinery developed in \Cref{sec:looprep} and an appetizer for the more involved $n=3$ one corresponding to $\mathsf{H}_{E_6}$.
\begin{remark}
    Since we want to deal with whole matrices and not only their traces, namely quantum objects in the representation space instead of the character variety, we select a base point in each fat graph. This corresponds to selecting an edge in the triangulation that is transported along the paths. 
\end{remark}
Our proofs are supported by the \texttt{NCAlgebra} extension for Mathematica \cite{NCA}.
This package allows to perform noncommutative multiplications and simplify symbolic expressions by repeated substitution of prescribed relations.
All Mathematica-aided computations are available in \cite{DalMartello2023}.

In the following two sections, the notation drops the $q$ superscript for better readability, namely $T_i$ and $S$ stay for the respective quantum matrices.

\subsection{The matrix algebra for \texorpdfstring{$\mathsf{H}_{D_4}$}{}}\label{Rep2}

For $n=2$, the expected recovery of classical Teichm\"uller theory manifests by choosing the fat graph of the four-holed Riemann sphere $\Sigma_{0,4,0}$, see \Cref{app:analysis} for the analytic rationale underlying this choice.
Indeed, the matrix algebra resulting from the transport matrix factorization recovers the  same representation of the GDAHA $H_{D_4}(t,q)$ that was found in the classical Teichm\"uller framework \cite{Mazzocco2018}.
\begin{figure}[!htb]
    \centering
    \includegraphics[height=9cm]{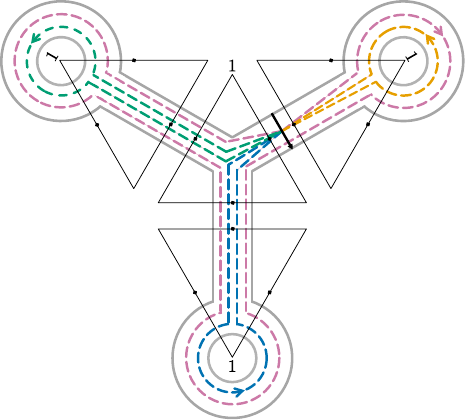}
    \caption{Fat graph $\Gamma_{0,4,0}$, its dual and relevant loops for $n=2$, with the transported edge displayed by the thick black arrow. The four triangles are labelled as follows: $(c)$ for the central one, $(r)$ for the rightmost one, $(l)$ for the leftmost one and $(d)$ for the downmost one. For each triangle, the $1$ indicates the choice of labelling and thus dictates its triple of transport matrices as in \Cref{fig:transport}.}\label{fig:2dimloops}
\end{figure}

The transport matrix \eqref{2dimT1}, computed in \Cref{ex-tr}, needs to be quantized and multiplied by the quantum correction $Q=\diag(q^{\nicefrac{1}{4}},q^{\nicefrac{-3}{4}})$:
\begin{equation*}
    T_1=Q\ \begin{pmatrix}
            -\ord{Z_{101}^{\nicefrac{1}{2}}Z_{110}^{\nicefrac{-1}{2}}} & -\ord{Z_{101}^{\nicefrac{1}{2}}Z_{110}^{\nicefrac{1}{2}}}\\[.5em]
            \ord{Z_{101}^{\nicefrac{-1}{2}}Z_{110}^{\nicefrac{-1}{2}}} & 0
        \end{pmatrix}=\begin{pmatrix}
        -Z_{101}^{\nicefrac{1}{2}}Z_{110}^{\nicefrac{-1}{2}} & -q^{\frac{1}{2}}Z_{101}^{\nicefrac{1}{2}}Z_{110}^{\nicefrac{1}{2}}\\[.5em]
        q^{-\frac{1}{2}}Z_{101}^{\nicefrac{-1}{2}}Z_{110}^{\nicefrac{-1}{2}} & 0
    \end{pmatrix}.
\end{equation*}
Quantum $T_2$ and $T_3$ follow the same recipe.

The transport matrix factorization can be read off from \Cref{fig:2dimloops}: denoting by {\color{3ochre}$O$} the matrix corresponding to the ochre loop, {\color{3blue}$B$} the matrix of the blue loop, {\color{3green}$G$} the one of the green loop and ${\color{3pink}P}$  that of the pink one, we have
\begin{equation}\label{2loopfactiorization}
\begin{aligned}
    {\color{3ochre}O}&=-q\ S \ T_3^{(r)} \ S \ T_2^{(r)} \ S,\\
    {\color{3blue}B}&=-q \ T_2^{(c)} \ S \ T_3^{(d)}\ S \ T_2^{(d)} \ S \ T_2^{(c)-1},\\
    {\color{3green}G}&=-q \ T_1^{(c)-1} \ S \ T_3^{(l)} \ S \ T_2^{(l)} \ S \ T_1^{(c)},\\
    {\color{3pink}P}&=\sqrt{q} \ T_1^{(c)-1} \ S \ T_2^{(l)-1} \ S \ T_3^{(l)-1} \ S  \ T_3^{(c)-1} \ S \ T_2^{(d)-1} \ S  \ T_3^{(d)-1} \ S  \ T_2^{(c)-1} \ S  \  T_2^{(r)-1} \ S \ T_3^{(r)-1} \ S,
\end{aligned}
\end{equation}
where $T_i^{(a)}$ stays for the quantum transport matrix $T_i$ in the Fock-Goncharov variables $Z_{\alpha}^{(a)}$ of the triangle $(a)$.
The $q$ and $\sqrt{q}$ factors, whose sign in neutral within $\PSL_2(\real)$, have been introduced in \eqref{2loopfactiorization} to set the product of each pair of Hecke parameters to the unit.
The base-point amalgamation is achieved conjugating formulae (\ref{2loopfactiorization}) by the diagonal matrix
\begin{equation}\label{eq:mat-C}
   C= \ H_1(\sqrt{q}Z^{(c)}_{110})=\begin{pmatrix}
        q^{\nicefrac{-1}{4}}Z^{(c)-1}_{110} & 0\\[.5em]
        0 & q^{\nicefrac{1}{4}}Z^{(c)}_{110}
    \end{pmatrix}. 
\end{equation}
For a matrix $M$ representing a path in a fat graph, we denote by $\overline{M}:=CMC^{-1}$ the one conjugated by the $C$ in \eqref{eq:mat-C}.

The final matrices $\overline O,\, \overline B, \overline G, \overline P$ depend only on the amalgamated variables
\begin{equation}
\begin{aligned}
    Z_{O1}&=qZ^{(r)}_{101}Z^{(r)}_{110},\qquad
    Z_{O2}=Z^{(c)}_{110}Z^{(r)}_{011},\\
    Z_{B1}&=qZ^{(d)}_{101}Z^{(d)}_{110},\qquad
    Z_{B2}=Z^{(c)}_{011}Z^{(d)}_{011},\\
    Z_{G1}&=qZ^{(l)}_{101}Z^{(l)}_{110},\qquad
    Z_{G2}=Z^{(c)}_{101}Z^{(l)}_{011},
\end{aligned}
\end{equation}
whose algebra relations are encoded by the triangular-shaped quiver in \Cref{fig:2amalquiver}.
\begin{figure}[!htb]
    \centering
    \includegraphics[height=5.5cm]{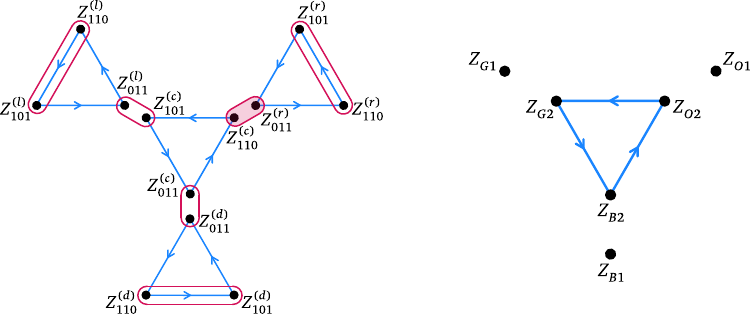}
    \caption{On the left, amalgamated pairs are highlighted in {\color{2red}red}, the shaded one triggered by the global conjugation. No variables from the original four triangles remain. On the right, the resulting quiver of amalgamated variables. The variables $Z_{O1},Z_{B1},Z_{G1}$ together with  $Z_{O2}Z_{B2}Z_{G2}$ generate the subalgebra of Casimir elements.}\label{fig:2amalquiver}
\end{figure}
\begin{remark}\label{rmk:2amalvars}
On the one hand, being isolated vertices in the amalgamated quiver, the variables $Z_{O1},Z_{B1},Z_{G1}$ are central elements. 
On the other hand, despite the transport matrices involve square roots of Fock-Goncharov variables, no fractional $Z_{O2},Z_{B2}$ or $Z_{G2}$ appear in the whole quadruple $(\overline O,\, \overline B, \overline G,\,\overline P)$.    
\end{remark}
As anticipated, the next theorem recovers the $\Mat_2(\mathbb{T}^2_q)$-embedding of the $\tilde{D}_4$-type GDAHA found by the second author in \cite{Mazzocco2018}.
\begin{theorem}\label{thm:TrueD4-emb}
Let $\pazocal{X}_{\pazocal{Q}_2}$ be the quantum $\pazocal{X}$-torus with coordinates $Z_{O1}$, $Z_{O2}$, $Z_{B1}$, $Z_{B2}$,$ Z_{G1}$, $Z_{G2}$ and $q$-commutations encoded by the quiver in \Cref{fig:2amalquiver}.
The $\SL_2(\pazocal{X}_{\pazocal{Q}_2}^{\nicefrac{1}{2}})$ matrices
\begin{equation}\label{OBGP}
\begin{aligned}
    \overline{O}&=\begin{pmatrix}
        0 & Z_{O1}^{\nicefrac{-1}{2}}Z_{O2}^{-1}\\[.5em]
        -Z_{O1}^{\nicefrac{1}{2}}Z_{O2} & Z_{O1}^{\nicefrac{1}{2}}+Z_{O1}^{\nicefrac{-1}{2}}
    \end{pmatrix},\\
    \overline{B}&=\begin{pmatrix}
        Z_{B1}^{\nicefrac{1}{2}}+Z_{B1}^{\nicefrac{-1}{2}}+Z_{B1}^{\nicefrac{-1}{2}}Z_{B2}^{-1} & Z_{B1}^{\nicefrac{1}{2}}+Z_{B1}^{\nicefrac{-1}{2}}+Z_{B1}^{\nicefrac{-1}{2}}Z_{B2}^{-1}+Z_{B1}^{\nicefrac{1}{2}}Z_{B2}\\[.5em]
        -Z_{B1}^{\nicefrac{-1}{2}}Z_{B2}^{-1} & -Z_{B1}^{\nicefrac{-1}{2}}Z_{B2}^{-1}
    \end{pmatrix},\\
    \overline{G}&=\begin{pmatrix}
        Z_{G1}^{\nicefrac{1}{2}}+Z_{G1}^{\nicefrac{-1}{2}}+Z_{
        G1}^{\nicefrac{1}{2}}Z_{G2} & Z_{
        G1}^{\nicefrac{1}{2}}Z_{G2}\\[.5em]
        -Z_{G1}^{\nicefrac{1}{2}}-Z_{G1}^{\nicefrac{-1}{2}}-Z_{
        G1}^{\nicefrac{-1}{2}}Z_{G2}^{-1}-Z_{
        G1}^{\nicefrac{1}{2}}Z_{G2} & -Z_{
        G1}^{\nicefrac{1}{2}}Z_{G2}
    \end{pmatrix},\\
    \overline{P}&=\begin{pmatrix}
        qZ_{O1}^{\nicefrac{1}{2}}Z_{B1}^{\nicefrac{1}{2}}Z_{G1}^{\nicefrac{1}{2}}Z_{O2}Z_{B2}Z_{G2} & 0\\
        -qz & qZ_{O1}^{\nicefrac{-1}{2}}Z_{B1}^{\nicefrac{-1}{2}}Z_{G1}^{\nicefrac{-1}{2}}Z_{O2}^{-1}Z_{B2}^{-1}Z_{G2}^{-1}
    \end{pmatrix},
\end{aligned}
\end{equation}
with
\begin{equation}\begin{aligned}
    z=&(Z_{O1}^{\nicefrac{1}{2}}-Z_{O1}^{\nicefrac{-1}{2}})Z_{B1}^{\nicefrac{-1}{2}}Z_{G1}^{\nicefrac{-1}{2}}Z_{B2}^{-1}Z_{G2}^{-1}+(Z_{B1}^{\nicefrac{1}{2}}-Z_{B1}^{\nicefrac{-1}{2}})Z_{G1}^{\nicefrac{-1}{2}}Z_{O1}^{\nicefrac{-1}{2}}Z_{O2}Z_{G2}^{-1}\\ &+(Z_{G1}^{\nicefrac{1}{2}}-Z_{G1}^{\nicefrac{-1}{2}})Z_{O1}^{\nicefrac{-1}{2}}Z_{B1}^{\nicefrac{-1}{2}}Z_{O2}Z_{B2}+Z_{O1}^{\nicefrac{1}{2}}Z_{B1}^{\nicefrac{1}{2}}Z_{G1}^{\nicefrac{1}{2}}Z_{O2}Z_{B2}Z_{G2}\\&+Z_{O1}^{\nicefrac{1}{2}}Z_{B1}^{\nicefrac{1}{2}}Z_{G1}^{\nicefrac{-1}{2}}Z_{O2}Z_{B2}Z_{G2}^{-1}+Z_{O1}^{\nicefrac{1}{2}}Z_{B1}^{\nicefrac{-1}{2}}Z_{G1}^{\nicefrac{-1}{2}}Z_{O2}Z_{B2}^{-1}Z_{G2}^{-1},
\end{aligned}
\end{equation}
satisfy the relations
\begin{equation}\label{H-rel-OBGP}
\begin{aligned}
    \left(\overline{O}-Z^{\nicefrac{1}{2}}_{O1}\One\right)\left(\overline{O}-Z^{\nicefrac{-1}{2}}_{O1}\One\right)&=\Zero,\\
    \left(\overline{B}-Z^{\nicefrac{1}{2}}_{B1}\One\right)\left(\overline{B}-Z^{\nicefrac{-1}{2}}_{B1}\One\right)&=\Zero,\\
    \left(\overline{G}-Z^{\nicefrac{1}{2}}_{G1}\One\right)\left(\overline{G}-Z^{\nicefrac{-1}{2}}_{G1}\One\right)&=\Zero,\\
    \left(\overline{P}-qZ_{O1}^{\nicefrac{1}{2}}Z_{B1}^{\nicefrac{1}{2}}Z_{G1}^{\nicefrac{1}{2}}Z_{O2}Z_{B2}Z_{G2}\One\right)\left(\overline{P}-qZ_{O1}^{\nicefrac{-1}{2}}Z_{B1}^{\nicefrac{-1}{2}}Z_{G1}^{\nicefrac{-1}{2}}Z_{O2}^{-1}Z_{B2}^{-1}Z_{G2}^{-1}\One\right)&=\Zero,\\
        \overline{O}\,\overline{B}\,\overline{G}\,\overline{P}&=q^{-1}\One.
\end{aligned}
\end{equation}
The map $K_1\to \overline{O}$,\, $K_2\to  \overline{B}$,\, $K_3\to  \overline{G},\, K_4\to \overline{P},\,q\to q^2$ defines an embedding of $\mathsf{H}_{D_4}$ into $\Mat_2(\pazocal{X}_{\pazocal{Q}_2}^{\nicefrac{1}{2}})$.
\end{theorem}
\begin{proof}
Specializing the variables in $\complex$, this result was proved in \cite{Mazzocco2018}, Theorem 3. For the formulae to match, we need to replace our $q$ by  $\sqrt{q}$ and perform the following substitutions:
\begin{equation}
    \overline{O}\mapsto M^{\hbar}_1,\quad \overline{B}\mapsto M^{\hbar}_2,\quad \overline{G}\mapsto M^{\hbar}_3,\quad \overline{P}\mapsto M^{\hbar}_\infty,
\end{equation}
as
\begin{equation}\label{ztoexp}
\begin{aligned}
    Z_{O1}&\mapsto e^{-p_1},\qquad
    Z_{O2}\mapsto e^{-s_1},\\
    Z_{B1}&\mapsto e^{-p_2},\qquad
    Z_{B2}\mapsto e^{-s_2},\\
    Z_{G1}&\mapsto e^{-p_3},\qquad
    Z_{G2}\mapsto e^{-s_3}.
\end{aligned}
\end{equation}
A direct computation in $\pazocal{X}_{\pazocal{Q}_2}^{\nicefrac{1}{2}}$ of relations \eqref{H-rel-OBGP}  can be found in the Mathematica companion \cite{DalMartello2023}.
The parameters in the Hecke relations are manifestly central: as previously noticed, the variables $Z_{O1}, Z_{B1}, Z_{G1}$ are isolated vertices while $Z_{O2}Z_{B2}Z_{G2}$ forms an isolated quiver cycle---as it involves just one-in one-out vertices.
Moreover, one easily checks that in $\pazocal{X}_{\pazocal{Q}_2}^{\nicefrac{1}{2}}$ the Hecke parameters of $\overline{P}$ multiply to the unit:
\begin{equation*}
     \big(qZ_{O1}^{\nicefrac{1}{2}}Z_{B1}^{\nicefrac{1}{2}}Z_{G1}^{\nicefrac{1}{2}}Z_{O2}Z_{B2}Z_{G2}\big)\big(qZ_{O1}^{\nicefrac{-1}{2}}Z_{B1}^{\nicefrac{-1}{2}}Z_{G1}^{\nicefrac{-1}{2}}Z_{O2}^{-1}Z_{B2}^{-1}Z_{G2}^{-1}\big)=1.
\end{equation*}
Finally, notice that \eqref{ztoexp} evaluates the central variables $Z_{O1},Z_{B1},Z_{G1}$ to the respective parameter $e^{p_i}$, setting all four matrices free from fractional coordinates (see \Cref{rmk:2amalvars}). By further evaluating the remaining central element $Z_{O2}Z_{B2}Z_{G2}$, one indeed reduces to the quantum 2-torus $\mathbb{T}^2_q:=\complex\langle x^{\pm1},y^{\pm1}\rangle\big/(xy-qyx)$, $q\in\complex^*$ not a root of unity.
\end{proof}

\subsection{The matrix algebra for \texorpdfstring{$\mathsf{H}_{E_6}$}{}}
For $n=3$, we take the fat graph shown in \Cref{fig:3dimpaths} (insights on this choice in \Cref{re:coll-h}).
The matrix algebra resulting from the transport matrix factorization indeed delivers a representation of the universal $\tilde{E}_6$-type GDAHA.
\begin{figure}[!htb]
    \centering
    \includegraphics[height=9.3cm]{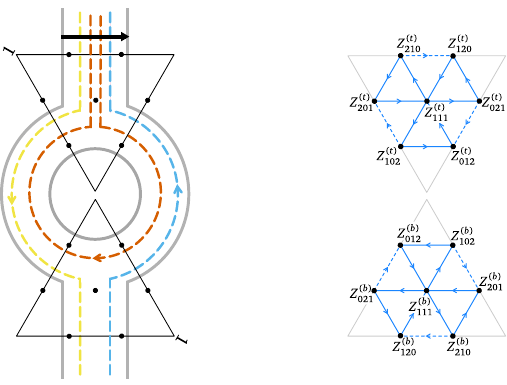}
    \caption{On the left: fat graph, triangles and relevant paths for $n=3$, with the transported edge displayed by the thick black arrow. The triangles are labelled as follows: $(t)$ for the top one and $(b)$ for the bottom one. $1$'s indicate the choice of labelling on the triangles. On the right: variables and corresponding quivers for each triangle \emph{before} any amalgamation is performed.}\label{fig:3dimpaths}
\end{figure}

The transport matrix \eqref{3dimT1}, computed in \Cref{ex-tr}, needs to be quantized and multiplied by the quantum correction $Q=\diag(q^{\nicefrac{5}{9}},q^{\nicefrac{-4}{9}},q^{\nicefrac{-13}{9}})$:
\begin{equation*}
\resizebox{\textwidth}{!}{$
T_1=
    \begin{pmatrix}
        q^{\frac{5}{9}}Z_ {111}^{\nicefrac{-1}{3}}Z_{102}^{\nicefrac{2}{3}}Z_{210}^{\nicefrac{-2}{3}}Z_{201}^{\nicefrac{1}{3}}Z_{120}^{\nicefrac{-1}{3}} &
  q^\frac{13}{18}\big(Z_{111}^{\nicefrac{-1}{3}}+q^{-1}Z_{111}^{\nicefrac{2}{3}}\big)Z_{102}^{\nicefrac{2}{3}}Z_{210}^{\nicefrac{2}{3}}Z_{120}^{\nicefrac{-1}{3}} & q^{\frac{2}{9}}Z_{111}^{\nicefrac{2}{3}}Z_{102}^{\nicefrac{2}{3}}Z_{210}^{\nicefrac{1}{3}}Z_{201}^{\nicefrac{1}{3}}Z_{120}^{\nicefrac{2}{3}} \\[0.5em]
    -q^{-\frac{11}{18}}Z_{111}^{\nicefrac{-1}{3}}Z_{102}^{\nicefrac{-1}{3}}Z_{210}^{\nicefrac{-2}{3}}Z_{201}^{\nicefrac{2}{3}} & -q^{-\frac{4}{9}}
   Z_{111}^{\nicefrac{-1}{3}}Z_{102}^{\nicefrac{-1}{3}}Z_{210}^{\nicefrac{1}{3}}Z_{201}^{\nicefrac{1}{3}}Z_{120}^{\nicefrac{-1}{3}} & 0 \\[0.5em]
  q^{-\frac{19}{9}}
   Z_{111}^{\nicefrac{-1}{3}}Z_{102}^{\nicefrac{-1}{3}}Z_ {210}^{\nicefrac{-2}{3}}Z_{201}^{\nicefrac{-2}{3}}Z_{120}^{\nicefrac{-1}{3}} & 0 & 0 \\
    \end{pmatrix}.$}
    \end{equation*}
Quantum $T_2$ and $T_3$ follow the same recipe.

The matrices corresponding to the paths can be read off from \Cref{fig:3dimpaths}. Denoting by  {\color{3yellow}Y} the matrix corresponding to the yellow path, {\color{3cyan}C} the one of the cyan path and  {\color{3red}R} that of the red one, we have
\begin{equation}\label{CYRfactorization}
\begin{aligned}
    {\color{3yellow}Y}&=q^{\frac{10}{9}}  \ S \ T_2^{(b)} \ S \ T_1^{(t)},\\
    {\color{3cyan}C}&=q^{\frac{10}{9}}  \ T_2^{(t)} \ S \ T_1^{(b)} \ S,\\
    {\color{3red}R}&=q^{\frac{10}{9}} \ T_1^{(t)-1} \ S  \  T_3^{(b)} \ S \ T_2^{(t)-1},
\end{aligned}
\end{equation}
with $q^{\nicefrac{10}{9}}$ factors introduced to set the product of each triple of Hecke parameters to the unit.
We then glue together the two open edges in \Cref{fig:3dimpaths}, closing all paths into loops.
Notice that the fat graph now corresponds to $\Sigma_{0,3,0}$, the three-holed Riemann sphere.
In order to amalgamate glued triangle sides properly, we perform a global conjugation by the following element in the Cartan subgroup (the so-called \textit{outer monodromy} in \cite{Goncharov2022}):
\begin{equation}\label{eq:Cartan}
    \diag(1,q^{\nicefrac{5}{6}},q^{\nicefrac{1}{3}})H_1(Z^{(t)}_{210}) H_2(Z^{(t)}_{120})=\begin{pmatrix}
        Z^{(t)\nicefrac{-2}{3}}_{210}Z^{(t)\nicefrac{-1}{3}}_{120} & 0 & 0\\[.5em]
        0 & q^{\nicefrac{5}{6}}Z^{(t)\nicefrac{1}{3}}_{210}Z^{(t)\nicefrac{-1}{3}}_{120} & 0\\[.5em]
        0 & 0 & q^{\nicefrac{1}{3}}Z^{(t)\nicefrac{1}{3}}_{210}Z^{(t)\nicefrac{2}{3}}_{120}
    \end{pmatrix},
    \end{equation}
the diagonal of $q$-factors chosen to simplify the resulting expressions.

The conjugated matrices, denoted again with the overbar, depend only on 
$Z^{(t)}_{111},\,Z^{(b)}_{111}$ and the following amalgamated variables
\begin{equation}
\begin{matrix}
    Z_{Y1}=Z^{(t)}_{201}Z^{(b)}_{021}, & Z_{Y2}=Z^{(t)}_{102}Z^{(b)}_{012}, & Z_{Y3}=Z^{(t)}_{210}Z^{(b)}_{120},\\[.5em]
    Z_{C1}=Z^{(b)}_{201}Z^{(t)}_{021}, & Z_{C2}=Z^{(b)}_{102}Z^{(t)}_{012}, & Z_{C3}=Z^{(b)}_{210}Z^{(t)}_{120}.
\end{matrix}
\end{equation}
The $q$-commutations are encoded by the diamond-shaped quiver in \Cref{fig:3amalquiver}.
\begin{figure}[!htb]
    \centering
    \includegraphics[width=\textwidth]{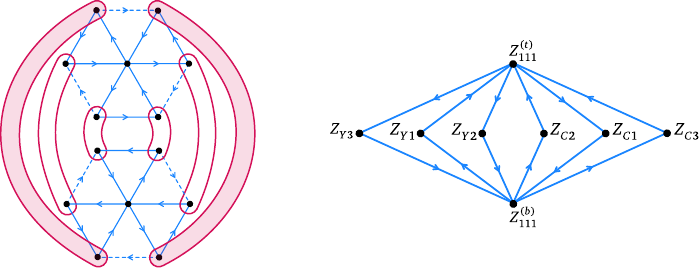}
    \caption{On the left, amalgamated pairs are highlighted in {\color{2red}red}, the shaded ones triggered by the global conjugation. Only the inner variables survive the amalgamations, one from each triangle. On the right, the resulting quiver of amalgamated variables: generators for the subalgebra of Casimir elements are given by $Z_{111}^{(t)}Z_{111}^{(b)}$ and all quiver cycles.}\label{fig:3amalquiver}
\end{figure}

We are finally ready to state the main \Cref{thm:E6-emb} in full detail:
\begin{theorem}\label{thm:TrueE6-emb}
Let $\pazocal{X}_{\pazocal{Q}_3}$ be the quantum $\pazocal{X}$-torus with coordinates
$$Z_{Y1},\,Z_{Y2},\,Z_{Y3},\,Z_{C1},\,Z_{C2},\,Z_{C3},\,Z^{(t)}_{111},\,Z^{(b)}_{111}$$
and $q$-commutations encoded by the quiver in \Cref{fig:3amalquiver}.
The $\SL_3(\pazocal{X}_{\pazocal{Q}_3}^{\nicefrac{1}{3}})$ matrices
\begin{equation}\label{RCY}
\resizebox{.95\textwidth}{!}{$
\begin{aligned}
    \overline{C}&=\begin{pmatrix}
             Z_{C1}^{\nicefrac{1}{3}}Z_{C2}^{\nicefrac{2}{3}}Z_{Y3}^{\nicefrac{2}{3}}Z_{C3}^{\nicefrac{1}{3}}Z_{111}^{(t)\nicefrac{2}{3}}Z_{111}^{(b)\nicefrac{2}{3}} & \overline{C}_{12} & \overline{C}_{13}\\[.5em]
            0 & Z_{C1}^{\nicefrac{1}{3}}Z_{C2}^{\nicefrac{-1}{3}}Z_{Y3}^{\nicefrac{-1}{3}}Z_{C3}^{\nicefrac{1}{3}}Z_{111}^{(t)\nicefrac{-1}{3}}Z_{111}^{(b)\nicefrac{-1}{3}} & \overline{C}_{23}\\[.5em]
     0 & 0 & Z_{C1}^{\nicefrac{-2}{3}}Z_{C2}^{\nicefrac{-1}{3}}Z_{Y3}^{\nicefrac{-1}{3}}Z_{C3}^{\nicefrac{-2}{3}}Z_{111}^{(t)\nicefrac{-1}{3}}Z_{111}^{(b)\nicefrac{-1}{3}} 
        \end{pmatrix},\\
    \overline{Y}&=\begin{pmatrix}
        Z_{Y1}^{\nicefrac{-2}{3}}Z_{Y2}^{\nicefrac{-1}{3}}Z_{Y3}^{\nicefrac{-2}{3}}Z_{C3}^{\nicefrac{-1}{3}}Z_{111}^{(t)\nicefrac{-1}{3}}Z_{111}^{(b)\nicefrac{-1}{3}} & 0 & 0\\[.5em]
        \overline{Y}_{21} & Z_{Y1}^{\nicefrac{1}{3}}Z_{Y2}^{\nicefrac{-1}{3}}Z_{Y3}^{\nicefrac{1}{3}}Z_{C3}^{\nicefrac{-1}{3}}Z_{111}^{(t)\nicefrac{-1}{3}}Z_{111}^{(b)\nicefrac{-1}{3}} & 0\\[.5em]
 \overline{Y}_{31} & \overline{Y}_{32} & Z_{Y1}^{\nicefrac{1}{3}}Z_{Y2}^{\nicefrac{2}{3}}Z_{Y3}^{\nicefrac{1}{3}}Z_{C3}^{\nicefrac{2}{3}}Z_{111}^{(t)\nicefrac{2}{3}}Z_{111}^{(b)\nicefrac{2}{3}}
    \end{pmatrix},
\end{aligned}$}
\end{equation}
with
\begin{equation*}\resizebox{\textwidth}{!}{$
\begin{aligned}
\overline{C}_{12}&=-q^\frac{1}{3}Z_{C1}^{\nicefrac{1}{3}}Z_{C2}^{\nicefrac{-1}{3}}Z_{Y3}^{\nicefrac{-1}{3}}Z_{C3}^{\nicefrac{1}{3}}Z_{111}^{(t)\nicefrac{-1}{3}}Z_{111}^{(b)\nicefrac{-1}{3}}-Z_{C1}^{\nicefrac{1}{3}}\big(Z_{C2}^{\nicefrac{-1}{3}}+q^{-1}Z_{C2}^{\nicefrac{2}{3}}\big)Z_{Y3}^{\nicefrac{-1}{3}}Z_{C3}^{\nicefrac{1}{3}}Z_{111}^{(t)\nicefrac{2}{3}}Z_{111}^{(b)\nicefrac{-1}{3}}\\
&\phantom{=}-Z_{C1}^{\nicefrac{1}{3}}Z_{C2}^{\nicefrac{2}{3}}Z_{Y3}^{\nicefrac{-1}{3}}Z_{C3}^{\nicefrac{1}{3}}Z_{111}^{(t)\nicefrac{2}{3}}Z_{111}^{(b)\nicefrac{2}{3}},\\[.5em]
 \overline{C}_{13}&=\big(Z_{C1}^{\nicefrac{1}{3}}+q^\frac{1}{3}Z_{C1}^{\nicefrac{-2}{3}}\big)Z_{C2}^{\nicefrac{-1}{3}}Z_{Y3}^{\nicefrac{-1}{3}}Z_{C3}^{\nicefrac{-2}{3}}Z_{111}^{(t)\nicefrac{-1}{3}}Z_{111}^{(b)\nicefrac{-1}{3}}+Z_{C1}^{\nicefrac{1}{3}}\big(q^\frac{5}{3}Z_{C2}^{\nicefrac{-1}{3}}+q^\frac{2}{3}Z_{C2}^{\nicefrac{2}{3}}\big)Z_{Y3}^{\nicefrac{-1}{3}}Z_{C3}^{\nicefrac{-2}{3}}Z_{111}^{(t)\nicefrac{2}{3}}Z_{111}^{(b)\nicefrac{-1}{3}},\\[.5em]
 \overline{C}_{23}&=-\big(q^{-\frac{1}{3}}Z_{C1}^{\nicefrac{1}{3}}+Z_{C1}^{\nicefrac{-2}{3}}\big)Z_{C2}^{\nicefrac{-1}{3}}Z_{Y3}^{\nicefrac{-1}{3}}Z_{C3}^{\nicefrac{-2}{3}}Z_{111}^{(t)\nicefrac{-1}{3}}Z_{111}^{(b)\nicefrac{-1}{3}},\\[1em]\overline{Y}_{21}&=\big(q^\frac{1}{3}Z_{Y1}^{\nicefrac{1}{3}}+
 Z_{Y1}^{\nicefrac{-2}{3}}\big)Z_{Y2}^{\nicefrac{-1}{3}}Z_{Y3}^{\nicefrac{1}{3}}Z_{C3}^{\nicefrac{-1}{3}}Z_{111}^{(t)\nicefrac{-1}{3}}Z_{111}^{(b)\nicefrac{-1}{3}},\\[.5em]
    \overline{Y}_{31}&=\big(Z_{Y1}^{\nicefrac{1}{3}} + q^\frac{1}{3}Z_{Y1}^{\nicefrac{-2}{3}}\big)Z_{Y2}^{\nicefrac{-1}{3}}Z_{Y3}^{\nicefrac{1}{3}}Z_{C3}^{\nicefrac{2}{3}}Z_{111}^{(t)\nicefrac{-1}{3}}Z_{111}^{(b)\nicefrac{-1}{3}} + 
 Z_{Y1}^{\nicefrac{1}{3}}\big(qZ_{Y2}^{\nicefrac{-1}{3}} + Z_{Y2}^{\nicefrac{2}{3}}\big)Z_{Y3}^{\nicefrac{1}{3}}Z_{C3}^{\nicefrac{2}{3}}Z_{111}^{(t)\nicefrac{-1}{3}}Z_{111}^{(b)\nicefrac{2}{3}},\\[.5em]
    \overline{Y}_{32}&=q^\frac{1}{3} 
 Z_{Y1}^{\nicefrac{1}{3}}Z_{Y2}^{\nicefrac{-1}{3}}Z_{Y3}^{\nicefrac{1}{3}}Z_{C3}^{\nicefrac{2}{3}}Z_{111}^{(t)\nicefrac{-1}{3}}Z_{111}^{(b)\nicefrac{-1}{3}}+Z_{Y1}^{\nicefrac{1}{3}}\big(q^\frac{4}{3}Z_{Y2}^{\nicefrac{-1}{3}}+q^\frac{1}{3}Z_{Y2}^{\nicefrac{2}{3}}\big)Z_{Y3}^{\nicefrac{1}{3}}Z_{C3}^{\nicefrac{2}{3}}Z_{111}^{(t)\nicefrac{-1}{3}}Z_{111}^{(b)\nicefrac{2}{3}}\\
 &\phantom{=}+Z_{Y1}^{\nicefrac{1}{3}}Z_{Y2}^{\nicefrac{2}{3}}Z_{Y3}^{\nicefrac{1}{3}}Z_{C3}^{\nicefrac{2}{3}}Z_{111}^{(t)\nicefrac{2}{3}}Z_{111}^{(b)\nicefrac{2}{3}},
\end{aligned}$}
\end{equation*}
satisfy, together with their quantum inverse product $\overline{R}$ displayed in \Cref{app:formulae}, the relations
\begin{equation}\label{H-rel-CYR}
\begin{aligned}
    \left(\overline{C}-Z_{C1}^{\nicefrac{-2}{3}}Z_{C2}^{\nicefrac{-1}{3}}Z_{Y3}^{\nicefrac{-1}{3}}Z_{C3}^{\nicefrac{-2}{3}}Z_{111}^{(t)\nicefrac{-1}{3}}Z_{111}^{(b)\nicefrac{-1}{3}}\One\right)&\left(\overline{C}-Z_{C1}^{\nicefrac{1}{3}}Z_{C2}^{\nicefrac{-1}{3}}Z_{Y3}^{\nicefrac{-1}{3}}Z_{C3}^{\nicefrac{1}{3}}Z_{111}^{(t)\nicefrac{-1}{3}}Z_{111}^{(b)\nicefrac{-1}{3}}\One\right)\\
    &\qquad \left(\overline{C}-Z_{C1}^{\nicefrac{1}{3}}Z_{C2}^{\nicefrac{2}{3}}Z_{Y3}^{\nicefrac{2}{3}}Z_{C3}^{\nicefrac{1}{3}}Z_{111}^{(t)\nicefrac{2}{3}}Z_{111}^{(b)\nicefrac{2}{3}}\One\right)=\Zero,\\
    \left(\overline{Y}-Z_{Y1}^{\nicefrac{-2}{3}}Z_{Y2}^{\nicefrac{-1}{3}}Z_{Y3}^{\nicefrac{-2}{3}}Z_{C3}^{\nicefrac{-1}{3}}Z_{111}^{(t)\nicefrac{-1}{3}}Z_{111}^{(b)\nicefrac{-1}{3}}\One\right)&\left(\overline{Y}-Z_{Y1}^{\nicefrac{1}{3}}Z_{Y2}^{\nicefrac{-1}{3}}Z_{Y3}^{\nicefrac{1}{3}}Z_{C3}^{\nicefrac{-1}{3}}Z_{111}^{(t)\nicefrac{-1}{3}}Z_{111}^{(b)\nicefrac{-1}{3}}\One\right) \\
    &\qquad \left(\overline{Y}-Z_{Y1}^{\nicefrac{1}{3}}Z_{Y2}^{\nicefrac{2}{3}}Z_{Y3}^{\nicefrac{1}{3}}Z_{C3}^{\nicefrac{2}{3}}Z_{111}^{(t)\nicefrac{2}{3}}Z_{111}^{(b)\nicefrac{2}{3}}\One\right)=\Zero,\\
    \left(\overline{R}-Z_{Y1}^{\nicefrac{-1}{3}}Z_{Y2}^{\nicefrac{-2}{3}}Z_{C1}^{\nicefrac{-1}{3}}Z_{C2}^{\nicefrac{-2}{3}}Z_{111}^{(t)\nicefrac{-1}{3}}Z_{111}^{(b)\nicefrac{-1}{3}}\One\right)&\left(\overline{R}-Z_{Y1}^{\nicefrac{-1}{3}}Z_{Y2}^{\nicefrac{1}{3}}Z_{C1}^{\nicefrac{-1}{3}}Z_{C2}^{\nicefrac{1}{3}}Z_{111}^{(t)\nicefrac{-1}{3}}Z_{111}^{(b)\nicefrac{-1}{3}}\One\right)\\
     &\qquad \left(\overline{R}-Z_{Y1}^{\nicefrac{2}{3}}Z_{Y2}^{\nicefrac{1}{3}}Z_{C1}^{\nicefrac{2}{3}}Z_{C2}^{\nicefrac{1}{3}}Z_{111}^{(t)\nicefrac{2}{3}}Z_{111}^{(b)\nicefrac{2}{3}}\One\right)=\Zero,\\
    &\phantom{----------------\,\,\,}\overline{C}\,\overline{Y}\,\overline{R}=q^{\nicefrac{2}{3}}\One.
\end{aligned}
\end{equation}
The map $J_1\to \overline{ C}$, $J_2\to  \overline{Y}$, $J_3\to  \overline{R},\, q \to q^{-2}$ defines an embedding of $\mathsf{H}_{E_6}$ into $\Mat_3(\pazocal{X}_{\pazocal{Q}_3}^{\nicefrac{1}{3}})$.
\end{theorem}
\begin{proof}
First observe that, unlike for the whole quadruple \eqref{OBGP}, all coordinates make a fractional appearance.
Proving that $\overline{Y}$, $\overline{C}$ and $\overline{R}$ satisfy relations \eqref{H-rel-CYR} is a  direct computation, which can be reproduced in the Mathematica companion \cite{DalMartello2023}. The Hecke parameters are central being products of pairs of variables having arrows with opposite directions. As an example, take 
\begin{align*}\overline{C}_{11}=Z_{C1}^{\nicefrac{1}{3}}Z_{C2}^{\nicefrac{2}{3}}Z_{Y3}^{\nicefrac{2}{3}}Z_{C3}^{\nicefrac{1}{3}}Z_{111}^{(t)\nicefrac{2}{3}}Z_{111}^{(b)\nicefrac{2}{3}}&=(Z_{C2}^{\nicefrac{2}{3}}Z_{Y3}^{\nicefrac{2}{3}})(Z_{C1}^{\nicefrac{1}{3}}Z_{C3}^{\nicefrac{1}{3}})(Z_{111}^{(t)\nicefrac{2}{3}}Z_{111}^{(b)\nicefrac{2}{3}})\\
&=(Z_{C2}Z_{Y3})^{\nicefrac{2}{3}}(Z_{C1}Z_{C3})^{\nicefrac{1}{3}}(Z_{111}^{(t)}Z_{111}^{(b)})^{\nicefrac{2}{3}}.
\end{align*}
For each bracketed pair, arrows cancel out: e.g., the $q$-factors due to arrows $Z_{C2}\rightarrow Z^{(b)}_{111}$ and $Z_{C2}\leftarrow Z^{(t)}_{111}$ are respectively absorbed by the ones due to $Z_{Y3}\leftarrow Z^{(b)}_{111}$ and $Z_{Y3}\rightarrow Z^{(t)}_{111}$.

The following inversion formulae, expressing central elements of $\pazocal{X}_{\pazocal{Q}_3}$ in terms of the tuple $t$, prove that this representation fully recovers the universal $\mathsf{H}_{E_6}$:
\begin{equation}
    \begin{matrix}
        Z_{C1}Z_{Y3}^{-1}= \frac{t_1^{(2)}}{t_2^{(1)}t_2^{(2)}t_3^{(1)}t_3^{(2)}}, & Z_{C2}Z_{Y3}= t_1^{(1)}t_2^{(2)}t_3^{(2)}, & Z_{Y1}Z_{Y3}= t_2^{(1)}t_2^{(2)},\\[.5em]
        Z_{Y2}Z_{Y3}^{-1}= \frac{1}{t_1^{(1)}t_2^{(2)}t_3^{(1)}}, & Z_{C3}Z_{Y3}= t_1^{(1)}t_1^{(2)}t_2^{(1)}t_2^{(2)}t_3^{(1)}t_3^{(2)}, & Z^{(t)}_{111}Z^{(b)}_{111}= \frac{1}{t_1^{(2)}t_2^{(2)}t_3^{(2)}}.
    \end{matrix}
\end{equation}

The fact that the map is an embedding can be proved by choosing a faithful representation of $\pazocal{X}_{\pazocal{Q}_3}^{\nicefrac{1}{3}}$, namely a vector space $V$ and an algebra homomorphism $\rho:\pazocal{X}_{\pazocal{Q}_3}^{\nicefrac{1}{3}}\to \mathrm{End}(V)$. The resulting map $\tilde\rho:\mathsf{H}_{E_6}\to \Mat_3(\mathrm{End}(V))$ gives a representation of $\mathsf{H}_{E_6}$ on $\bigoplus_3V$.
Now, the rank 1 GDAHA of type $\tilde{E}_6$ is prime. Indeed,  for generic values of parameters, it is Morita equivalent to its spherical subalgebra,
whose associated graded algebra is a twisted homogeneous coordinate ring of an irreducible curve, and therefore is a domain (Theorems 6.5, 6.10 in \cite{Etingof2007}). Furthermore, for $q\neq1$ it has no finite dimensional representations and is in fact simple\footnote{We thank P. Etingof for clarifying this argument to us.}.
This proves that $\tilde\rho$ is injective, thus so is our map\footnote{A direct proof of the representation's faithfulness, using a true polynomial PBW basis for $\mathsf{H}_{E_6}$,  will appear in a forthcoming paper by the first author.}.
\end{proof}

\section{Quiver seizure}\label{sec:cluster-seizure}
In this section, after applying the functor $\mathcal{F}_q$ to the matrix triple $(K_1,K_2,K_3)=(\overline{O},\overline{B},\overline{G})$ from \Cref{thm:TrueD4-emb}, we prove \Cref{thm:last}.

As prescribed in \Cref{sec:MC}, we start by rescaling:
\begin{equation}\label{rescOBG}
\begin{aligned}
    \widehat{K}_1=Z_{O1}^{\nicefrac{-1}{2}}\,\overline{O}&=\begin{pmatrix}
        0 & Z_{O1}^{-1}Z_{O2}^{-1}\\[.5em]
       -Z_{O2} & 1+Z_{O1}^{-1}
    \end{pmatrix},\\
    \widehat{K}_2=Z_{B1}^{\nicefrac{-1}{2}}\,\overline{B}&=\begin{pmatrix}
        1+Z_{B1}^{-1}+Z_{B1}^{-1}Z_{B2}^{-1} & 1+Z_{B1}^{-1}+Z_{B1}^{-1}Z_{B2}^{-1}+Z_{B2}\\[.5em]
        -Z_{B1}^{-1}Z_{B2}^{-1} & -Z_{B1}^{-1}Z_{B2}^{-1}
    \end{pmatrix},\\
    \widehat{K}_3=Z_{G1}^{\nicefrac{-1}{2}}\,\overline{G}&=\begin{pmatrix}
        1+Z_{G1}^{-1}+Z_{G2} & Z_{G2}\\[.5em]
        -1-Z_{G1}^{-1}-Z_{
        G1}^{-1}Z_{G2}^{-1}-Z_{G2} & -Z_{G2}
    \end{pmatrix}.
\end{aligned}
\end{equation}
Notice that all rescaled matrices are free from fractional coordinates, allowing to be more easily handled over $\pazocal{X}_{\pazocal{Q}_2}$.
Moreover, since our input to $\mathcal F_q$ is given by $2\times 2$ matrices, the operation $\mathcal C$ will produce $6\times 6$ matrices $N_1,N_2,N_3$.

In order to perform concretely the quotient in \Cref{prop:Misfunctor}, we need an explicit characterization of the eigenspaces $V_1^{(2)},V_2^{(2)},V_3^{(2)}$ we have to resctict to.
Selecting a representation of $\pazocal{X}_{\pazocal{Q}_2}$ on a vector space  $\pazocal{V}$, we fit the framework of genuine representations on vector spaces developed in \Cref{sec:functor}: indeed, this allows to view the matrices $\widehat K_i$ as elements in $\mathrm{End}(\pazocal{V}\oplus\pazocal{V})$, namely $V:=\pazocal{V}\oplus\pazocal{V}$.
Now, solving for eigenspaces is more conveniently carried out by reading formulae \eqref{rescOBG} as arrows in $\mathrm{Hom}_{\mathsf{Mod}\text{-}\pazocal{X}_{\pazocal{Q}_2}}(\oplus_2\pazocal{X}_{\pazocal{Q}_2} ,\oplus_2\pazocal{X}_{\pazocal{Q}_2})$, where $\mathsf{Mod}$-$\pazocal{X}_{\pazocal{Q}_2}$ denotes the category of right $\pazocal{X}_{\pazocal{Q}_2}$-modules, namely having the rescaled matrices act in the usual way on columns in $\Mat_{2\times1}(\pazocal{X}_{\pazocal{Q}_2})$. Indeed, computing eigenspaces in $\pazocal{V}\oplus\pazocal{V}$ amounts to solving $\pazocal{X}_{\pazocal{Q}_2}$-linear equations.
\begin{remark}
    $\pazocal{X}_{\pazocal{Q}_2}$, being essentially a quantum 3-torus, is known to be an Ore domain \cite{Berenstein2005} whose ring of fractions $\mathrm{Frac}(\pazocal{X}_{\pazocal{Q}_2})$ is a division algebra. Therefore, $\oplus_2\pazocal{X}_{\pazocal{Q}_2}$ is well-defined as the free rank $2$ $\pazocal{X}_{\pazocal{Q}_2}$-module.
\end{remark}
\begin{proposition}
As rank $1$ $\pazocal{X}_{\pazocal{Q}_2}$-submodules of $\oplus_2\pazocal{X}_{\pazocal{Q}_2}$, the eigenspaces $V^{(2)}_i=e_i(\pazocal{V}\oplus\pazocal{V})$ read
\begin{equation}
V^{(2)}_1=\big\langle(Z_{O2}^{-1},1)^T\big\rangle,\quad V^{(2)}_2=\big\langle(-1-Z_{B2},1)^T\big\rangle,\quad V^{(2)}_3=\big\langle(-1,1+Z_{G2}^{-1})^T\big\rangle.
\end{equation}
Then,
\begin{equation}\label{charE(V)}    E(\pazocal{V}\oplus\pazocal{V})=\big\langle(Z_{O2}^{-1},1,0,0,0,0)^T,\,(0,0,-1-Z_{B2},1,0,0)^T,\,(0,0,0,0,-1,1+Z_{G2}^{-1})^T\big\rangle
\end{equation}
\end{proposition}
\begin{proof}
    It is a straightforward computation: e.g., looking for $(a,b)^T\in\oplus_2\pazocal{X}_{\pazocal{Q}_2}$ such that
    \begin{equation*}
        \begin{pmatrix}
        1+Z_{B1}^{-1}+Z_{B1}^{-1}Z_{B2}^{-1} & 1+Z_{B1}^{-1}+Z_{B1}^{-1}Z_{B2}^{-1}+Z_{B2}\\[.5em]
        -Z_{B1}^{-1}Z_{B2}^{-1} & -Z_{B1}^{-1}Z_{B2}^{-1}
    \end{pmatrix}\begin{pmatrix}
        a \\
        b
    \end{pmatrix}=Z_{B1}^{-1}\begin{pmatrix}
        a \\
        b
    \end{pmatrix},
    \end{equation*}
    one immediately obtains $a=-b-Z_{B2}b$ from the second equation and this tautologically satisfies the first one. Notice that all three pairs of equations can be solved in $\pazocal{X}_{\pazocal{Q}_2}$, providing each a rank-$1$ submodule, due to the very special entries of the triple. In general, one must resort to  $\mathrm{Frac}(\pazocal{X}_{\pazocal{Q}_2})$ to invert generic entries.
\end{proof}
Reading $\mathcal{C}(\widehat{\mathbf{K}})$ over the generators \eqref{charE(V)}, we obtain a triple of pseudo-reflections $(R_1,R_2,R_3)$ encoded by the following $A$ matrix (see \Cref{se:q-killing}):
\begin{equation}
\resizebox{\textwidth}{!}{$
   {A}=
                \begin{pmatrix}
                    Z_{O1}^{-1} & (1-Z_{B1}^{-1})\frac{1+Z_{O1}Z_{O2}(1+Z_{B2})}{Z_{O1}-1} & (1-Z_{G1}^{-1})\frac{1+Z_{O1}Z_{O2}+Z_{G2}^{-1}}{Z_{O1}-1}\\[.5em] 
                    (Z_{O1}^{-1}-1)\frac{Z_{B1}+(1+q^2Z_{O2}^{-1})Z_{B2}^{-1}}{Z_{B1}-1} & Z_{B1}^{-1} &  
                    (Z_{G1}^{-1}-1)\frac{Z_{B1}+(Z_{B1}+Z_{B2}^{-1})Z_{G2}^{-1}}{Z_{B1}-1}\\[.5em] 
                    (Z_{O1}^{-1}-1)\frac{Z_{G1}Z_{G2}+Z_{O2}^{-1}(1+q^2Z_{G1}Z_{G2})}{Z_{G1}-1} & (1-Z_{B1}^{-1})\frac{1+Z_{B2}(1+q^2Z_{G1}Z_{G2})}{Z_{G1}-1} & Z_{G1}^{-1}
                \end{pmatrix}.$}
\end{equation}

We choose a more polished $(\overline{R}_1,\overline{R}_2,\overline{R}_3)$  by performing a global diagonal conjugation, which manifestly preserves the pseudo-reflection structure of the whole triple:
\begin{align*}
   \overline{R}_1 &:=CR_1C^{-1}= 
                \begin{pmatrix}
                    Z_{O1}^{-1} & -1-(1+Z_{O1}^{-1}Z_{O2}^{-1})Z_{B2}^{-1} & -q^{-1}Z_{B2}^{-1}-q^{-1}Z_{O1}^{-1}Z_{O2}^{-1}Z_{B2}^{-1}(1+q^2Z_{G2}^{-1})\\ 
                    0 & 1 & 0\\ 
                    0 & 0 & 1\\
                \end{pmatrix},\allowdisplaybreaks\\
   \overline{R}_2 &:=CR_2C^{-1}= 
                \begin{pmatrix}
                    1 & 0 & 0\\ 
                     Z_{B1}^{-1}+Z_{B1}^{-1}Z_{O2}+q^2Z_{O2}Z_{B2} & Z_{B1}^{-1} & -q^{-1}-q(1+Z_{B1}^{-1}Z_{B2}^{-1})Z_{G2}^{-1}\\ 
                    0 & 0 & 1\\
                \end{pmatrix},\allowdisplaybreaks\\
    \overline{R}_3 &:=CR_3C^{-1}= 
                \begin{pmatrix}
                    1 & 0 & 0\\ 
                    0 & 1 & 0\\ 
                   qZ_{G1}^{-1}Z_{B2}+q(Z_{O2}+1)Z_{B2}Z_{G2} & qZ_{G1}^{-1}(1+Z_{B2})+qZ_{B2}Z_{G2} & Z_{G1}^{-1}
                \end{pmatrix},
\end{align*}
for $C:=\diag(Z_{O2}^{-1}-Z_{O1}^{-1}Z_{O2}^{-1},-Z_{B2}+Z_{B1}^{-1}Z_{B2},-qZ_{B2}+qZ_{B1}^{-1}Z_{G2})\in\GL_3(\pazocal{X}_{\pazocal{Q}_2})$.
As detailed below in \Cref{prop-match}, this special choice selects in the conjugacy class of pseudo-reflections the one allowing for a direct match with the triple \eqref{RCY}.

The quantum Killing factorization of $\overline{R}_1 \overline{R}_2 \overline{R}_3 $ reads
\begin{equation}\label{eq:ULafterMC}
\resizebox{\textwidth}{!}{$
\begin{aligned}
    &U= 
                \begin{pmatrix}
                    1 & -1-(1+Z_{O1}^{-1}Z_{O2}^{-1})Z_{B2}^{-1} & q^{-1} + qZ_{G2}^{-1} + q(1 + Z_{B1}^{-1} + Z_{B1}^{-1}Z_{B2}^{-1} + Z_{O1}^{-1}Z_{B1}^{-1}Z_{O2}^{-1}Z_{B2}^{-1})Z_{B2}^{-1}Z_{G2}^{-1}\\ 
                    0 & 1 & -q^{-1}-qZ_{G2}^{-1}-qZ_{B1}^{-1}Z_{B2}^{-1}Z_{G2}^{-1}\\ 
                    0 & 0 & 1
                \end{pmatrix},\\
    &L= 
                \begin{pmatrix}
                    Z_{O1}^{-1} & 0 & 0\\
                    Z_{B1}^{-1}+Z_{B1}^{-1}Z_{O2}+q^2Z_{O2}Z_{B2} & Z_{B1}^{-1} & 0\\ q(Z_{G1}^{-1}Z_{B2}+Z_{B2}Z_{G2}+Z_{O2}Z_{B2}Z_{G2})    &  q(Z_{G1}^{-1}+Z_{G1}^{-1}Z_{B2}+Z_{B2}Z_{G2}) & Z_{G1}^{-1}
                \end{pmatrix},
\end{aligned}$}
\end{equation}
with inverse triple product $\Pi=L^{-1}U^{-1}$ obtained by formulae (\ref{invU}-\ref{invL}), see \cite{DalMartello2023}.
Performing the rescalings \eqref{rescalings} by passing to $\pazocal{X}_{\pazocal{Q}_2}^{\nicefrac{1}{3}}$, we obtain
\begin{equation}\label{Lmatrices}
    \widehat L:=Z_{O1}^{\nicefrac{1}{3}}Z_{B1}^{\nicefrac{1}{3}}Z_{G1}^{\nicefrac{1}{3}}L, \quad \widehat{\Pi}:=q^{-\frac{2}{3}}Z_{O1}^{\nicefrac{-1}{3}}Z_{B1}^{\nicefrac{-1}{3}}Z_{G1}^{\nicefrac{-1}{3}}\Pi.
\end{equation}
\begin{proposition}\label{th:E6D4FG}
The matrices $U,\widehat L,\widehat\Pi\in\SL_3(\pazocal{X}_{\pazocal{Q}_2}^{\nicefrac{1}{3}})$ satisfy the relations
\begin{equation}\label{H-rel-T}
    \begin{aligned}
    \left(U-\One\right)\left(U-\One\right)\left(U-\One\right)&=\Zero,\\
    \left(\widehat L-Z_{O1}^{\nicefrac{-2}{3}}Z_{B1}^{\nicefrac{1}{3}}Z_{G1}^{\nicefrac{1}{3}}\One\right)\left(\widehat L-Z_{O1}^{\nicefrac{1}{3}}Z_{B1}^{\nicefrac{-2}{3}}Z_{G1}^{\nicefrac{1}{3}}\One\right)\left(\widehat L-Z_{O1}^{\nicefrac{1}{3}}Z_{B1}^{\nicefrac{1}{3}}Z_{G1}^{\nicefrac{-2}{3}}\One\right)&=\Zero,\\
        \left(\widehat \Pi-q^{-\frac{2}{3}}Z_{O1}^{\nicefrac{-1}{3}}Z_{B1}^{\nicefrac{-1}{3}}Z_{G1}^{\nicefrac{-1}{3}}\One\right)\left(\widehat \Pi-q^{\frac{4}{3}}Z_{O1}^{\nicefrac{2}{3}}Z_{B1}^{\nicefrac{2}{3}}Z_{G1}^{\nicefrac{2}{3}}Z_{O2}Z_{B2}Z_{G2}\One\right)\phantom{---\,\,\,}\\\left(\widehat \Pi-q^{\frac{4}{3}}Z_{O1}^{\nicefrac{-1}{3}}Z_{B1}^{\nicefrac{-1}{3}}Z_{G1}^{\nicefrac{-1}{3}}Z_{O2}^{-1}Z_{B2}^{-1}Z_{G2}^{-1}\One\right)&=\Zero,\\
    U\,\widehat L\,\widehat\Pi &= q^{\nicefrac{-2}{3}}\One.
\end{aligned}
\end{equation}
The map $J_1\to U$, $J_2\to  \widehat{L}$, $J_3\to  \widehat{\Pi}$ gives a faithful representation of $H_{E_6}(\tilde{t},q^{2})$.
\end{proposition}
\begin{proof}
The whole statement follows as a corollary of \Cref{thm:truefunctor}.
All relations can be checked directly in the Mathematica companion \cite{DalMartello2023}.
\end{proof}
Before proving  \Cref{thm:last}, we discuss the quiver seizure operation defined in the introduction.
Given a $4$-cycle in $\pazocal{Q}$ with vertices labelled cyclically by variables $Z_1,Z_2,Z_3,Z_4\in\pazocal{X}_{\pazocal{Q}}$ such that the indegree and outdegree of both $Z_2$ and $Z_4$ equal one, namely $\mathrm{deg}^+(Z_i)=\mathrm{deg}^-(Z_i)=1$ for $i=2,4$, the monomial $ Z_2 Z_4$ is automatically central in the corresponding quantum $\pazocal{X}$-torus.
Then, for any $c^{(0)}\in\mathbb C^*$, the assignment
\begin{equation}\label{seizure-map}
\begin{cases}
    Z_4^{-1}\longmapsto \frac{1}{c^{(0)}} Z_2,\\
    Z_j \longmapsto Z_j,\quad j\neq 2,    
\end{cases}
\end{equation}
extends to a quantum torus isomorphism $\pazocal{X}_{\pazocal{Q}}\big\slash{\left(Z_2 Z_4 - c^{(0)}\right)} \xrightarrow{\raisebox{-.75ex}[0ex][0ex]{\,\,$\sim$\,\,}} \pazocal{X}_{\pazocal{Q}\setminus Z_4}$, where ${\pazocal{Q}\setminus Z_4}$ is the full subquiver of $\pazocal{Q}$ obtained by erasing $Z_4$ and the two arrows incident with it. Analogously, we get $\pazocal{Q}\setminus Z_2$ by resolving $Z_2$ instead.
Notice that a quiver seizure, as a special case of vertex erasure, commutes with quiver mutations.
\begin{remark}\label{rmk:seizuremonomials}
    The monomial $Z_2 Z_4$ always allows for a seizure, but does not cover all the ways a seizure can manifest: when the rhombus attaches to the rest of $\pazocal{Q}$ so that even $Z_1Z_3$ is central, the whole monomial $Z_1Z_2Z_3Z_4$ can be also chosen. 
\end{remark}
We are now ready to prove our final \Cref{thm:last}, namely show that the triple $(U,\widehat L, \widehat\Pi)$ can be found within $(\overline C,\overline Y,\overline R)$ from \Cref{thm:E6-emb}.
\begin{theorem}\label{prop-match}
    Let $\pazocal{X}_{\pazocal{Q}_1}:= \pazocal{X}_{\pazocal{Q}_3}\big/I$ be the quotient by the ideal $I=( Z_{C1}Z_{C3}-1,\,Z^{(t)}_{111}Z^{(b)}_{111}Z_{C2}Z_{Y3}-1)$ and denote by $(\overline C_{I},\overline Y_{\!I}, \overline R_{I})$ the restriction of the triple \eqref{RCY} to $\pazocal{X}_{\pazocal{Q}_1}$.
    Then, 
    \begin{equation}\label{match}
        (\overline C_{I},\overline Y_{\!I}, \overline R_{I})=\mu\,\iota\tau(U,\widehat{L},\widehat{\Pi})
    \end{equation}
    via the entry-wise action of the following maps: the algebra isomorphism   $$\tau:\pazocal{X}_{\pazocal{Q}_2}\xrightarrow{\raisebox{-.75ex}[0ex][0ex]{\,\,$\sim$\,\,}}\pazocal{X}_{\pazocal{Q}_2}
    $$
    reversing the $q$ parameter, i.e.,
\begin{equation}\label{tau}
\begin{aligned}
    \tau(Z_{O1})=Z_{O1},\quad \tau(Z_{B1})&=Z_{B1},\quad \tau(Z_{G1})=Z_{G1}, \\
    \tau(Z_{O2})=Z_{O2},\quad \tau(Z_{B2})&=Z_{B2},\quad \tau(Z_{G2})=Z_{G2}, \\
    \tau(q)&=q^{-1};
\end{aligned}
\end{equation}
the algebra isomorphism
$$\iota:\tau(\pazocal{X}_{\pazocal{Q}_2})\xrightarrow{\raisebox{-.75ex}[0ex][0ex]{\,\,$\sim$\,\,}} \pazocal{X}'_{\pazocal{Q}_1}
$$
given by
    \begin{equation}\label{iota}
        \begin{aligned}
            &Z^{-1}_{O2}\mapsto q^{\nicefrac{-1}{3}}Z'_{Y1}, &&Z^{-1}_{O1}\mapsto q^{\nicefrac{-2}{3}}Z'_{C2}Z'^{-1}_{Y1},\\
    &Z^{-1}_{B2}\mapsto q^{\nicefrac{-1}{3}}{Z'}^{(t)}_{111}, &&Z^{-1}_{B1}\mapsto q^{\nicefrac{-2}{3}}{Z'}^{(b)}_{111}{Z'}^{(t)-1}_{111},\\
    &Z^{-1}_{G2}\mapsto q^{\nicefrac{5}{3}}Z'_{C1}, &&Z^{-1}_{G1}\mapsto q^{\nicefrac{-2}{3}}Z'_{Y2}Z'^{-1}_{C1};
        \end{aligned}
    \end{equation}
and the quantum cluster mutation
    $$
    \mu:\mathrm{Frac}(\pazocal{X}'_{\pazocal{Q}_1})\xrightarrow{\raisebox{-.75ex}[0ex][0ex]{\,\,$\sim$\,\,}}\mathrm{Frac}(\pazocal{X}_{\pazocal{Q}_1})
    $$ 
    at vertex $Z_{111}^{(b)}$:
    \begin{equation}\label{mut}
    \begin{aligned}
    &\mu(Z'^{(b)}_{111}) = Z^{(b)-1}_{111}, && \mu(Z'^{(t)}_{111})=Z^{(t)}_{111}\\
    &\mu(Z'_{C1})=Z_{C1}\left(1+qZ^{(b)-1}_{111}\right)^{-1}\!\!\!, && \mu(Z'_{Y2})=Z_{Y2}\left(1+qZ^{(b)-1}_{111}\right)^{-1}\!\!\!,\\
    &\mu(Z'_{Y1})=Z_{Y1}\left(1+qZ^{(b)}_{111}\right), && \mu(Z'_{C2})=Z_{C2}\left(1+qZ^{(b)}_{111}\right).
    \end{aligned}
    \end{equation}    
\end{theorem}
\begin{proof} We start by noticing that $\pazocal{X}_{\pazocal{Q}_1}$ is obtained by a well-defined quantum quotient: both $Z_{C1}Z_{C3}$ and $Z^{(t)}_{111}Z^{(b)}_{111}Z_{C2}Z_{Y3}$ are central in $ \pazocal{X}_{\pazocal{Q}_3}$.
Both these monomials in $\pazocal{X}_{\pazocal{Q}_3}$ can be recognized as seizures for the quiver in \Cref{fig:3amalquiver}: at vertex $Z_{C3}$ for the rhombus $\{Z^{(t)}_{111},Z_{C1},Z^{(b)}_{111},Z_{C3}\}$ and at vertex $Z_{Y3}$ for the rhombus $\{Z^{(t)}_{111},Z_{Y3},Z^{(b)}_{111},Z_{C2}\}$ (see also \Cref{rmk:seizuremonomials}).
By the seizure's properties, the $q$-commutations for $\pazocal{X}_{\pazocal{Q}_1}$ are encoded by the reduced quiver in which we have erased the vertices $Z_{Y3}$ and $Z_{C3}$ together with their incident arrows.
This is indeed $\pazocal{Q}_1$, displayed in two equivalent shapes in \Cref{fig:rearrangedquiver}.
\begin{figure}[!htb]
    \centering
    \includegraphics[width=.7\textwidth]{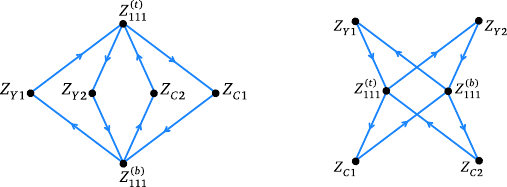}
    \caption{The reduced quiver in two equivalent shapes. On the left, the diamond obtained erasing the vertices $Z_{Y3}$ and $Z_{C3}$ directly in \Cref{fig:3amalquiver}. On the right, a rearranged star allowing for a better visualization of the mutation's action in \Cref{fig:mutquiver}.}\label{fig:rearrangedquiver}
\end{figure}

Therefore, the reduced triple $(\overline C_{I},\overline Y_{\!I}, \overline R_{I})$ is obtained via the identifications
\[
Z^{\nicefrac{1}{3}}_{C3}\mapsto Z_{C1}^{\nicefrac{-1}{3}},\quad Z_{Y3}^{\nicefrac{1}{3}}\mapsto Z_{C2}^{\nicefrac{-1}{3}}Z^{(b)\nicefrac{-1}{3}}_{111}Z^{(t)\nicefrac{-1}{3}}_{111}, \]
and its entries only involve the six variables $\{Z^{\nicefrac{1}{3}}_{Y1},Z^{\nicefrac{1}{3}}_{Y2},Z^{\nicefrac{1}{3}}_{C1},Z^{\nicefrac{1}{3}}_{C2},Z^{(t)\nicefrac{1}{3}}_{111}\!,Z^{(b)\nicefrac{1}{3}}_{111}\}$ generating $\pazocal{X}_{\pazocal{Q}_1}$.
It turns out that $\overline C_{I}$ is free from fractional powers and thus a genuine element in $\SL_3(\pazocal{X}_{\pazocal{Q}_1})$:
\begin{equation*}
    \overline C_{I}=\begin{pmatrix}
             1 & -1-q^{\nicefrac{-1}{3}}Z^{(t)}_{111}-q^{\nicefrac{-1}{3}}Z_{C2}Z^{(t)}_{111}(q^{-1}+Z^{(b)}_{111}) & q+q^{\nicefrac{1}{3}}Z_{C1}(q^{\nicefrac{1}{3}}+q^2Z^{(t)}_{111}+qZ_{C2}Z^{(t)}_{111})\\[.5em]
            0 & 1 & -q-q^{\nicefrac{2}{3}}Z_{C1}\\[.5em]
     0 & 0 & 1 
        \end{pmatrix}.
\end{equation*}
By the very definition of $I$, all its diagonal elements are turned into unities matching $\diag(U)$.

To push the match further, we need to take advantage of the cluster nature of the $\pazocal{X}$-space \cite{Fock2009}.
Indeed, to connect the quantum $\pazocal{X}$-torus $\pazocal{X}_{\pazocal{Q}_1}$ to the $\pazocal{X}_{\pazocal{Q}_2}$ one of the triple $(U,\widehat{L},\widehat{\Pi})$, we need to change chart via the mutation \eqref{mut}. 
In quiver terms, mutating at vertex $\alpha$ translates to a $3$-step recipe \cite{Fomin2021}:
\begin{enumerate}
    \item For each oriented two-arrow path $i\rightarrow\alpha\rightarrow j$, add a new arrow $i \rightarrow j$;
    \item Flip all arrows incident with $\alpha$;
    \item Remove all pairwise disjoint 2-cycles.
\end{enumerate}
Therefore, mutating at vertex $Z_{111}^{(b)}$, we turn the reduced quiver in \Cref{fig:rearrangedquiver} from star-shaped to box-shaped as in \Cref{fig:mutquiver}.
\begin{figure}[!htb]
    \centering
    \includegraphics[width=\textwidth]{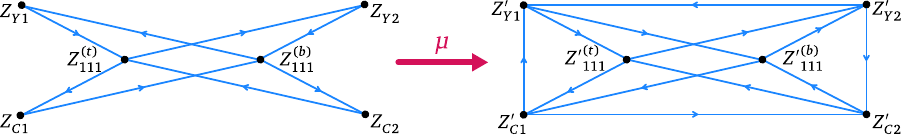}
    \caption{Reduced quiver, before and after quantum cluster mutation.}\label{fig:mutquiver}
\end{figure}
As expected, this mutated quiver encodes the $q$-commutations in $\pazocal{X}'_{\pazocal{Q}_1}$.
On the corresponding quantum cluster $\pazocal{X}$-tori, $\mu$ acts as a quantum analogue of a pullback sending $\pazocal{X}'_{\pazocal{Q}_1}$ to $\pazocal{X}_{\pazocal{Q}_1}$.

The gain in using $\mu$ is made manifest by the algebra isomorphism \eqref{iota}.
Indeed, $\iota$  reveals that the mutated quiver is equivalent to the triangular one in the right hand side of \Cref{fig:2amalquiver}, provided all arrows are reversed.
This is visually displayed in \Cref{fig:iotaquiver}.
\begin{figure}[!htb]
    \centering
    \includegraphics[width=.9\textwidth]{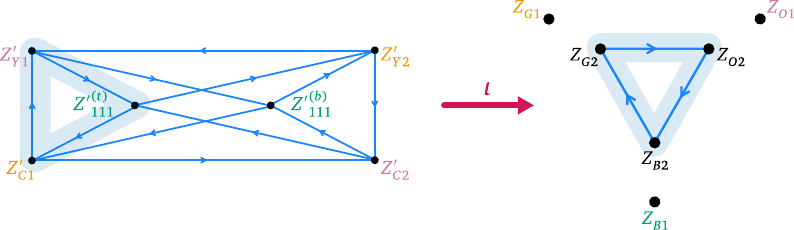}
    \caption{The quiver counterpart of the isomorphism $\iota$. Highlighted are the $3$-cycles identified by the map, while we color-coded each isolated vertex on the right with the corresponding pair of vertices on the left: e.g., {\color{3green}$Z_{B1} \propto {Z'}^{(t)}_{111}{Z'}^{(b)-1}_{111}$}.}\label{fig:iotaquiver}
\end{figure}
The quantum $\pazocal{X}$-torus counterpart of this arrow reversal is the $\tau$ map \eqref{tau}.

Now that quantum algebras agree, direct computations prove that the entry-wise action of the composition $\mu\iota\tau$ on $(U,\widehat{L},\widehat{\Pi})$ matches the reduced triple $(\overline C_{I},\overline Y_{\!I}, \overline R_{I})$.
Before we detail these computations, let us illustrate the phenomena allowing them to run successfully.
On the one hand, only $Z_{O1},Z_{B1},Z_{G1}$ make a fractional appearance in $(U,\widehat{L},\widehat{\Pi})$ and their image under $\mu\iota$ does not involve the formal inverse of $1+qZ^{(b)-1}_{111}$:
\begin{equation*}
    \mu\,\iota(Z^{-1}_{O1})=q^{\nicefrac{-2}{3}}Z_{C2}Z^{-1}_{Y1}, \quad \mu\,\iota(Z^{-1}_{B1})=q^{\nicefrac{-2}{3}}{Z}^{(b)}_{111}{Z}^{(t)-1}_{111}, \quad \mu\,\iota(Z^{-1}_{G1})=q^{\nicefrac{-2}{3}}Z_{Y2}Z^{-1}_{C1}.
\end{equation*}
Therefore, no fractional power of a formal inverse appears.
On the other hand, $\big(1+qZ^{(b)-1}_{111}\big)^{-1}$ does appear through $\mu\,\iota(Z_{G2}^{-1})$ but its algebra relations are easily figured out: indeed, for a formal power series $f(x)$,
\begin{equation}\label{formalpower}
Z_\beta Z_\alpha=q^wZ_\alpha Z_\beta \implies f(Z_\beta)Z_\alpha=Z_\alpha f(q^wZ_\beta).
\end{equation}
As a result, despite resorting to the fraction field for the mutation to act, the entry-wise action of  $\mu$ delivers genuine elements in $\SL_3(\pazocal{X}_{\pazocal{Q}_1}^{\nicefrac{1}{3}})$: using \eqref{formalpower}, each formal inverse simplifies.
As explained in \Cref{rmk:clustervarieties}, this is an instance of the quantum Laurent phenomenon \cite{Berenstein2005} within the $\pazocal{X}$-framework \cite{Fock2009}: $Z^{(b)}_{111}$ is mutable in the quantum cluster $\pazocal{X}$-variety over which the matrix entries are universally Laurent.

We conclude the proof detailing the computations behind the correspondence $U\mapsto\overline{C}_{I}$: with analogous operations, $\widehat{L}$ matches $\overline{Y}_{\!I}$ and $\widehat{\Pi}$ matches $\overline{R}_{I}$.
\begin{equation*}
\begin{aligned}
\mu\,\iota\tau U_{12}&=\mu\big(-1-q^{\nicefrac{-1}{3}}{Z'}^{(t)}_{111}+q^{\nicefrac{-4}{3}}Z'_{C2}{Z'}^{(t)}_{111}\big)=-1-q^{\nicefrac{-1}{3}}{Z}^{(t)}_{111}+q^{\nicefrac{-4}{3}}Z_{C2}(1+q{Z}^{(b)}_{111}){Z}^{(t)}_{111}\\
&=-1-q^{\nicefrac{-1}{3}}Z^{(t)}_{111}-q^{\nicefrac{-1}{3}}Z_{C2}Z^{(t)}_{111}(q^{-1}+Z^{(b)}_{111})\\
&=(\overline{C}_{I})_{12}\vspace{.5em}
\end{aligned}
\end{equation*}
\begin{equation*}
\begin{aligned}
\mu\,\iota\tau U_{13}&=\mu\,\iota\big(q + q^{-1}Z_{G2}^{-1} + q^{-1}(1 + Z_{B1}^{-1} + Z_{B1}^{-1}Z_{B2}^{-1} + Z_{O1}^{-1}Z_{B1}^{-1}Z_{O2}^{-1}Z_{B2}^{-1})Z_{B2}^{-1}Z_{G2}^{-1}\big)\\
&=\mu\big(q + q^{\nicefrac{2}{3}}Z'_{C1} + q^{\nicefrac{1}{3}}(1 + q^{\nicefrac{-2}{3}}{Z'}^{(b)}_{111}{Z'}^{(t)-1}_{111} + q^{-1}{Z'}^{(b)}_{111} + q^{-2}Z'_{C2}{Z'}^{(b)}_{111}){Z'}^{(t)}_{111}Z'_{C1}\big)\\
&=\mu\big(q + (q^{\nicefrac{2}{3}} + q^{\nicefrac{1}{3}}{Z'}^{(t)}_{111} + q^{\nicefrac{-1}{3}}{Z'}^{(b)}_{111} + q^{\nicefrac{-2}{3}}{Z'}^{(b)}_{111}{Z'}^{(t)}_{111} + q^{\nicefrac{-5}{3}}Z'_{C2}{Z'}^{(b)}_{111}{Z'}^{(t)}_{111})Z'_{C1}\big)\\
&=q + q^{\nicefrac{1}{3}}\big(q^{\nicefrac{1}{3}} + {Z}^{(t)}_{111} + q^{-1}Z_{C2}{Z}^{(t)}_{111}\big)(1+q^{-1}{Z}^{(b)-1}_{111})Z_{C1}(1+qZ^{(b)-1}_{111})^{-1}\\
&\overset{\eqref{formalpower}}{=\joinrel=}q + q^{\nicefrac{1}{3}}\big(q^{\nicefrac{1}{3}} + {Z}^{(t)}_{111} + q^{-1}Z_{C2}{Z}^{(t)}_{111}\big)Z_{C1}(1+q{Z}^{(b)-1}_{111})(1+qZ^{(b)-1}_{111})^{-1}\\
&=q+q^{\nicefrac{1}{3}}Z_{C1}(q^{\nicefrac{1}{3}}+q^2Z^{(t)}_{111}+qZ_{C2}Z^{(t)}_{111})\\
&=(\overline{C}_{I})_{13}\vspace{.5em}
\end{aligned}
\end{equation*}
\begin{equation*}
\begin{aligned}
\mu\,\iota\tau U_{23}&=\mu\,\iota(-q-q^{-1}Z_{G2}^{-1}-q^{-1}Z_{B1}^{-1}Z_{B2}^{-1}Z_{G2}^{-1})=\mu(-q-q^{\nicefrac{2}{3}}Z'_{C1}-q^{\nicefrac{-1}{3}}{Z'}_{111}^{(b)}Z'_{C1})\\
&=-q-q^{\nicefrac{2}{3}}Z_{C1}\left(1+qZ^{(b)-1}_{111}\right)^{-1}-q^{\nicefrac{-1}{3}}{Z}_{111}^{(b)-1}Z_{C1}\left(1+qZ^{(b)-1}_{111}\right)^{-1}\\
&=-q-q^{\nicefrac{2}{3}}Z_{C1}\left(1+qZ^{(b)-1}_{111}\right)^{-1}-q^{\nicefrac{5}{3}}Z_{C1}{Z}_{111}^{(b)-1}\left(1+qZ^{(b)-1}_{111}\right)^{-1}\\
&=-q-q^{\nicefrac{2}{3}}Z_{C1}
\end{aligned}\vspace{-.5em}
\end{equation*}
$\hspace{7em}=(\overline{C}_{I})_{23}$
\end{proof}

\begin{remark}
    We give further insight into the seizures making \Cref{prop-match} happen.    
    The one at $Z_{C3}$ has the natural central monomial of a rhombus, given by multiplying its two vertices not incident with the rest of the quiver---whose product is always central.
    The other monomial has two further factors instead, corresponding to the other two vertices of its rhombus: indeed, the product $Z^{(t)}_{111}Z^{(b)}_{111}$ is central in the quantum $\pazocal{X}$-torus encoded by the quiver in \Cref{fig:3amalquiver}.    
    The reason why exactly these two monomials appear must be found in the need of setting $\diag(\overline{C})=(1,1,1)$ for the match \eqref{match} to happen: the very way the transport matrix factorization forms this diagonal implies that the relations one must impose are those defining the ideal $I$, as one can check in \eqref{RCY} by collecting variables with the same power. \Cref{fig:seizure-mon} offers a visual interpretation of this phenomenon.
\end{remark}
\begin{figure}[!htb]
    \centering
    \includegraphics[width=.9\textwidth]{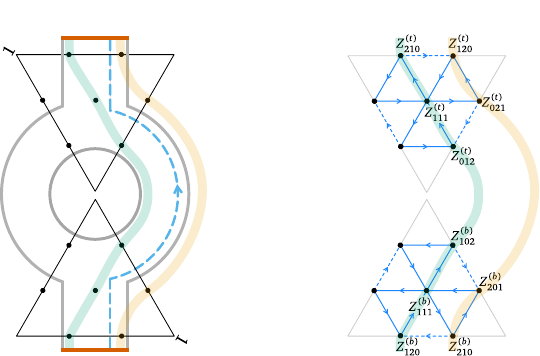}
    \caption{On the left, the $n=3$ fat graph with the {\color{3cyan}cyan} loop represented by $\overline{C}$ (the {\color{3red}red} segments are identified). On the right, the only Fock-Goncharov variables involved by this loop. The transparent closed ribbons highlight the way the central monomials of the two seizures in \Cref{prop-match} are formed as cycles, with the amalgamations bridging the gaps between the quivers of the two triangles. E.g., the monomial
    $Z_{C1}Z_{C3}=(Z^{(t)}_{021}Z^{(b)}_{201})(Z^{(b)}_{210}Z^{(t)}_{120})$ corresponds to the outer ribbon while the inner one, passing via the centers of the triangles, triggers the appearance of both $Z^{(t)}_{111}$ and $Z^{(b)}_{111}$ in the monomial  $Z^{(t)}_{111}Z^{(b)}_{111}Z_{C2}Z_{Y3}$.}\label{fig:seizure-mon}
\end{figure}
\begin{remark}\label{rmk:clustervarieties}
    Dealing with fat graph loops, each matrix we obtain in \Cref{sec:GDAHAreps} is a monodromy one whose entries are, unlike its trace, in general \emph{not} universally Laurent with respect to mutations at arbitrary coordinates. The global conjugation is a major player in this, forcing to freeze the variables it involves together with all those incident to them. It follows that the cluster structure underlying \Cref{thm:TrueD4-emb} has $Z_{O2}$, $Z_{B2}$, and $Z_{G2}$ frozen. Analogously, $Z^{(t)}_{111}$, $Z_{Y3}$, $Z_{C3}$, $Z_{Y1}$ and $Z_{C1}$ are frozen in \Cref{thm:TrueE6-emb}. In the latter case, notice that $Z^{(b)}_{111}$ from \eqref{mut} remains mutable, and that the remaining triple of frozen variables in $\pazocal{Q}_1$ is mapped by $\iota$ to the (frozen) triangle of $\pazocal{Q}_2$.
\end{remark}
\begin{remark}\label{re:coll-h}
    The process of reducing to $\pazocal{X}_{\pazocal{Q}_1}$ can be interpreted as colliding holes in the sense of \cite{Chekhov2017}: we are breaking an edge in the fat graph and treating the two open edges as bordered cusps on the boundary (see also \Cref{rmk:markedfatgrpah}). In order to perform this operation correctly, we need first to  
   perform a global conjugation by acting with the Cartan subgroup \cite{Goncharov2022} as we did in formula \eqref{eq:Cartan}.
\end{remark}

\appendix

\section{The quantum matrix \texorpdfstring{\,$\overline{R}$}{}}\label{app:formulae}

In this appendix, we give the explicit formulae for the entries of the matrix $\overline{R}$ in \Cref{thm:TrueE6-emb}.
As expected, they are equivalently obtained by defining $\overline{R}:=q^{\nicefrac{2}{3}}(\overline{C}\,\overline{Y})^{-1}$ as the quantum inverse or by conjugating its transport factorization \eqref{CYRfactorization}.

It follows from the factorization that each entry is a $\integer[q^{\nicefrac{\pm1}{3}}]$-linear combination of monomials in the coordinates of $\pazocal{X}_{\pazocal{Q}_3}$:
\begin{equation}
    \overline{R}=\begin{pmatrix}
        \,\overline{R}_{11} & \overline{R}_{12} & \overline{R}_{13}\,\\
        \,\overline{R}_{21} & \overline{R}_{22} & \overline{R}_{23}\,\\
        \,\overline{R}_{31} & \overline{R}_{32} & \overline{R}_{33}\,
    \end{pmatrix}
\end{equation}
with
\begin{equation*}
\resizebox{\textwidth}{!}{$
\begin{aligned}
    \overline{R}_{11}&=q^{\nicefrac{2}{3}}Z_{Y1}^{\nicefrac{2}{3}}Z_{Y2}^{\nicefrac{1}{3}} Z_{C1}^{\nicefrac{-1}{3}} Z_{C2}^{\nicefrac{-2}{3}} Z_{111}^{(t)\nicefrac{-1}{3}}Z_{111}^{(b)\nicefrac{-1}{3}},\allowdisplaybreaks\\[1.25em]
\overline{R}_{12}&={q}^{\nicefrac{1}{3}} Z_{Y1}^{\nicefrac{2}{3}}Z_{Y2}^{\nicefrac{1}{3}} Z_{C1}^{\nicefrac{-1}{3}}Z_{C2}^{\nicefrac{1}{3}}Z_{111}^{(t)\nicefrac{2}{3}}(
 Z_{111}^{(b)\nicefrac{2}{3}}+{q}^{\nicefrac{1}{3}}Z_{111}^{(b)\nicefrac{-1}{3}})+
 {q^{\nicefrac{1}{3}}} Z_{Y1}^{\nicefrac{2}{3}}Z_{Y2}^{\nicefrac{1}{3}} Z_{C1}^{\nicefrac{-1}{3}} Z_{C2}^{\nicefrac{-2}{3}} (q^{\nicefrac{1}{3}}Z_{111}^{(t)\nicefrac{-1}{3}} + Z_{111}^{(t)\nicefrac{2}{3}})Z_{111}^{(b)\nicefrac{-1}{3}},\allowdisplaybreaks\\[1.25em]
\overline{R}_{13}&={q}^{\nicefrac{1}{3}} Z_{Y1}^{\nicefrac{2}{3}}Z_{Y2}^{\nicefrac{1}{3}} Z_{C1}^{\nicefrac{-1}{3}}(Z_{C2}^{\nicefrac{1}{3}}+qZ_{C2}^{\nicefrac{-2}{3}})Z_{111}^{(t)\nicefrac{2}{3}}Z_{111}^{(b)\nicefrac{-1}{3}}+qZ_{Y1}^{\nicefrac{2}{3}}Z_{Y2}^{\nicefrac{1}{3}}(q^{\nicefrac{1}{3}}Z_{C1}^{\nicefrac{-1}{3}}Z_{C2}^{\nicefrac{1}{3}}+Z_{C1}^{\nicefrac{2}{3}}Z_{C2}^{\nicefrac{1}{3}})Z_{111}^{(t)\nicefrac{2}{3}}Z_{111}^{(b)\nicefrac{2}{3}},\allowdisplaybreaks\\[1.25em]
\overline{R}_{21}&= -q^{\nicefrac{2}{3}}(Z_{Y1}^{\nicefrac{2}{3}}+q^{\nicefrac{1}{3}}Z_{Y1}^{\nicefrac{-1}{3}})Z_{Y2}^{\nicefrac{1}{3}} Z_{C1}^{\nicefrac{-1}{3}} Z_{C2}^{\nicefrac{-2}{3}}Z_{111}^{(t)\nicefrac{-1}{3}}Z_{111}^{(b)\nicefrac{-1}{3}},\allowdisplaybreaks\\[1.25em]
\overline{R}_{22}&=\begin{aligned}[t]
  &-q^{\nicefrac{1}{3}}Z_{Y1}^{\nicefrac{2}{3}}Z_{Y2}^{\nicefrac{1}{3}} Z_{C1}^{\nicefrac{-1}{3}}Z_{C2}^{\nicefrac{1}{3}}Z_{111}^{(t)\nicefrac{2}{3}}(Z_{111}^{(b)\nicefrac{2}{3}}-q^{-1}Z_{111}^{(b)\nicefrac{-1}{3}})-
 q^{\nicefrac{1}{3}} Z_{Y1}^{\nicefrac{2}{3}}Z_{Y2}^{\nicefrac{1}{3}} Z_{C1}^{\nicefrac{-1}{3}} Z_{C2}^{\nicefrac{-2}{3}}(Z_{111}^{(t)\nicefrac{2}{3}}+q^{\nicefrac{1}{3}}Z_{111}^{(t)\nicefrac{-1}{3}}) Z_{111}^{(b)\nicefrac{-1}{3}}\\
 &-
 q^{\nicefrac{2}{3}}Z_{Y1}^{\nicefrac{-1}{3}}Z_{Y2}^{\nicefrac{1}{3}} (Z_{C1}^{\nicefrac{-1}{3}} Z_{C2}^{\nicefrac{-2}{3}}Z_{111}^{(t)\nicefrac{2}{3}}+q^{\nicefrac{1}{3}}Z_{C1}^{\nicefrac{-1}{3}} 
 Z_{C2}^{\nicefrac{-2}{3}}Z_{111}^{(t)\nicefrac{-1}{3}}+q^{-1}Z_{C1}^{\nicefrac{-1}{3}}Z_{C2}^{\nicefrac{1}{3}}Z_{111}^{(t)\nicefrac{2}{3}}) Z_{111}^{(b)\nicefrac{-1}{3}},\end{aligned}\allowdisplaybreaks\\[1.25em]
\overline{R}_{23}&=\begin{aligned}[t]&-q^{\nicefrac{1}{3}}(Z_{Y1}^{\nicefrac{2}{3}}+q^{\nicefrac{1}{3}}Z_{Y1}^{\nicefrac{-1}{3}})Z_{Y2}^{\nicefrac{1}{3}} Z_{C1}^{\nicefrac{-1}{3}}Z_{C2}^{\nicefrac{1}{3}}Z_{111}^{(t)\nicefrac{2}{3}} Z_{111}^{(b)\nicefrac{-1}{3}}-
qZ_{Y1}^{\nicefrac{2}{3}}Z_{Y2}^{\nicefrac{1}{3}} (q^{\nicefrac{1}{3}}Z_{C1}^{\nicefrac{-1}{3}}+Z_{C1}^{\nicefrac{2}{3}})Z_{C2}^{\nicefrac{1}{3}}Z_{111}^{(t)\nicefrac{2}{3}}Z_{111}^{(b)\nicefrac{2}{3}} \\
&-q^{\nicefrac{4}{3}}(Z_{Y1}^{\nicefrac{2}{3}}+q^{\nicefrac{1}{3}}Z_{Y1}^{\nicefrac{-1}{3}})Z_{Y2}^{\nicefrac{1}{3}} Z_{C1}^{\nicefrac{-1}{3}} Z_{C2}^{\nicefrac{-2}{3}}Z_{111}^{(t)\nicefrac{2}{3}}Z_{111}^{(b)\nicefrac{-1}{3}},\end{aligned}
\end{aligned}$}
\end{equation*}
\begin{equation}
\resizebox{\textwidth}{!}{$
\begin{aligned}
\overline{R}_{31}&={q^{\nicefrac{-5}{3}}}Z_{Y1}^{\nicefrac{-1}{3}} (Z_{Y2}^{\nicefrac{1}{3}}+qZ_{Y2}^{\nicefrac{-2}{3}})Z_{C1}^{\nicefrac{-1}{3}}  Z_{C2}^{\nicefrac{-2}{3}}  Z_{111}^{(t)\nicefrac{-4}{3}}Z_{111}^{(b)\nicefrac{-1}{3}}+
(q^{\nicefrac{-1}{3}}Z_{Y1}^{\nicefrac{2}{3}}+Z_{Y1}^{\nicefrac{-1}{3}})Z_{Y2}^{\nicefrac{1}{3}}  Z_{C1}^{\nicefrac{-1}{3}}  Z_{C2}^{\nicefrac{-2}{3}}   Z_{111}^{(t)\nicefrac{-1}{3}}Z_{111}^{(b)\nicefrac{-1}{3}},\\[1.25em]
\overline{R}_{32}&=\begin{aligned}[t]&q Z_{Y1}^{\nicefrac{-1}{3}} Z_{Y2}^{\nicefrac{1}{3}}  Z_{C1}^{\nicefrac{-1}{3}} Z_{C2}^{\nicefrac{1}{3}} (Z_{111}^{(t)\nicefrac{-1}{3}}+q^{\nicefrac{-7}{3}}Z_{111}^{(t)\nicefrac{2}{3}})Z_{111}^{(b)\nicefrac{-1}{3}}+(1+q^{-2}) Z_{Y1}^{\nicefrac{-1}{3}} Z_{Y2}^{\nicefrac{1}{3}}  Z_{C1}^{\nicefrac{-1}{3}}  Z_{C2}^{\nicefrac{-2}{3}}Z_{111}^{(t)\nicefrac{-1}{3}}Z_{111}^{(b)\nicefrac{-1}{3}}\\
&+{q^{\nicefrac{-5}{3}}}Z_{Y1}^{\nicefrac{2}{3}}  Z_{Y2}^{\nicefrac{1}{3}}  Z_{C1}^{\nicefrac{-1}{3}} Z_{C2}^{\nicefrac{1}{3}}Z_{111}^{(t)\nicefrac{2}{3}}(Z_{111}^{(b)\nicefrac{-1}{3}}+qZ_{111}^{(b)\nicefrac{2}{3}})+q^{-\nicefrac{5}{3}}Z_{Y1}^{\nicefrac{-1}{3}} (Z_{Y2}^{\nicefrac{1}{3}}+qZ_{Y2}^{\nicefrac{-2}{3}})  Z_{C1}^{\nicefrac{-1}{3}}  Z_{C2}^{\nicefrac{-2}{3}}  Z_{111}^{(t)-\nicefrac{4}{3}}Z_{111}^{(b)\nicefrac{-1}{3}}\\
&+q^{\nicefrac{-1}{3}} (Z_{Y1}^{\nicefrac{-1}{3}}+Z_{Y1}^{\nicefrac{2}{3}}) Z_{Y2}^{\nicefrac{1}{3}}  Z_{C1}^{\nicefrac{-1}{3}}  Z_{C2}^{\nicefrac{-2}{3}}  Z_{111}^{(t)\nicefrac{2}{3}}Z_{111}^{(b)\nicefrac{-1}{3}}+q^{\nicefrac{-1}{3}}(Z_{Y1}^{\nicefrac{2}{3}}Z_{Y2}^{\nicefrac{1}{3}}+q^{\nicefrac{-2}{3}}Z_{Y1}^{\nicefrac{-1}{3}}  Z_{Y2}^{\nicefrac{-2}{3}})Z_{C1}^{\nicefrac{-1}{3}}  Z_{C2}^{\nicefrac{-2}{3}}Z_{111}^{(t)\nicefrac{-1}{3}}Z_{111}^{(b)\nicefrac{-1}{3}},\end{aligned}\allowdisplaybreaks\\[1.25em]
\overline{R}_{33}&=\begin{aligned}[t]&Z_{Y1}^{\nicefrac{-1}{3}} Z_{Y2}^{\nicefrac{1}{3}}  Z_{C1}^{\nicefrac{-1}{3}}  (q^{\nicefrac{2}{3}}Z_{C2}^{\nicefrac{-2}{3}}+Z_{C2}^{\nicefrac{1}{3}})Z_{111}^{(t)\nicefrac{2}{3}}Z_{111}^{(b)\nicefrac{-1}{3}}+
{q^{\nicefrac{-2}{3}}}Z_{Y1}^{\nicefrac{2}{3}}  Z_{Y2}^{\nicefrac{1}{3}}  Z_{C1}^{\nicefrac{-1}{3}} (Z_{C2}^{\nicefrac{1}{3}}+qZ_{C2}^{\nicefrac{-2}{3}})Z_{111}^{(t)\nicefrac{2}{3}} Z_{111}^{(b)\nicefrac{-1}{3}}\\
&+Z_{Y1}^{\nicefrac{2}{3}} Z_{Y2}^{\nicefrac{1}{3}}  (q^{\nicefrac{1}{3}}Z_{C1}^{\nicefrac{-1}{3}}+Z_{C1}^{\nicefrac{2}{3}}) Z_{C2}^{\nicefrac{1}{3}} Z_{111}^{(t)\nicefrac{2}{3}} Z_{111}^{(b)\nicefrac{2}{3}}+Z_{Y1}^{\nicefrac{-1}{3}} Z_{Y2}^{\nicefrac{1}{3}}  Z_{C1}^{\nicefrac{-1}{3}}  (q^{-1} Z_{C2}^{\nicefrac{-2}{3}}+Z_{C2}^{\nicefrac{1}{3}})Z_{111}^{(t)\nicefrac{-1}{3}} Z_{111}^{(b)\nicefrac{-1}{3}}\\
&+
Z_{Y1}^{\nicefrac{-1}{3}}  Z_{Y2}^{\nicefrac{-2}{3}}   Z_{C1}^{\nicefrac{-1}{3}}  Z_{C2}^{\nicefrac{-2}{3}} Z_{111}^{(t)\nicefrac{-1}{3}} Z_{111}^{(b)\nicefrac{-1}{3}}.\end{aligned}
\end{aligned}$}
\end{equation}

\section{Analytical theory}\label{app:analysis}

In this appendix, we briefly discuss the analytic counterpart of the functor $\mathcal F_q$.

As proved in \cite{Mazzocco2018}, the classical limit of the representation of $H_{D_4}(t,q)$ on $\mathsf{Mat}_2(\mathbb{T}^2_q)$ contained in \Cref{thm:TrueD4-emb}
produces the generators of the monodromy group of a $2\times 2$ Fuchsian system with $4$ poles:
\begin{equation}\label{Fuch2}
    \frac{\mathrm{d}}{\mathrm{d}\lambda} \Phi=\left(
    \sum_{k=1}^{3}\frac{{A}_k}{\lambda-u_k} \right)\Phi,
    \end{equation}
    where each matrix ${A}_k\in\mathfrak{sl}_2(\mathbb C)$ has spectrum $\{\pm\frac{\theta_k}{2}\}$ for constants $\theta_k\not\in\mathbb Z$, and
    \begin{equation}
    -\left({A}_1+{A}_2+{A}_3\right)=:{A}_\infty=\frac{1}{2}\left(\begin{array}{cc}
    \theta_\infty & 0\\  0 & -    \theta_\infty\\ \end{array}\right)
    \end{equation}
    for a constant $\theta_\infty\not\in\mathbb Z$.
Fixing the fundamental solution at infinity and a basis in $\pi_1(\Sigma_{0,4,0})$ of non-intersecting loops circling each singularity once, the monodromy group of this system is given by 
\begin{equation}\label{eq:monodromy}
\left\{(M_1,M_2,M_3,M_\infty)  \in \SL_2(\mathbb C)^{\times4}\,\big|\, M_1M_2M_3M_\infty=\One, \,\hbox{eigen}(M_k)= e^{\pm\pi i \theta_k},\, k=1,2,3,\infty\right\}.
\end{equation}

The classical limits are obtained as $q\to 1$ by the correspondence
\begin{equation}
\overline O\to  M_1, \quad \overline B\to  M_2, \quad \overline G\to  M_3,\quad  \overline P \to  M_\infty,
\end{equation}
and the quantum cluster $\pazocal{X}$-variety is replaced by the classical Poisson cluster $\pazocal{X}$-variety.

In the following, we prove a lemma stating that the classical limit of the matrices \eqref{eq:ULafterMC} leads to the generalized monodromy data (namely Stokes and monodromy matrices) of a $3\times 3$ linear system with a simple pole at $0$ and a double pole at $\infty$:
\begin{equation}\label{irr}
    \frac{\mathrm{d}}{\mathrm{d}z}Y=\left(U+\frac{V-\mathbb I}{z}\right) Y,
    \end{equation}
    where $U=\diag(u_1,u_2,u_3)$ and $V$ is any constant matrix with diagonal part
    \begin{equation}
    \Theta=\diag(-\theta_1,-\theta_2,-\theta_3)
    \end{equation}
    and eigenvalues $\mu_1=\frac{1}{2}(\hbox{Tr}(\Theta)+ \theta_\infty)$, $ 
    \mu_2=\frac{1}{2}(\hbox{Tr}(\Theta)- \theta_\infty)$, $\mu_3=0$.
\begin{remark}\label{rmk:markedfatgrpah}
   According to the theory developed in \cite{Chekhov2017}, this $3\times3$ scenario corresponds to having a $\SL_3$-connection on $\Sigma_{0,2,2}$. This surface is precisely represented by the fat graph in \Cref{fig:3dimpaths} with two open edges, see also \Cref{re:coll-h}.
\end{remark}
Let us briefly describe the generalized monodromy data of the linear system \eqref{irr}: fixing both the fundamental matrix on a sector at infinity and an appropriate branch cut, the monodromy data are collected in the following subset of $\GL_3(\mathbb C)^{\times3}$:
\begin{equation}\label{eq:gen-mon}\begin{aligned}
\left\{(M_0,S_1,S_2) \in\GL_3(\mathbb C)\times B_+^{(1)}\times  B_-\,\big|\, M_0S_1S_2=\One,\right.\, &\mathrm{eigen}(M_0)=(1,e^{2\pi i \mu_1},e^{2\pi i \mu_2}),\\
&\mathrm{eigen}(S_2)\left.=(e^{i\pi\theta_1},e^{i\pi\theta_2},e^{i\pi\theta_3})\right\},
\end{aligned}
\end{equation}
where $B_+^{(1)}$ is the Borel subgroup of upper unitriangular matrices and $B_-$ is the Borel subgroup of lower triangular matrices.

The two linear systems \eqref{Fuch2} and \eqref{irr} are dual in the sense of Harnad \cite{Harnad1994}. This duality was first produced by Dubrovin in the case $\theta_i=0$ for $ i=1,2,3$. He showed that, given a loop $\gamma$ avoiding branch cuts in the $\lambda$-plane, the inverse Laplace transform 
\begin{equation}\label{eq:lapl}
Y(z)=\int_\gamma \Psi(\lambda) e^{\lambda z} d\lambda
\end{equation}
converges in a non-empty sector as $z\to\infty$ and
maps the system \eqref{irr} to a $3\times 3$ Fuchsian system of the form
\begin{equation}\label{Fuch1}
    \frac{\mathrm{d}}{\mathrm{d}\lambda} \Psi=\left(
    \sum_{k=1}^{3}\frac{{B}_k}{\lambda-u_k} \right)\Psi,\qquad B_k=-E_k V,
    \end{equation}
where $E_k$ is the matrix with zero entries everywhere except a $1$ in position $kk$.
In fact, the system \eqref{Fuch1} can be reduced to \eqref{Fuch2} by a simple quotient.
Subsequently, this result was generalized to any values of  $\theta_i$ by the second author of the present paper in \cite{Mazzocco2002}. Following this, Filipuk and Haraoka \cite{Filipuk2007} showed that the link between the $3\times 3$ Fuchsian system \eqref{Fuch1} and the $2\times 2$ one \eqref{Fuch2} could be recast in terms of the so-called additive middle convolution.
We summarise this chain of results in the following diagram:
\begin{equation}\label{sys-diagram}
    \begin{tikzcd}[column sep=huge]
\frac{\mathrm{d}}{\mathrm{d}\lambda}\Phi=\left(\sum_{k=1}^{3}\frac{{A}_k}{\lambda-u_k} \right)\Phi
  \arrow[r,"\mathrm{additive}","\mathrm{middle\phantom{.}conv.}"']
  \arrow[rr,bend left=15,dashed,"\mathrm{Harnad}","\mathrm{duality}"']
& \frac{\mathrm{d}}{\mathrm{d}\lambda} \Psi=\left(
    \sum_{k=1}^{3}\frac{{B}_k}{\lambda-u_k} \right)\Psi  \arrow[r,"\mathrm{inverse}","\mathrm{Laplace}"']
& \frac{\mathrm{d}}{\mathrm{d} z}Y=\left(U+\frac{V-\mathbb I}{z}\right) Y.
\end{tikzcd}
\end{equation}
It is important to point out that the Harnad duality maps the isomonodromic deformations equations of system \eqref{Fuch2} to the ones of system \eqref{irr}.

On the level of monodromy data, Dettweiler and Reiter showed that the monodromy matrices of two Fuchsian systems related by addittive middle convolution transform by the Katz middle convolution, also called multiplicative middle convolution \cite{Dettweiler2007}.
In our case, the monodromy matrices obtained via  multiplicative middle convolution are given by three pseudo-reflections $R_1,R_2,R_3$ together with the matrix product $R_4=(R_1 R_2 R_3)^{-1}$.  By constructing $S_1\in B_+^{(1)}$ and $S_2\in B_-$ as Killing factors of the triple $(R_1,R_2,R_3)$ and setting $M_0=R_4$, one obtains exactly an element in the set \eqref{eq:gen-mon} \cite{Boalch2005}.
 However, to prove that these are indeed the monodromy data  
of the system \eqref{irr} corresponding to the fundamental solution coming out of the  inverse Laplace transform, one needs to be careful with the loops of integration, the basis of loops in $\pi_1$ and the choice of sectors. This computation was performed by Dubrovin, following \cite{Balser1979}, in the case when $\theta_i=0$ for $ i=1,2,3$ and extended to all possible cases in \cite{Guzzetti2016}.
This discussion can be visualized by expanding \eqref{sys-diagram} to the following diagram, in which the commutativity of the left square is due to Dettweiler and Reiter while that of the right square is due to Dubrovin and Guzzetti:
\begin{equation}\label{diag-class}
    \begin{tikzcd}[row sep=large, column sep=huge]
\frac{\mathrm{d}}{\mathrm{d}\lambda}\Phi=\left(\sum_{k=1}^{3}\frac{{A}_k}{\lambda-u_k} \right)\Phi 
    \arrow[d,swap]  \arrow[r,"\mathrm{additive}","\mathrm{middle\phantom{.}conv.}"']
  \arrow[rr,bend left=15,dashed,"\mathrm{Harnad}","\mathrm{duality}"']
& \frac{\mathrm{d}}{\mathrm{d}\lambda} \Psi=\left(
    \sum_{k=1}^{3}\frac{{B}_k}{\lambda-u_k} \right)\Psi  \arrow[r,"\mathrm{inverse}","\mathrm{Laplace}"']
  \arrow[d,swap]
& \frac{\mathrm{d}}{\mathrm{d}z}Y=\left(U+\frac{V-\mathbb I}{z}\right) Y
  \arrow[d,swap]
\\
 \arrow[rr,bend right=15,2blue,dashed,"\mathrm{\mathcal F}_c"]
 (M_1,M_2,M_3,M_\infty) \arrow[r,"\mathrm{multiplicative}","\mathrm{middle\phantom{.}conv.}"'] & (R_1,R_2,R_3,R_4) \arrow[r,"\mathrm{Killing}","\mathrm{ factorization}"'] & (S_1,S_2,M_0)
\end{tikzcd}
\end{equation}

The {\color{2blue}blue} map, defined as the classical analogue $\mathcal F_c: (M_1,M_2,M_3,M_\infty)\to (S_1,S_2,M_0)$ of the functor $\mathcal F_q$, delivers the main result of this appendix:
\begin{lemma}
The classical limit $q\to 1$ of the matrices $U$ and $L$ \eqref{eq:ULafterMC} gives the generalized monodromy data of the linear system \eqref{irr}: recalling that $\Pi=(UL)^{-1}$,
\begin{equation}
    U\rightarrow S_1,\quad L\rightarrow S_2,\quad \Pi\rightarrow M_0.
\end{equation}
\end{lemma}
\begin{proof} This is a straightforward consequence of the fact that, by construction of $\mathcal F_q$ and $\mathcal F_c$, the following diagram commutes
  \begin{equation*}
    \begin{tikzcd}
(\overline O,\overline B,\overline G,\overline P )
\arrow{r}{\mathcal F_q} \arrow[swap]{d}{q\to 1} & (U, L, \Pi) \arrow{d}{q\to 1} \\%
(M_1,M_2,M_3,M_\infty) \arrow{r}{\mathcal F_c}&{(S_1,S_2,M_0)}
\end{tikzcd}\vspace{-2.5em}
    \end{equation*} 
\end{proof}
Following \cite{Laredo2022}, a possible quantization of the irregular system \eqref{irr} is given by the dynamical Knizhnik-Zamolodchikov (DKZ) equation
\begin{equation}
\mathrm{d}-\left(\hbar \frac{\Omega}{z}+ \mathrm{ad}(\mu\otimes 1)\right)\mathrm{d}z,
\end{equation}
where $\Omega\in\mathfrak g\otimes \mathfrak g$ is the Casimir element of $\mathfrak g=\mathfrak{sl}_3(\mathbb C)$, $\mu\in\mathfrak h$ belongs to the Cartan subalgebra $\mathfrak{h}$ of $\mathfrak g$, and $\hbar$ is a formal parameter.
The isomonodromic deformation equations of DKZ are given by the Casimir equation, which is dual to the standard Knizhnik-Zamolodchikov (KZ) equation  \cite{Laredo2002}. The latter is the quantization of the isomonodromic deformation equations of the $2\times 2$ linear system \eqref{Fuch2} \cite{Reshetikhin1992}.
 As observed by Harnad (introduction of  \cite{Tarasov2002}), this duality ``is essentially the ‘quantum’ version'' of the duality between the isomonodromic deformation equations of the linear systems \eqref{Fuch2} and \eqref{irr}.
 A complete proof of this statement, including the quantization of the linear system \eqref{Fuch2}, the quantum analogue of diagram \eqref{diag-class} as well as the 
 interplay between nonlinear monodromy and linear one, is still unknown. 
 Our \Cref{prop-match} provides a direct map on the level of quantum monodromy data. We postpone to future work the construction of a quantum analogue of the whole diagram \eqref{diag-class}.

\begin{remark}
    Another quantization of the Stokes matrices $S_1$ and $S_2$, developed in \cite{Xu2024}, was pointed out to us by the referee, whom we thank for the insightful comment.
    From the $R$-matrix formulation, that paper proves that a suitable combination of the entries of the resulting quantum Stokes matrices gives rise to a representation of the Drinfeld-Jimbo quantum group $U_q(\mathfrak{gl}_3)$.
    Given that the moduli spaces of pinnings provides a geometric realization of quantum groups \cite{Goncharov2022}, it would be interesting to compare the two quantizations. Transport matrices are already known to satisfy the so-called $RTT$ relations via the quantum trigonometric $R$-matrix \cite{Chekhov2022}, and we thus expect the Stokes ones to satisfy quantum reflection equations---as is the case when $V$ is skew-symmetric \cite{Chekhov2011}.
\end{remark}

{
\small
\bibliographystyle{abbrv}
\bibliography{GDAHA-ref}
}

\end{document}